\documentclass[11pt]{amsart}

\usepackage{amsmath,amssymb,amsxtra,epsf,amscd,graphics,color,array,ulem,etex, colortbl}
\usepackage{mathrsfs}
\usepackage[latin1]{inputenc}
\usepackage[T1]{fontenc}
\usepackage[frenchb, english]{babel}

\usepackage{amsfonts}
\usepackage{mathtools}

\usepackage{ytableau}
\usepackage{fancyhdr}
\usepackage{enumerate,array,calc}
\usepackage{epsf}
\usepackage{epsfig}
\usepackage{amsfonts,amssymb,wasysym}

\usepackage{hhline}
\usepackage{tikz}
\usepackage{hyperref}
\usepackage{caption}

\newcommand{\g}{\mathfrak{g}}
\newcommand{\ad}{\operatorname{ad}}

\newcommand{\codim}{\operatorname{codim}}
\def\ep{\varepsilon}

\def\p{\mathfrak p}
\def\s{\mathfrak s}
\def\h{\mathfrak h}
\def\l{\mathfrak l}
\def\a{\mathfrak a}
\def\m{\mathfrak m}
\def\b{\mathfrak b}
\def\n{\mathfrak n}

\def\q{\mathfrak q}
\def\ep{\varepsilon}

\newtheorem{theorem}{Theorem}[subsection]
\newtheorem{thm}[theorem]{Theorem}

\newtheorem{prop}[theorem]{Proposition}
\newtheorem{lm}[theorem]{Lemma}
\newtheorem{cor}[theorem]{Corollary}
\newtheorem{defi}[theorem]{Definition}
\newtheorem{Rqs}[theorem]{Remarks}
\newtheorem{Rq}[theorem]{Remark}

\newtheorem{nota}[theorem]{Notation}

\title[Symmetric semi-invariants for some Contractions-II-Case B even.]{Symmetric semi-invariants for some In\"on\"u-Wigner Contractions-II-Case B even.}

\author{Florence Fauquant-Millet}

\address{
Universit\'e Jean Monnet, 
ICJ UMR5208, CNRS, Ecole Centrale de Lyon, INSA Lyon, Universite Claude Bernard Lyon 1,
42023 Saint-Etienne, France.}
\email{florence.millet@univ-st-etienne.fr}

\begin{document}

\begin{abstract}

Let $\p$ be a proper parabolic subalgebra of a simple Lie algebra $\g$. Writing $\p=\mathfrak r\oplus\mathfrak m$ with $\mathfrak r$ being the standard Levi factor of $\p$ and $\m$  the nilpotent radical of $\p$, we consider the In\"on\"u-Wigner contraction $\widetilde\p$ of $\p$ with respect to this decomposition : this is the Lie algebra which is the semi-direct product $\mathfrak r\ltimes\m^a$, where $\m^a$ is an abelian ideal of $\widetilde\p$, isomorphic to $\m$ as an $\mathfrak r$-module.
 The study of the algebra of  symmetric semi-invariants $Sy\bigl(\widetilde\p\bigr)$  in the symmetric algebra $S\bigl(\widetilde\p\bigr)$ of $\widetilde\p$ under the adjoint action of $\widetilde\p$ was initiated in [Fauquant-Millet, F., Transformation Groups (2025) DOI : \url{https://doi.org/10.1007/s00031-024-09897-6}],
wherein
a lower bound for the formal character of  the algebra $Sy\bigl(\widetilde\p\bigr)$ was built, when the latter is well defined.
 
Here in this paper we build an upper bound for this formal character, when $\p$  is a maximal parabolic subalgebra in a classical simple Lie algebra $\g$ in type B, when  the Levi subalgebra of $\p$ is associated with the set of all simple roots without a simple root of even index with Bourbaki's notation (we call this case the even case). We show that both bounds coincide. This provides a Weierstrass section for $Sy\bigl(\widetilde\p\bigr)$ and the polynomiality of $Sy\bigl(\widetilde\p\bigr)$ follows. We also obtain that the derived subalgebra $\widetilde\p'$ of $\widetilde\p$ is nonsingular, by computation of the degrees of homogeneous generators of $Sy\bigl(\widetilde\p\bigr)$.

\end{abstract}

\maketitle



\section{Introduction.}\label{Intro}

The base field $\Bbbk$ is algebraically closed of characteristic zero. Consider a parabolic subalgebra $\p$ in a simple Lie algebra $\g$ and its In\"on\"u-Wigner contraction (or one-parameter contraction) $\widetilde\p=\mathfrak r\ltimes\m^a$ with respect to the decomposition $\p=\mathfrak r\oplus\m$ where $\mathfrak r$ is the standard Levi factor of $\p$ and $\m$ the nilpotent radical of $\p$. 
 As vector spaces and also as $\mathfrak r$-modules, the parabolic subalgebra $\p$ and its contraction $\widetilde\p$ are isomorphic. The latter may be viewed as a degeneration of the former (see for instance \cite[Sec. 4]{Y1} and \cite[Remark 2.3]{Fe}) and is still a Lie algebra, where $\m^a$ is an abelian ideal.
For any finite-dimensional Lie algebra $\a$, denote by $Y(\a)=S(\a)^{\a}$ the algebra of symmetric invariants in the symmetric algebra $S(\a)$ of $\a$ under adjoint action and by $Sy(\a)$ the algebra of symmetric semi-invariants in $S(\a)$ under adjoint action. Of course we have that $Y(\a)\subset Sy(\a)$. Denote also by $\a'=[\a,\,\a]$ the derived subalgebra of $\a$ (where $[\,,\,]$ is the Lie bracket in $\a$).
We continue the study of polynomiality of the algebra $Sy\bigl(\widetilde\p\bigr)$, that we have initiated in our paper \cite{F00}.
If $\p$ is a maximal parabolic subalgebra, then we have that $Sy\bigl(\widetilde\p\bigr)=Y\bigl(\widetilde\p'\bigr)$. More generally there exists a socalled canonical truncation $\widetilde\p_{\Lambda}$ of $\widetilde\p$ such that $Sy\bigl(\widetilde\p\bigr)=Y\bigl(\widetilde\p_{\Lambda}\bigr)=Sy\bigl(\widetilde\p_{\Lambda}\bigr)$ and $\widetilde\p_{\Lambda}$ is an ideal of $\widetilde\p$ which always contains $\widetilde\p'$. This inclusion may be strict, for instance when $\g$ is simple of type A  and $\p$ is a symmetric parabolic subalgebra of $\g$ (that is, the diagonal blocks forming the Levi factor of $\p$ are symmetric with respect to the antidiagonal). On the other hand, when $\p$ is a maximal parabolic subalgebra in any simple Lie algebra $\g$, then $\widetilde\p_{\Lambda}=\widetilde\p'$.
Inspired by \cite{J5}, \cite{FL} and \cite{FL1} we will construct in this paper an adapted pair (see definition below) for the canonical truncation subalgebra $\widetilde\p_{\Lambda}$ of the contraction $\widetilde\p$ of $\p$, 
when $\p$ is a maximal parabolic subalgebra of $\g$ simple of type ${\rm B}$ with a particular set of simple roots (set of simple roots of $\g$ without a simple root of even index in Bourbaki's notation \cite[Planche II]{BOU}). We call such maximal parabolic subalgebra an {\it even} maximal parabolic subalgebra.
Observe that for $\g$ simple of type ${\rm A}$, any maximal parabolic subalgebra $\p$ of $\g$ coincides with its contraction $\widetilde\p$ since the nilpotent radical $\m$ of $\p$  is already abelian in $\p$.

By  \cite[Lem. 6.11]{J6bis} the adapted pair for $\widetilde\p_{\Lambda}$ provides an upper bound for the formal character of $Sy\bigl(\widetilde\p\bigr)=Y\bigl(\widetilde\p_{\Lambda}\bigr)$. We will check that this bound coincides with the lower bound constructed in \cite{F00}. Applying again \cite[Lem. 6.11]{J6bis}, this implies that restriction of functions gives an isomorphism of algebras between $Y\bigl(\widetilde\p_{\Lambda}\bigr)$ and the algebra of polynomial functions on some affine space $\mathscr V$ of the dual space $\widetilde\p_{\Lambda}^*$. This means that $\mathscr V$ is a Weierstrass section for $Sy\bigl(\widetilde\p\bigr)$ in the sense of \cite{FJ4}.
Of course polynomiality of $Sy\bigl(\widetilde\p\bigr)=Y\bigl(\widetilde\p_{\Lambda}\bigr)$ follows. By \cite{FJ4}, $\mathscr V$ is also an affine slice in the sense of \cite[7.3]{J8}.
The study for $Sy(\p)$ will be called the nondegenerate case, whilst we will call the study for $Sy\bigl(\widetilde\p\bigr)$  the degenerate case.

The paper is organized as follows :

Sections 2 to 4 are generalities about In\"on\"u-Wigner contractions of parabolic subalgebras (not necessarily maximal)  in a simple Lie algebra and about adapted pairs and Weierstrass sections for such contractions. These sections will also be used in a future work when we will construct adapted pairs and Weierstrass sections for In\"on\"u-Wigner contractions of other parabolic subalgebras. In Section 3, we give some lemmas which allow to compute the Gelfand-Kirillov dimension of the algebra $Sy\bigl(\widetilde\p\bigr)$, thanks to the Gelfand-Kirillov dimension of the algebra $Sy(\p)$ in the nondegenerate case (which we know by \cite{FJ1}). Moreover in Section 4, we define as in \cite[Sec. 4]{FL1} socalled {\it stationary roots} and we give some interesting properties about them. These stationary roots are very useful to verify some nondegeneracy property which gives that our second element $y$ of the adapted pair is indeed regular for coadjoint action.

Finally Section 5 deals with the case of the In\"on\"u-Wigner contraction of an even maximal parabolic subalgebra of a simple Lie algebra in type B. In particular we show in Thm. \ref{nonsingthm} that in this case the algebra of symmetric semi-invariants $Sy\bigl(\widetilde\p\bigr)$ is a polynomial algebra and admits a Weierstrass section. Moreover we show that the derived subalgebra $\widetilde\p'$ of $\widetilde\p$ is nonsingular, namely that the set of regular elements is big in $\widetilde\p'^*$.

\section{Preliminaries.}
\subsection{Notation}\label{resprel}
 Let $\g$  be a simple Lie algebra over the base field $\Bbbk$. We denote by $[\,,\,]$ the Lie bracket in $\g$. Choose a Cartan subalgebra $\h$ of $\g$, which provides a (finite) root system $\Delta$ of $(\g,\,\h)$. Then choose a set $\pi=\{\alpha_1,\,\ldots,\,\alpha_n\}$ of simple roots, which  defines the set $\Delta^+=\sum_{i=1}^n \mathbb N\alpha_i\cap\Delta$, resp. $\Delta^-=\sum_{i=1}^n(-\mathbb N\alpha_i)\cap\Delta$ of positive, resp. negative, roots for $(\g,\,\h)$, so that $\Delta=\Delta^+\sqcup\Delta^-$.
 For a root $\alpha\in\Delta$, denote by $x_{\alpha}$ a nonzero root vector in $\g$ and by $\g_{\alpha}$ the subspace of $\g$ formed by all root vectors of weight $\alpha$. Recall that $\g_{\alpha}=\Bbbk\, x_{\alpha}=\{x\in\g\mid\forall h\in\h,\,[h,\,x]=\alpha(h)x\}$. We set $\n=\oplus_{\alpha\in\Delta^+}\g_{\alpha}$ and $\n^-=\oplus_{\alpha\in\Delta^-}\g_{\alpha}$. Then we have the triangular decomposition
 $$\g=\n\oplus\h\oplus\n^-.$$
 With every root $\alpha\in\Delta$ is associated a coroot $\alpha^{\vee}$ and we have that $$\h=\bigoplus_{\alpha\in\pi}\Bbbk\,\alpha^{\vee}.$$
 Moreover the elements of $\Delta$ span the dual vector space $\h^*$ of $\h$ (see for instance \cite[19.8.7]{TY}).
 For all $\alpha\in\pi$, we denote by $\varpi_{\alpha}\in\h^*$ the fundamental weight associated with $\alpha$ with respect to $(\g,\,\h,\,\pi)$ and we denote by $P^+(\pi)=\sum_{\alpha\in\pi}\mathbb N\varpi_{\alpha}$ the set of dominant weights.

Let $\pi'$ denote a subset of $\pi$. Set $\Delta_{\pi'}^{\pm}=\Delta^{\pm}\cap(\pm\mathbb N\pi')$ and $\Delta_{\pi'}=\Delta^+_{\pi'}\sqcup\Delta^-_{\pi'}$.  
We set also \begin{equation*}\n_{\pi'}=\oplus_{\alpha\in\Delta_{\pi'}^+}\g_{\alpha};\;\;\;\n^-_{\pi'}=\oplus_{\alpha\in\Delta_{\pi'}^-}\g_{\alpha};\;\;\;\m=\oplus_{\alpha\in\Delta^+\setminus\Delta_{\pi'}^+}\g_{\alpha}.
\end{equation*}

Finally set
\begin{equation}\mathfrak r=\n_{\pi'}\oplus\h\oplus\n^-_{\pi'}.\label{decompr}\end{equation}

By \cite[20.8.6,\, 20.8.8]{TY} the Lie subalgebra $\p=\mathfrak r\oplus\m$ is a  parabolic subalgebra of $(\g,\,\h,\,\pi)$, with $\mathfrak r$ being a Levi factor of $\p$  and $\m$ being the nilpotent radical and also the largest nilpotent ideal of $\p$ (as defined respectively in \cite[29.5.6]{TY},\,\cite[19.6.1]{TY} and  in \cite[19.5.5 and 19.5.8]{TY}).

Such a parabolic subalgebra $\p=\mathfrak r\oplus\m$ is called the standard parabolic subalgebra of $(\g,\,\h,\,\pi)$ associated with the subset $\pi'$ of the set $\pi$ of simple roots (see for instance \cite[1.2]{J5bis}). It contains the standard Borel subalgebra $\b=\h\oplus\n$ of $(\g,\,\h,\,\pi)$. The Levi factor $\mathfrak r$ of $\p$ given by equation (\ref{decompr}) will be called the standard Levi factor of $\p$.

We denote by $\p^-$ the parabolic subalgebra of $\g$ defined by 
$$\p^-=\mathfrak r\oplus\m^-$$
with $\m^-=\oplus_{\alpha\in\Delta^-\setminus\Delta_{\pi'}^-}\g_{\alpha}$ being the nilpotent radical (and also the largest nilpotent ideal) of $\p_-$.
The Lie algebra $\p^-$ is called the opposite parabolic subalgebra of $\p$. Denote by $K$ the Killing form of $\g$. Then the vector space $\p^-$ is isomorphic to the dual vector space $\p^*$ of $\p$ through the Killing form of $\g$.

Set $\h_{\pi'}=\bigoplus_{\alpha\in\pi'}\Bbbk\, \alpha^{\vee}$ and
$\h^{\pi\setminus\pi'}=\{h\in\h\mid \pi'(h)=0\}$, so that  we have that $\h=\h_{\pi'}\oplus\h^{\pi\setminus\pi'}$. By \cite[20.8.6]{TY} the radical ${\rm rad}\,\p$ of $\p$ is
$${\rm rad}\,\p=\h^{\pi\setminus\pi'}\oplus\m.$$

The derived subalgebra $\p'$ of $\p$ is such that 
$\p'=\mathfrak r'\oplus\m$ where $\mathfrak r'$ is the derived subalgebra of $\mathfrak r$. 
We have that \begin{equation}\mathfrak r'=\n_{\pi'}\oplus\h_{\pi'}\oplus\n^-_{\pi'}\end{equation} so that $\mathfrak r'$ is a semisimple Lie algebra such that $\Delta_{\pi'}$ is the root system of $(\mathfrak r',\,\h_{\pi'})$.
Moreover \begin{equation}\p={\rm rad}\,\p\oplus\mathfrak r'.\end{equation}
It follows that $\mathfrak r'$ is a Levi subalgebra of $\p$ by \cite[20.3.5]{TY}.

We denote by $W$, resp. $W'$, the Weyl group of $(\g,\,\h)$, resp. of $(\mathfrak r',\,\h_{\pi'})$ and by $w_0$, resp. $w_0'$, the longest element in $W$, resp. $W'$.

\subsection{In\"on\"u-Wigner contraction}
Recall the definition of an  In\"on\"u-Wigner contraction or one parameter contraction of a Lie algebra $\q$ as defined for instance in \cite[Sec. 3]{Y0} or \cite[Sec. 4]{Y1}.
Let $\q$ be a Lie algebra, $\mathfrak f$ a Lie subalgebra of $\q$ and $V$ a vector subspace of $\q$ such that $\q=\mathfrak f\oplus V$ (we do not require $V$ to be $\mathfrak f$-stable).
With the above decomposition of $\q$ may be associated the Lie algebra $\widetilde\q=\mathfrak f\ltimes V^a$ where $V$ is an abelian ideal of $\widetilde\q$ (that is why it is denoted with a superscript $a$) and where the action of $\mathfrak f$ on $V$ is given by the projection ${\rm pr}_V$ onto $V$. More precisely, if we denote by $[\,,\,]$ the Lie bracket in $\q=\mathfrak f\oplus V$, then the Lie bracket $[\,,\,]_{\widetilde\q}$ in $\widetilde\q=\mathfrak f\ltimes V^a$ is given by the following.
\begin{align}\forall \xi,\,\eta\in\mathfrak f,\,\forall v,\, w\in V,\,\,[\xi,\, \eta]_{\widetilde\q}=[\xi,\,\eta],\,[\xi,\,v]_{\widetilde\q}={\rm pr}_V([\xi,\,v]),\,[v,\,w]_{\widetilde\q}=0.\label{deficontract}\end{align}

It is easy to check that actually $(\widetilde\q,\,[\,,\,]_{\widetilde\q})$ is a Lie algebra. Moreover $\widetilde\q$ may be viewed as a degeneration of the Lie algebra $\q$, since it can be obtained by passing to the limit to zero a $t$-commutator $[\,,\,]_t$ defined on the vector space $\q$ (and which defines on $\q$ a new Lie structure, isomorphic to the Lie structure on $(\q,\,[\,,\,])$, when $t\neq 0$). See \cite[Sec. 3]{Y0} for more details.

In particular when $V$ is $\mathfrak f$-stable then there is no need to apply the projection onto $V$.

Thus we can define the In\"on\"u-Wigner contration $\widetilde\p$ of $\p=\mathfrak r\oplus\m$ by setting $\widetilde\p=\mathfrak r\ltimes\m^a$
which is isomorphic to $\p$ as a vector space and is endowed with
a Lie bracket $[\,,\,]_{\widetilde\p}$ defined by the following.
\begin{align}\forall z,\,z'\in\mathfrak r,\,\forall x,\,x'\in\m,\,\,[z,\,z']_{\widetilde\p}=[z,\,z'],\,[z,\,x]_{\widetilde\p}=[z,\,x],\,[x,\,x']_{\widetilde\p}=0\label{brackettildep}\end{align}
where recall $[\,,\,]$ is the Lie bracket in $\g$. 

In a same way  the In\"on\"u-Wigner contraction $\widetilde\p^-=\mathfrak r\ltimes(\m^-)^a$ of $\p^-=\mathfrak r\oplus\m^-$ is a Lie algebra  for the Lie bracket $[\,,\,]_{\widetilde\p^-}$ defined as follows.
\begin{align}\forall z,\,z'\in\mathfrak r,\,\forall y,\,y'\in\m^-,\,\,[z,\,z']_{\widetilde\p^-}=[z,\,z'],\,[z,\,y]_{\widetilde\p^-}=[z,\,y],\,[y,\,y']_{\widetilde\p^-}=0.\end{align}

\subsection{Symmetric semi-invariants}\label{symsi}
Let $\a$ be a finite-dimensional Lie algebra.  We denote by $\a'$ the derived subalgebra of $\a$, by $\mathfrak z(\a)$ the centre of $\a$ and by $S(\a)$ the symmetric algebra of $\a$. The algebra $S(\a)$ may also be identified with the algebra $\Bbbk[\a^*]$ of polynomial functions on the dual space $\a^*$ of $\a$. The adjoint action of $\a$ on itself given by Lie bracket may be uniquely extended by derivation on the associative and commutative algebra $S(\a)$, and we still call it the adjoint action of $\a$ on $S(\a)$ and denote it by $\ad$. 
The algebra of symmetric invariants $Y(\a)=S(\a)^{\a}$ in $S(\a)$ under adjoint action of $\a$ is defined as follows.
\begin{equation}Y(\a)=\{s\in S(\a)\mid \forall x\in \a,\,\ad x(s)=0\}.\end{equation}

The algebra of symmetric semi-invariants $Sy(\a)$ in $S(\a)$ under adjoint action of $\a$ is defined as follows.

\begin{equation}Sy(\a)=\oplus_{\lambda\in\a^*}S(\a)_{\lambda}\end{equation}
where, for all $\lambda\in\a^*$
\begin{equation}S(\a)_{\lambda}=\{s\in S(\a)\mid \forall x\in \a,\,\ad x(s)=\lambda(x)s\}.\end{equation}
The vector space $Sy(\a)$ is indeed an algebra since, for all $\lambda,\,\mu\in\a^*$, one has
$$S(\a)_{\lambda}S(\a)_{\mu}\subset S(\a)_{\lambda+\mu}.$$
When $S(\a)_{\lambda}\neq\{0\}$, we call $\lambda$ a weight of $Sy(\a)$ and $S(\a)_{\lambda}$ the vector space of semi-invariants of weight $\lambda$. 
We denote by $\Lambda(\a)$ the set of weights of $Sy(\a)$.
Of course we always have that 
\begin{equation}Y(\a)=S(\a)_0\subset Sy(\a)\end{equation}
and 
\begin{equation} Sy(\a)\subset S(\a)^{\a'}=\{s\in S(\a)\mid\forall x\in\a',\,\ad x(s)=0\}\label{inc}\end{equation}
by say \cite[Chap. I, Sec. B, 5.13]{F}.

\begin{Rq}\label{semiinvp}\rm
When $\a=\p$, resp.  $\a=\widetilde\p$,  we have equality in (\ref{inc}).
Indeed $\p=\p'\oplus\h^{\pi\setminus\pi'}$, resp. $\widetilde\p=\widetilde\p'\oplus\h^{\pi\setminus\pi'}$ with $\p'=\mathfrak r'\oplus\m$, resp. $\widetilde\p'=\mathfrak r'\ltimes\m^a$  and the commutative subalgebra $\h^{\pi\setminus\pi'}$ of $\p$, resp. of $\widetilde\p$,  acts (rationally) ad-reductively on $\p$, resp. on $\widetilde\p$. Therefore inclusion (\ref{inc}) becomes an equality in case $\a=\p$, resp. $\a=\widetilde\p$.
 Moreover the set $\Lambda(\p)$, resp. $\Lambda\bigl(\widetilde\p\bigr)$, may be viewed as a subset of $\h^*$ and even of $(\h^{\pi\setminus\pi'})^*$.
\end{Rq}

More generally for every $\h$-module $M$, the set $\Lambda(M)$ of weights of $M$ is defined to be
$$\Lambda(M)=\{\lambda\in\h^*\mid M_{\lambda}\neq\{0\}\}$$
where, for all $\lambda\in\h^*$, $M_{\lambda}=\{m\in M\mid\forall h\in\h,\,h.m=\lambda(h)m\}$ which is called the weight subspace of $M$ of weight $\lambda$.

Observe that the set $\Lambda(Sy(\p))$, resp. $\Lambda(Sy\bigl(\widetilde\p)\bigr)$, of weights of $Sy(\p)$, resp. of $Sy\bigl(\widetilde\p\bigr)$, is simply denoted by $\Lambda(\p)$, resp. $\Lambda(\widetilde\p)$.

\subsection{Canonical truncation}\label{truncation}
Assume from now on that $\a\subset\g\l(V)$ is algebraic  (with $V$ being a finite-dimensional vector space)  : see for instance \cite[Chap. VI, beginning]{F} for a definition. Then by \cite[Lem. 6.1]{BGR} there exists a canonically determined ideal 
$\a_{\Lambda}$ of $\a$ such that \begin{equation}Sy(\a)=Y(\a_{\Lambda})=Sy(\a_{\Lambda}).\label{semi}\end{equation}
The Lie algebra $\a_{\Lambda}$ is also algebraic and we call it the canonical truncation of $\a$.
More precisely $\a_{\Lambda}$ is the largest ideal of $\a$ on which every weight of $Sy(\a)$ vanishes. In other words \begin{equation}\a_{\Lambda}=\cap_{\lambda\in\Lambda(\a)}\ker(\lambda).\label{deftrunc}\end{equation}

We always have that (\cite[Chap. I, Sec. B, 7.2]{F})
\begin{equation}\a'\subset\a_{\Lambda}\label{derinc}\end{equation} but the inclusion may be strict, even for a parabolic subalgebra (see for instance \cite[Chap. V, Sec. B, 3.1]{F}). 

Note that the Lie algebra $\p$ is algebraic (see for instance \cite[Chap. I, Sec. B, 6.4]{F}) and its  canonical truncation $\p_{\Lambda}$  is given by the equality
\begin{equation*}\p_{\Lambda}=\p'\oplus\h_{\Lambda}\end{equation*} with \begin{equation}\h_{\Lambda}=\h^{\pi\setminus\pi'}\cap\p_{\Lambda}\label{truncp}\end{equation} 
(see \cite[Proof of Lemma 5.2.2 and Proof of Cor. 5.2.10]{FJ3}).

\begin{Rq}\label{Rqtruncation}\rm
As we already said, the vector space  $\h_{\Lambda}$ may be reduced to $\{0\}$ or not. For instance $\h_{\Lambda}=\{0\}$ whenever $\p$ is a maximal parabolic subalgebra of any simple Lie algebra $\g$ or whenever $w_0=-{\rm Id}$ that is, for any simple Lie algebra $\g$ outside type ${\rm A}_n$, ${\rm D}_{2n+1}$, or  ${\rm E}_6$ and for any parabolic subalgebra $\p$  of $\g$ (see for instance \cite[2.2]{FL}).
For an example where $\h_{\Lambda}\neq \{0\}$, consider the simple Lie algebra $\g$ of type ${\rm A}_3$ and $\p$ the symmetric parabolic subalgebra of $\g$ associated with $\pi'=\{\alpha_2\}$ (Bourbaki's notation, \cite[Planche I]{BOU}). Then one has that $\h_{\Lambda}=\Bbbk\mathscr H^{-1}(\varpi_{\alpha_1}-\varpi_{\alpha_3})$, where $\mathscr H:\h\longrightarrow\h^*$ is the isomorphism induced by the Killing form of $\g$.
\end{Rq}


\subsection{Centre and nilpotent radical of the contraction}\label{prop}
Similarly as for $\p$  we will see in Corollary \ref{propcor1} below that the contraction $\widetilde\p$ of  $\p$ is algebraic and then that equation (\ref{semi}) holds with $\a=\widetilde\p$ (see Corollary \ref{propcor2}).

\begin{lm}\label{proplm}
Let $\p=\mathfrak r\oplus\m$ be a standard parabolic subalgebra in a simple Lie algebra $\g$ with $\mathfrak r$ being the standard Levi factor of $\p$ and $\m$ the nilpotent radical of $\p$.  Let $\widetilde\p$ be the In\"on\"u-Wigner contraction with respect  to the above decomposition. We have that 
\begin{equation}\mathfrak z\bigl(\widetilde\p\bigr)=\{0\}\end{equation}
where $\mathfrak z\bigl(\widetilde\p\bigr)$ denotes the centre of $\widetilde\p$.

\end{lm}
\begin{proof}
The proof is similar as in \cite[20.8.6]{TY}. We give it for the reader's convenience. Set $\Delta(\pi')=\Delta^+\sqcup \Delta^-_{\pi'}\subset \Delta$. Observe that the set $\Delta(\pi')$ is a parabolic subset of $\Delta$ as defined in \cite[18.10.1]{TY}. Set $\g_{\Delta(\pi')}=\bigoplus_{\alpha\in\Delta(\pi')}\g_{\alpha}$ so that $\p=\h\oplus\g_{\Delta(\pi')}=\widetilde\p$  as a vector space. Let $x=h+x'\in\mathfrak z\bigl(\widetilde\p\bigr)$ with $h\in\h$ and $x'\in\g_{\Delta(\pi')}$. Since by equation (\ref{brackettildep}) $\bigl[x,\,\h\bigr]_{\widetilde\p}=[x,\,\h]=[x',\,\h]=0$ it follows that $x'=0$. Next we have that $\bigl[h,\,x_{\alpha}\bigr]_{\widetilde\p}=[h,\,x_{\alpha}]=\alpha(h)x_{\alpha}=0$ for all $\alpha\in\Delta(\pi')$. Since $\Delta=\Delta(\pi')\cup(-\Delta(\pi'))$, it follows that $h=0$.
\end{proof}

\begin{prop}\label{propprop}

We keep the  hypotheses of the above Lemma. 

Then the radical of $\widetilde\p$ is $\h^{\pi\setminus\pi'}\oplus\m$
and  $\m$ is the largest nilpotent ideal of $\widetilde\p$ and is also its nilpotent radical.

\end{prop}

\begin{proof}

Let us show that $\m\oplus\h^{\pi\setminus\pi'}$ is the  radical ${\rm rad}\,\widetilde\p$ of $\widetilde\p$. We check easily that $\m\oplus\h^{\pi\setminus\pi'}$ is an ideal of $\widetilde\p$ and that the Lie quotient algebra $\widetilde\p/(\m\oplus\h^{\pi\setminus\pi'})$ is isomorphic to the semisimple Lie algebra $\mathfrak r'$. Then  by \cite[1.4.3]{D} we have that ${\rm rad}\,\widetilde\p\subset \m\oplus\h^{\pi\setminus\pi'}$. Moreover since $\m$ is abelian in $\widetilde\p$ and since $[\h^{\pi\setminus\pi'},\,\m]_{\widetilde\p}\subset\m$, it follows by definition \cite[1.3.7 (i)]{D} that $\m\oplus\h^{\pi\setminus\pi'}$ is a solvable ideal of $\widetilde\p$. This implies equality ${\rm rad}\,\widetilde\p= \m\oplus\h^{\pi\setminus\pi'}$ by definition of the radical \cite[1.4.1]{D}. Now $\m$ is  the nilpotent radical of $\widetilde\p$, because by \cite[1.7.1, 1.7.2]{D} the latter is equal to $\bigl[\widetilde\p,\,\widetilde\p\bigr]_{\widetilde\p}\cap{\rm rad}\,\widetilde\p=\widetilde\p'\cap{\rm rad}\,\widetilde\p$ 
and because $\widetilde\p'=\mathfrak r'\ltimes\m^a$.

Moreover since $\m$ is an abelian ideal of $\widetilde\p$ we have that, for all $x\in\m$, $\ad_{\m} x=0$ and then $(\ad_{\widetilde\p}x)^2=0$. It follows by \cite[1.4.7]{D} that $\m$ is contained in the largest nilpotent ideal of $\widetilde\p$ : denote it by $\n_{\widetilde\p}$. 
Since $\bigl[\h^{\pi\setminus\pi'},\,\n_{\widetilde\p}\bigr]_{\widetilde\p}\subset\n_{\widetilde\p}$ and since $\n_{\widetilde\p}$ is nilpotent, it follows that $\n_{\widetilde\p}+\h^{\pi\setminus\pi'}$ is a solvable ideal of $\widetilde\p$.
Then $\n_{\widetilde\p}+\h^{\pi\setminus\pi'}\subset {\rm rad}\,\widetilde\p= \m\oplus\h^{\pi\setminus\pi'}$.

Let $x\in\n_{\widetilde\p}$ and write $x=x'+z$ with $x'\in\m$ and $z\in\h^{\pi\setminus\pi'}$. Since $\m\subset\n_{\widetilde\p}$ by the above, we have that $z\in\n_{\widetilde\p}\cap\h$.   

By definition of $\n_{\widetilde\p}$, we have that ${\rm ad}_{\widetilde\p}\, z$ is nilpotent. The element $z\in\h$ acting also $\ad_{\widetilde\p}$-reductively, this implies that $z$ belongs to the centre of $\widetilde\p$ and then that $z=0$ by the above Lemma. Thus $\m=\n_{\widetilde\p}$.
\end{proof}
\begin{cor}\label{propcor1}
With the same hypotheses and notation as before, the Lie algebra $\widetilde\p$ is algebraic in $\g\mathfrak l\bigl(\widetilde\p\bigr)$.
 
\end{cor}
\begin{proof}

Since the centre of $\widetilde\p$ is reduced to $\{0\}$ by Lemma \ref{proplm}, the Lie algebra $\widetilde\p$ may be viewed as a subalgebra of $\g\mathfrak l\bigl(\widetilde\p\bigr)$. Recall that $\widetilde\p=\mathfrak r\oplus\m=\mathfrak r'\oplus\h^{\pi\setminus\pi'}\oplus\m$ as a vector space. By the above Proposition, $\m$ is the set of $\ad_{\widetilde\p}$-nilpotent elements in ${\rm rad}\,\widetilde\p=\m\oplus\h^{\pi\setminus\pi'}$. 

Moreover by definition in \cite[2.4.5]{J2}, the Lie subalgebra $\mathfrak r$ is rationally reductive in $\g\mathfrak l\bigl(\widetilde\p\bigr)$, namely it is the product of a semisimple Lie subalgebra $\mathfrak r'$ and of an abelian subalgebra $\h^{\pi\setminus\pi'}$ of $\ad_{\widetilde\p}$-semisimple elements, so that $\h^{\pi\setminus\pi'}$ admits a basis whose members have rational eigenvalues.
The assertion follows by \cite[Prop. 2.4.5]{J2} since $\m$ is complemented in $\widetilde\p$ by the rationally reductive Lie subalgebra $\mathfrak r$ in $\g\mathfrak l\bigl(\widetilde\p\bigr)$. \end{proof}
\begin{cor}\label{propcor2}
With the same hypotheses and notation as before, we have  that \begin{equation}Sy\bigl(\widetilde\p\bigr)=Y\bigl(\widetilde\p_{\Lambda}\bigr)=Sy\bigl(\widetilde\p_{\Lambda}\bigr)\label{egSy}\end{equation}
where $\widetilde\p_{\Lambda}$ is the canonical truncation of $\widetilde\p$.
\end{cor}
\begin{proof}
It is a consequence of the above Corollary and of equation (\ref{semi}). 
\end{proof}

\subsection{The contraction of the canonical truncation}\label{trunc}

Denote by $\widetilde{\p_{\Lambda}}$ the In\"on\"u-Wigner contraction of the canonical truncation $\p_{\Lambda}$ of $\p$ with respect to the decomposition
\begin{equation}\p_{\Lambda}=\mathfrak r_{\rm trunc}\oplus\m\label{decomptruncp}\end{equation} where $\mathfrak r_{\rm trunc}=\mathfrak r'\oplus\h_{\Lambda}$ with $\h_{\Lambda}$ the vector subspace of $\h^{\pi\setminus\pi'}$ given by (\ref{truncp}). We may pay attention that $\mathfrak r_{\rm trunc}$ is not  the canonical truncation $\mathfrak r_{\Lambda}$ of the reductive Lie algebra $\mathfrak r$. Indeed $\mathfrak r_{\Lambda}=\mathfrak r$ since $\Lambda(\mathfrak r)=\{0\}$ : actually one has that $Sy(\mathfrak r)=Y(\mathfrak r')S(\h^{\pi\setminus\pi'})=Y(\mathfrak r)$.
 By equation (\ref{deficontract}) one has that
\begin{equation}\widetilde{\p_{\Lambda}}=\mathfrak r_{\rm trunc}\ltimes\m^a\label{contplamb}\end{equation}
and $\widetilde{\p_{\Lambda}}$ is a Lie subalgebra of $\widetilde\p$ by equation (\ref{brackettildep}).

Similarly as for $\widetilde\p$, we have the following Lemma.

\begin{lm}\label{truncalgebraic}
With the above notation, the In\"on\"u-Wigner contraction $\widetilde{\p_{\Lambda}}$ is algebraic in $\g\mathfrak l\bigl(\widetilde\p\bigr)$.

\end{lm}
\begin{proof}
It is easily checked that  $\widetilde{\p_{\Lambda}}$ is an ideal of $\widetilde\p$, which contains $\widetilde\p'$. Moreover as in the proof of Prop. \ref{propprop} it is easily checked that the radical ${\rm rad}\,\widetilde{\p_{\Lambda}}$ of $\widetilde{\p_{\Lambda}}$ is ${\rm rad}\,\widetilde{\p_{\Lambda}}=\m\oplus\h_{\Lambda}$ and that $\m$ is the  largest nilpotent ideal of $\widetilde{\p_{\Lambda}}$. Thus $\m$ is the set of $\ad_{\widetilde{\p_{\Lambda}}}$-nilpotent elements and then the set of $\ad_{\widetilde\p}$-nilpotent elements in ${\rm rad}\,\widetilde{\p_{\Lambda}}$. Moreover $\m$ is complemented in $\widetilde{\p_{\Lambda}}$ by the rationally reductive Lie subalgebra $\mathfrak r_{\rm trunc}$ in $\g\mathfrak l\bigl(\widetilde\p\bigr)$ (as defined in the proof of  Cor. \ref{propcor1}).
 Then \cite[Prop. 2.4.5]{J2} implies that the Lie algebra $\widetilde{\p_{\Lambda}}$ is algebraic in $\g\mathfrak l\bigl(\widetilde\p\bigr)$. 
 \end{proof}
 
 \begin{Rq}\rm
 We may pay attention that the centre of $\widetilde{\p_{\Lambda}}$ need not be reduced to $\{0\}$ and then  $\widetilde{\p_{\Lambda}}$ cannot be viewed as a subalgebra of $\g\l(\widetilde{\p_{\Lambda}})$ in general. For instance take $\g$ to be simple of type ${\rm B}_3$ and consider  the parabolic subalgebra $\p$ associated with $\pi'=\pi\setminus\{\alpha_2\}$ (with Bourbaki's notation, \cite[Planche II]{BOU}). Then $\widetilde{\p_{\Lambda}}=\widetilde\p'$ (since $\h_{\Lambda}=\{0\}$ by Remark \ref{Rqtruncation}). Denote by $\beta=\alpha_1+2\alpha_2+2\alpha_3$ the highest root in $\g$ and $x_{\beta}$ a nonzero root vector of weight $\beta$. Then one checks that $x_{\beta}\in\mathfrak z\bigl(\widetilde\p'\bigr)\cap\m$ and that $x_{\beta}$ does not belong to the derived subalgebra $\bigl[\widetilde\p',\,\widetilde\p'\bigr]_{\widetilde\p'}$ of $\widetilde\p'$. Thus $\mathfrak z(\widetilde{\p_{\Lambda}})\neq\{0\}$ and by the above, $\m$ is the largest nilpotent ideal of $\widetilde{\p_{\Lambda}}$ (and also its radical) but $\m$ strictly contains the nilpotent radical ${\rm rad}\,\widetilde\p'\cap\bigl[\widetilde\p',\,\widetilde\p'\bigr]_{\widetilde\p'}$ of $\widetilde\p'=\widetilde{\p_{\Lambda}}$ in this case.
\end{Rq}

\begin{Rq}\rm \label{Rqtrunc}
Recall that $\widetilde\p'=\mathfrak r'\ltimes\m^a$ is the derived subalgebra of $\widetilde\p$. One may observe that $\widetilde\p'$ is also the In\"on\"u-Wigner contraction $\widetilde{\p'}$ of the derived subalgebra $\p'$ of $\p$ with respect to the decomposition $\p'=\mathfrak r'\oplus\m$, as defined in equation (\ref{deficontract}). 
\end{Rq}
Recall that $\widetilde\p_{\Lambda}$ denotes the canonical truncation of $\widetilde\p$. 
\begin{lm}\label{lmtrunc}
We have that
\begin{equation} \widetilde\p_{\Lambda}\subset\widetilde{\p_{\Lambda}}.\end{equation}

\end{lm}

\begin{proof}
Recall the notation in subsection \ref{symsi},  notably the notation $\Lambda(M)$ for the set of weights of an $\h$-module $M$. Denote by $U(\a)$ the enveloping algebra of any finite-dimensional Lie algebra $\a$ and recall that the dual vector space $U(\a)^*$ of $U(\a)$ inherits a structure of associative (and commutative) algebra through the dual map of the coproduct in $U(\a)$. For some representation $\rho$ of $\a$, denote by $C(\rho)\subset U(\a)^*$ the vector space formed by matrix coefficients of $\rho$ (see for instance \cite[2.7.8]{D}). Finally  denote by $C(\a)$ the sum $\sum_{\rho}C(\rho)$ where the sum runs over the finite-dimensional representations $\rho$ of $\a$. By \cite[2.7.12]{D} $C(\a)$ is a subalgebra of $U(\a)^*$.
In \cite[Thm. 9.6.1]{F00} we have exhibited a polynomial algebra and also an $\h$-module $\widetilde{C}_{\mathfrak r}^{U(\mathfrak r')}\subset C\bigl(\widetilde\p^-\bigr)$ which is   formed by some invariant matrix coefficients under the coadjoint representation of  $U(\mathfrak r')$ (see \cite[2.3,\,2.4,\,2.5]{F00} for more details).
Then \cite[Thm. 9.6.1]{F00} implies that  $\Lambda\bigl(\widetilde\p\bigr)$ contains the set $\Lambda( \widetilde{C}_{\mathfrak r}^{U(\mathfrak r')})$. 
By \cite[Proof of Prop. 8.2.2]{F00} and \cite[7.1]{FJ2} the set $\Lambda( \widetilde{C}_{\mathfrak r}^{U(\mathfrak r')})$ is also the set of weights of the lower bound for $Sy(\p)$ constructed in \cite{FJ2}  and
by \cite[5.4.2, 7.1]{FJ2} we have more precisely that
$$\Lambda( \widetilde{C}_{\mathfrak r}^{U(\mathfrak r')})\subset\Lambda(\p)\subset\frac{1}{2}\Lambda( \widetilde{C}_{\mathfrak r}^{U(\mathfrak r')}).$$
Then in particular we have that
\begin{equation}2\Lambda(\p)\subset\Lambda\bigl(\widetilde\p\bigr).\label{incpoids}\end{equation}

Since $\widetilde\p'\subset\widetilde\p_{\Lambda}$ and since $\widetilde\p'=\mathfrak r'\ltimes\m^a$, there exists a vector subspace $\widetilde\h_{\Lambda}$ of $\h^{\pi\setminus\pi'}$ such that $\widetilde\p_{\Lambda}=\widetilde\p'\oplus\widetilde\h_{\Lambda}$.
Now take $h\in\widetilde\h_{\Lambda}$. Then by definition of the canonical truncation, $\lambda(h)=0$ for all $\lambda\in\Lambda\bigl(\widetilde\p\bigr)$. By the inclusion (\ref{incpoids}) we deduce that, for all $\lambda'\in\Lambda(\p)$, $\lambda'(h)=0$. Then $h\in\p_{\Lambda}\cap\h^{\pi\setminus\pi'}=\h_{\Lambda}$ by (\ref{truncp}).
We deduce that
\begin{equation*}\widetilde\p_{\Lambda}=\widetilde\p'\oplus\widetilde\h_{\Lambda}\subset\widetilde\p'\oplus\h_{\Lambda}=\widetilde{\p_{\Lambda}}.\end{equation*}
\end{proof}

The previous Lemma and Corollary \ref{propcor2} imply the following.
\begin{cor}\label{trunccor}
Let $\p$ be a standard parabolic subalgebra in a simple Lie algebra $\g$ and $\widetilde\p$ its In\"on\"u-Wigner contraction with respect  to the decomposition $\p=\mathfrak r\oplus\m$ with $\mathfrak r$ being the standard Levi factor and $\m$ the nilpotent radical of $\p$. If the canonical truncation $\p_{\Lambda}$ of $\p$ is equal to the derived subalgebra $\p'$ of $\p$, then the canonical truncation $\widetilde\p_{\Lambda}$ of $\widetilde\p$ verifies the equality
$$\widetilde\p_{\Lambda}=\widetilde\p'.$$ 
Moreover $$Sy\bigl(\widetilde\p\bigr)=Y\bigl(\widetilde\p'\bigr)=Sy\bigl(\widetilde\p'\bigr).$$
\end{cor}
\begin{proof}

It is an immediate consequence of the above Lemma, since in this case we have 
$$\widetilde\p'\subset\widetilde\p_{\Lambda}\subset\widetilde{\p'}=\widetilde\p'$$
by Remark \ref{Rqtrunc}. Applying Corollary \ref{propcor2} completes the proof.\end{proof}

\begin{Rq}\label{truncrq}\rm
By Remark \ref{Rqtruncation} the condition that $\p_{\Lambda}=\p'$ of the above Corollary  holds in particular when $\p$ is a maximal parabolic subalgebra.\end{Rq}

\section{Adapted pairs and Weierstrass sections.}

Let $\a$ be a finite-dimensional Lie algebra.
We will recall in this section the definition of an adapted pair for $\a$ and of a Weierstrass section for $Y(\a)$.

\subsection{The index}\label{index}
The index of $\a$, denoted by ${\rm index}\,\a$, is the integer defined by
$${\rm index}\,\a=\min_{f\in\a^*}\dim\a^f.$$
where for $f\in\a^*$, $\a^f=\{x\in\a\mid\forall y\in\a,\,f([x,\,y])=0\}$. An element $f\in\a^*$ is said to be regular if $\dim a^f={\rm index}\,\a$.

If $\a$ is algebraic, denoting by ${\bf A}$ its adjoint group, then ${\rm index}\,\a$ is also given by
\begin{equation}{\rm index}\,\a=\min_{f\in\a^*}\codim {\bf A}.f\label{defind}\end{equation}
where ${\bf A}.f$ is the coadjoint orbit of $f$. Then $f\in\a^*$ is regular if and only if $\codim{\bf A}.f={\rm index}\,\a$ or if and only if $\dim{\bf A}.f$ is maximal.

For every subalgebra $S'$ of $S(\a)$, denote by ${\rm GKdim}\,S'$ the Gelfand-Kirillov dimension of $S'$ (see \cite[1.2]{BK}), which is also equal to the transcendence degree ${\rm degtr}_{\Bbbk}$ over $\Bbbk$ of the field of fractions ${\rm Frac}\,S'$ of $S'$, since $S'$ is a commutative algebra and a domain (see \cite[2.1]{BK}).

If $\a$ is algebraic, then a result of Chevalley-Dixmier (\cite[Lem. 7]{D0}) implies the equality
\begin{equation}{\rm index}\,\a={\rm degtr}_{\Bbbk}({\rm Frac}\,S(\a))^{\a}\label{Ro}\end{equation}
where $({\rm Frac}\,S(\a))^{\a}$ is the field of invariants under the induced adjoint action of $\a$ in the field of fractions ${\rm Frac}\,S(\a)$ of $S(\a)$.
The above equality is also known as a Rosenlicht theorem (\cite{Ro}).

Assume now that $\a$ is algebraic and recall that  its canonical truncation $\a_{\Lambda}$ is also algebraic. Then equations (\ref{semi}) and (\ref{Ro}) imply that
\begin{equation}{\rm index}\,\a_{\Lambda}={\rm degtr}_{\Bbbk}{\rm Frac}\,Y(\a_{\Lambda})={\rm GKdim}\,Y(\a_{\Lambda})={\rm GKdim}\,Sy(\a).\label{ind}\end{equation}
Indeed since  $Y(\a_{\Lambda})=Sy(\a_{\Lambda})$, we have that
$({\rm Frac}\,S(\a_{\Lambda}))^{\a_{\Lambda}} ={\rm Frac}\,Y(\a_{\Lambda})$ (see for instance \cite[Chap. I, Sec. B, 5.11]{F}). 

As a consequence of Lemma \ref{lmtrunc} we have the following Proposition.

\begin{prop}\label{propindex}
Let $\widetilde\p=\mathfrak r\ltimes \m^a$ be the In\"on\"u-Wigner contraction of $\p=\mathfrak r\oplus\m$ and $\widetilde\p_{\Lambda}$ be the canonical truncation of $\widetilde\p$. Let $\widetilde{\p_{\Lambda}}=\mathfrak r_{\rm trunc}\ltimes\m^a$ be the In\"on\"u-Wigner contraction of $\p_{\Lambda}=\mathfrak r_{\rm trunc}\oplus\m$. Then one has that
\begin{equation} {\rm index}\,\widetilde\p_{\Lambda}={\rm index}\,\widetilde{\p_{\Lambda}}.\label{inegindex}\end{equation}

\end{prop}

\begin{proof}
Firstly we show that the set of weights $\Lambda\bigl(\widetilde{\p_{\Lambda}}\bigr)$  of $Sy\bigl(\widetilde{\p_{\Lambda}}\bigr)$ is reduced to $\{0\}$. Then  \begin{equation}(\widetilde{\p_{\Lambda}})_{\Lambda}=\widetilde{\p_{\Lambda}}\label{egSyy}\end{equation} and
\begin{equation}Sy\bigl(\widetilde{\p_{\Lambda}}\bigr)=Y\bigl(\widetilde{\p_{\Lambda}}\bigr).\label{Y}\end{equation}

Indeed let $\lambda\in\Lambda(\widetilde{\p_{\Lambda}})$. Since $\widetilde{\p_{\Lambda}}=\widetilde\p'\oplus\h_{\Lambda}$ as a vector space, we have (by a similar argument as in Remark \ref{semiinvp}) that \begin{equation}Sy\bigl(\widetilde{\p_{\Lambda}}\bigr)=S\bigl(\widetilde{\p_{\Lambda}}\bigr)^{\widetilde\p'}\label{egSyyy}\end{equation} and then one may view $\lambda$ as an element in $\h_{\Lambda}^*$.
There exists $s\in S\bigl(\widetilde{\p_{\Lambda}}\bigr)_{\lambda}\setminus\{0\}$. Since as vector spaces and also as $\h_{\Lambda}$-modules (see (\ref{brackettildep})) $\widetilde{\p_{\Lambda}}=\p_{\Lambda}$, one has that $S\bigl(\widetilde{\p_{\Lambda}}\bigr)=S\bigl(\p_{\Lambda}\bigr)$ as algebras and $\h_{\Lambda}$-modules. Then  one has $s\in S\bigl(\p_{\Lambda}\bigr)_{\lambda}\subset Sy\bigl(\p_{\Lambda}\bigr)=Y\bigl(\p_{\Lambda}\bigr)$. It follows that $\lambda=0$ and equations (\ref{egSyy}) and (\ref{Y}) are true.

Secondly one has that 
\begin{equation}Sy\bigl(\widetilde\p\bigr)=Sy\bigl(\widetilde\p_{\Lambda}\bigr)=S\bigl(\widetilde\p_{\Lambda}\bigr)^{\widetilde\p'}.\label{egSSy}\end{equation}

Indeed since $\widetilde\p'\subset\widetilde\p_{\Lambda}\subset \widetilde\p$ one has, by Remark \ref{semiinvp},  that
\begin{equation*}Sy\bigl(\widetilde\p\bigr)=S\bigl(\widetilde\p\bigr)^{\widetilde\p'}\supset S\bigl(\widetilde\p_{\Lambda}\bigr)^{\widetilde\p'}\supset S\bigl(\widetilde\p_{\Lambda}\bigr)^{\widetilde\p_{\Lambda}}=Sy\bigl(\widetilde\p\bigr)\label{}\end{equation*}
by equation (\ref{egSy}).

Equations (\ref{egSyyy}), (\ref{egSSy}) and the inclusion $\widetilde\p_{\Lambda}\subset\widetilde{\p_{\Lambda}}$ (by Lemma \ref{lmtrunc}) imply  that
\begin{equation}Sy\bigl(\widetilde\p\bigr)\subset S\bigl(\widetilde{\p_{\Lambda}}\bigr)^{\widetilde\p'}=Sy\bigl(\widetilde{\p_{\Lambda}}\bigr)=Y\bigl(\widetilde{\p_{\Lambda}}\bigr)\label{incc}\end{equation} by (\ref{Y}).

Then by \cite[3.1]{BK} and equation (\ref{ind}) one has that
$${\rm index}\,\widetilde\p_{\Lambda}={\rm GKdim}\,Sy\bigl(\widetilde\p\bigr)\le {\rm GKdim}\,Sy\bigl(\widetilde{\p_{\Lambda}}\bigr)={\rm index}\,(\widetilde{\p_{\Lambda}})_{\Lambda}.$$
Equation (\ref{egSyy}) gives the  inequality ${\rm index}\,\widetilde\p_{\Lambda}\le{\rm index}\,\widetilde{\p_{\Lambda}}$.

Moreover by (\ref {egSyyy}) one has that
$$Sy\bigl(\widetilde{\p_{\Lambda}}\bigr)=S\bigl(\widetilde{\p_{\Lambda}}\bigr)^{\widetilde\p'}\subset S\bigl(\widetilde\p\bigr)^{\widetilde\p'}=Sy\bigl(\widetilde\p\bigr)$$ since $\widetilde{\p_{\Lambda}}\subset\widetilde\p$ and by Remark \ref{semiinvp}.

Then by \cite[3.1]{BK} one has that
$${\rm GKdim}\,Sy\bigl(\widetilde{\p_{\Lambda}}\bigr)\le{\rm GKdim}\,Sy\bigl(\widetilde\p\bigr)$$
that is,
$${\rm index}\,\widetilde{\p_{\Lambda}}\le{\rm index}\,\widetilde\p_{\Lambda}.$$
This completes the proof. \end{proof}

For any algebraic Lie algebra $\a$, denote by $G(\a)=\mathbb Z\Lambda(\a)$ the additive group generated by the set of weights $\Lambda(\a)$ of $Sy(\a)$. By \cite[Appendice C]{FJ3} one knows that this group is a free abelian group of finite type. Denote by $n(\a)$ its rank. By \cite[Chap. I, Sec. B, 9.6]{F} one has that
\begin{equation}{\rm index}\,\a={\rm GKdim}\,Sy(\a)-n(\a).\label{indexn}\end{equation}
Moreover the set of weights $\Lambda(\p)$ of $Sy(\p)$, resp. $\Lambda\bigl(\widetilde\p\bigr)$ of $Sy\bigl(\widetilde\p\bigr)$, is included in $\sum_{\alpha\in\pi\setminus\pi'}\mathbb Z\varpi_{\alpha}$. Hence the rank $n(\p)$ of $G(\p)$, resp. $n\bigl(\widetilde\p\bigr)$ of $G\bigl(\widetilde\p\bigr)$, is equal to the dimension of the $\Bbbk$-vector space $\Bbbk\Lambda(\p)$, resp. $\Bbbk\Lambda\bigl(\widetilde\p\bigr)$, generated by $\Lambda(\p)$, resp. $\Lambda\bigl(\widetilde\p\bigr)$.
 One has the following.
\begin{prop}\label{egindex}
Assume that :
$${\rm index}\,\p_{\Lambda}={\rm index}\,\widetilde{\p_{\Lambda}}.$$
Then $\widetilde\p_{\Lambda}=\widetilde{\p_{\Lambda}}$. 

\end{prop}
\begin{proof}
Since $\widetilde\p$ is algebraic (Cor. \ref{propcor1}), one has by (\ref{indexn}) that
$${\rm index}\,\widetilde\p={\rm GKdim}\,Sy\bigl(\widetilde\p\bigr)-n\bigl(\widetilde\p\bigr)={\rm index}\,\widetilde\p_{\Lambda}-n\bigl(\widetilde\p\bigr)$$ by (\ref{ind}).
Similarly one has that
$${\rm index}\,\p={\rm GKdim}\,Sy(\p)-n(\p)={\rm index}\,\p_{\Lambda}-n(\p).$$

Recall the inclusion (\ref{incpoids}).  This implies that $\Bbbk\Lambda(\p)\subset\Bbbk\Lambda\bigl(\widetilde\p\bigr)$ and then that $n(\p)\le n\bigl(\widetilde\p\bigr)$ by what we said above.

 Then the above Proposition and the
 hypothesis  imply that ${\rm index}\,\widetilde\p\le{\rm index}\,\p$.
Now the index cannot decrease under contraction (see for instance \cite[Sec. 4]{Y1}). Then our hypothesis implies that ${\rm index}\,\widetilde\p={\rm index}\,\p$ and that $n\bigl(\widetilde\p\bigr)=n(\p)$. Hence we have that \begin{equation}\Bbbk\Lambda(\p)=\Bbbk\Lambda\bigl(\widetilde\p\bigr).\label{egev}\end{equation}
Recall that 
$\widetilde\p_{\Lambda}=\widetilde\p'\oplus\widetilde\h_{\Lambda}\subset\widetilde{\p_{\Lambda}}=\widetilde\p'\oplus\h_{\Lambda}$ (Lemma \ref{lmtrunc}) where $\widetilde\h_{\Lambda}$ is a subspace of $\h_{\Lambda}=\p_{\Lambda}\cap\h^{\pi\setminus\pi'}$. We want to prove the inverse inclusion. 
Take $h\in\h_{\Lambda}$. 
For all $\lambda\in \Bbbk\Lambda(\p)$, we have that $\lambda(h)=0$ by definition of the canonical truncation (\ref{deftrunc}). It follows that $\lambda(h)=0$ for all $\lambda\in\Lambda\bigl(\widetilde\p\bigr)$ by equality (\ref{egev}) and then that $h$ vanishes on $\Lambda\bigl(\widetilde\p\bigr)$. This means that $h\in\widetilde\p_{\Lambda}\cap\h^{\pi\setminus\pi'}=\widetilde\h_{\Lambda}$, which gives the required inclusion.
\end{proof}

\subsection{Adapted pair}\label{AP}

Denote still by ad the coadjoint action of $\a$ on $\a^*$. By say \cite[6.1]{JS}, we have the following definition.
\begin{defi}\label{APdefi}\rm Let $\a$ be an algebraic Lie algebra.
An adapted pair $(h,\,y)$ for $\a$ is a pair formed by an ad-semisimple (as endomorphism of $\a^*$) element $h\in\a$ and a regular element $y\in\a^*$ such that $\ad h(y)=-y$. 
\end{defi}

\subsection{Weierstrass section}\label{WS}

We still consider an algebraic Lie algebra $\a$. Recall the algebra of symmetric invariants $Y(\a)\subset \Bbbk[\a^*]$. The definition below may be found in \cite{FJ4}. 

\begin{defi}\label{WSdefi}\rm
A Weierstrass section for $Y(\a)$ (or for $\a$, for short) is an affine subset $y+V$ of $\a^*$ (with $y\in\a^*$ and $V$ a vector subspace of $\a^*$) such that restriction of functions induces an algebra isomorphism between 
$Y(\a)$ and the algebra of polynomial functions $\Bbbk[y+V]$ on $y+V$.\end{defi}

Since the algebra $\Bbbk[y+V]$ is isomorphic to the  symmetric algebra $S(V^*)$, the existence of a Weierstrass section for $Y(\a)$ implies obviously the polynomiality of $Y(\a)$ but the inverse is not always true (see for instance \cite[Chap. VI, Sec. A, 2.10]{F}). The notion of a Weierstrass section was introduced by the Russian school and in particular by Popov (see \cite[2.2.1]{Po}) to linearize invariant generators in $S(X^*)^A$ in the case when a semisimple Lie algebra $\a$ (with adjoint group $A$) acts on a finite-dimensional vector space $X$. This definition coincides with definition in \ref{WSdefi} when $\a$ acts on $X=\a^*$ by coadjoint action (note that in \ref{WSdefi}, semisimplicity of $\a$ is not needed). Similarly in \cite{Po} an analogue of an adapted pair as defined in \ref{APdefi} was given. In particular when $\a$ is a simple Lie algebra, Popov showed \cite[2.2.10]{Po} that this analogue of an adapted pair exists whenever $X$ is a simple $\a$-module : this may fail if $X$ is not a simple $\a$-module, \cite[2.2.16, Example 3]{Po}. When such an adapted pair exists then he showed that it provides a Weierstrass section for $S(X^*)^A$ (see \cite[1.2]{FJ4} for more details).

A first example of a Weierstrass section is given by the socalled Kostant section (see \cite[2.2.2]{Po}).
More precisely take $\a=\g$ to be a semisimple Lie algebra. Then there exists a principal $\s\l_2$-triple $(x,\,h,\,y)$ for $\g$ with $h$ ad-semisimple and $x$ and $y$ regular in $\g^*\simeq\g$ such that $[h,\,y]=-y$. It follows that $(h,\,y)$ is an adapted pair for $\g$ and Kostant \cite{K} showed that $y+\g^x$ is a Weierstrass section for $Y(\g)$ : this Weierstrass section is called a Kostant section or a Kostant slice for $\g$. Observe also that we have that $\g\simeq\g^*=\ad\g(y)\oplus\g^x$.

More generally by \cite[6.3]{JS}, we have the following Theorem.

\begin{thm} Let $\a$ be an algebraic Lie algebra such that $Sy(\a)=Y(\a)$.
If $Y(\a)$ is a polynomial algebra and if $(h,\,y)$ is an adapted pair for $\a$, then $y+V$ is a Weierstrass section for $Y(\a)$, where $V$ is an $\ad h$-stable complement to $\ad\a(y)$ in $\a^*$.
\end{thm}

The disadvantage of the previous Theorem is that polynomiality of $Sy(\a)=Y(\a)$ is required to apply it.

Another result in \cite[Lem. 6.11]{J6bis} will give the Propositions \ref{WScontractionprop} and \ref{WScontractionpropgen} below : this will be more useful in our present study of polynomiality of $Sy\bigl(\widetilde\p\bigr)=Y\bigl(\widetilde\p_{\Lambda}\bigr)=Sy\bigl(\widetilde\p_{\Lambda}\bigr)$. But before applying \cite[Lem. 6.11]{J6bis} in the case of $\widetilde\p_{\Lambda}$, we have to compute the Gelfand-Kirillov dimension of $Y\bigl(\widetilde\p_{\Lambda}\bigr)$, which is also equal to ${\rm index}\,\widetilde\p_{\Lambda}$ by (\ref{ind}). Under some hypotheses we will give below, we will show that ${\rm index}\,\widetilde\p_{\Lambda}$ is equal to ${\rm index}\,\p_{\Lambda}$, which we know how to compute it (see for instance \cite[Prop. 3.2]{FJ1}).

\subsection{Coadjoint action}\label{coact}

Through the Killing form $K$ of $\g$, we have that $\p^*={\widetilde\p}^*=\mathfrak r\oplus\m^-=\p^-$ as a vector space. 

Recall that the coadjoint action, which we still denote by $\ad$, of $\p$ on $\p^-$ is given by the following :

\begin{equation}\forall x\in\p,\,\forall y\in\p^-,\,\ad x(y)={\rm pr}_{\p^-}([x,\,y])\end{equation}
where ${\rm pr}_{\p^-}$ is the projection of $\g=\p^-\oplus\m$ onto $\p^-$.
Hence we have that :
\begin{equation}\forall x,\,x'\in\p,\forall y\in\p^-,\,K(y,\,[x,\,x'])=K(x,\,\ad x'(y)).\label{actioncoad}\end{equation}

By \cite[6.2]{F00}  the coadjoint action of $\widetilde\p$ on ${\widetilde\p}^*$, which we denote by $\ad^*$, is given by the following  :
\begin{align}\forall x\in\mathfrak r,\,\forall y\in\p^-, \ad^*x(y)=[x,\,y]\label{coad1}\\
\forall x\in\m,\,\forall y\in\p^-,\,\ad^*x(y)={\rm pr}_{\mathfrak r}([x,\,y])\label{coad2}\end{align}
where ${\rm pr}_{\mathfrak r}$ is the projection of $\g=\mathfrak r\oplus\m\oplus\m^-$ onto $\mathfrak r$.

By \cite[Lem. 9.3.1]{F00} we have that
\begin{equation} \forall x,\,x'\in\widetilde\p,\,\forall y\in\p^-,\,K(y,\,[x,\,x']_{\widetilde\p})=K(x,\,\ad^*x'(y)).\label{actioncoadcont}\end{equation}

\begin{Rqs}\rm

\label{rqs} 
(1) Observe that $\ad^*$ induces a coadjoint action of $\widetilde\p'$ on $\widetilde\p'^*=\p'^-=[{\p}^-,\,{\p}^- ]$ (as a vector space) similarly as in (\ref{coad1}) and (\ref{coad2}) with in (\ref{coad2}) ${\rm pr}_{\mathfrak r}$ replaced by ${\rm pr}_{\mathfrak r'}$ where ${\rm pr}_{\mathfrak r'}$ is the projection of $\g=\mathfrak r'\oplus\h^{\pi\setminus\pi'}\oplus\m\oplus\m^-$ onto $\mathfrak r'$. Of course we still have (\ref{actioncoadcont}), with $\widetilde\p$ replaced by $\widetilde\p'$ and $\p^-$ replaced by $\p'^-$.

(2) Recall the notation of subsection \ref{trunc}. Denote by $\g_{\mathbb Q}$ the $\mathbb Q$-vector space generated by the Chevalley basis of $\g$ formed by the $\alpha^{\vee}$, for all $\alpha\in\pi$, and the $x_{\gamma}$, for all $\gamma\in\Delta$, and set $\h_{\mathbb Q}^{\pi\setminus\pi'}=\g_{\mathbb Q}\cap\h^{\pi\setminus\pi'}$ and $(\h_{\Lambda})_{\mathbb Q}=\h_{\Lambda}\cap\g_{\mathbb Q}$. Since the restriction of $K$ to $\h^{\pi\setminus\pi'}\times\h^{\pi\setminus\pi'}$ is nondegenerate (see for instance \cite[5.2.2]{FJ3}),  it follows that the restriction of $K$ to $\h_{\mathbb Q}^{\pi\setminus\pi'}\times\h_{\mathbb Q}^{\pi\setminus\pi'}$ is positive definite. Hence the orthogonal  $\h'_{\mathbb Q}$ of $(\h_{\Lambda})_{\mathbb Q}$ in $\h_{\mathbb Q}^{\pi\setminus\pi'}$ with respect to the Killing form $K$  is such that $\h_{\mathbb Q}^{\pi\setminus\pi'}=(\h_{\Lambda})_{\mathbb Q}\oplus\h'_{\mathbb Q}$. Then set $\h'=\h'_{\mathbb Q}\otimes_{\mathbb Q}\Bbbk$, so that $\h'\oplus\h_{\Lambda}=\h^{\pi\setminus\pi'}$ and $K(\h_{\Lambda},\,\h')=0$. We obtain that $\ad^*$ induces a coadjoint action of $\widetilde{\p_{\Lambda}}$ on the vector space $\widetilde{\p_{\Lambda}}^*$ as in (\ref{coad1}) and (\ref{coad2}) with in (\ref{coad2}) ${\rm pr}_{\mathfrak r}$ replaced by ${\rm pr}_{\mathfrak r'\oplus\h_{\Lambda}}$ where ${\rm pr}_{\mathfrak r'\oplus\h_{\Lambda}}$ is the projection of $\g=\mathfrak r'\oplus\h_{\Lambda}\oplus\h^{\prime}\oplus\m\oplus\m^-$ onto $\mathfrak r'\oplus\h_{\Lambda}$.
\end{Rqs}

\subsection{The index of the canonical truncation of the contraction }\label{indexcontraction}


\begin{lm}\label{indexcontractionlm}
Let $\widetilde\p$ be the In\"on\"u-Wigner contraction of the standard parabolic subalgebra $\p$ with respect to the decomposition $\p=\mathfrak r\oplus\m$ with $\mathfrak r$ being the standard Levi factor and $\m$ being the nilpotent radical of $\p$. 
We assume that :
\begin{enumerate}
 \item[i)] There exist an element $y\in\widetilde\p'^*$ and  a vector subspace $V$ of $\widetilde\p'^*$ such that $$\ad^*\widetilde\p'(y)+ V=\widetilde\p'^*.$$
 \item[ii)]  $\dim V={\rm index}\,\p'$.
 \end{enumerate}
 Then $${\rm index}\,\p'={\rm index}\,\widetilde\p'$$ 
and the sum $\ad^*\widetilde\p'(y)+ V$ is direct.

\end{lm}

\begin{proof}
Since the index  cannot decrease under contraction \cite[Sec. 4]{Y1} and by Remark (\ref{Rqtrunc}) we have already that ${\rm index}\,\widetilde\p'\ge{\rm index}\,\p'$.
Since $\widetilde\p'$ is  the derived subalgebra of $\widetilde\p\subset\g\l\bigl(\widetilde\p\bigr)$, then $\widetilde\p'$ is algebraic in $\g\l\bigl(\widetilde\p\bigr)$ by \cite[2.4.5]{J2}. Now the hypotheses $i)$ and $ii)$ and the definition of the index (\ref{defind}) imply that
$$\dim V={\rm index}\,\p'={\rm index}\,\widetilde\p'=\codim_{\widetilde\p'^*}\ad^*\widetilde\p'(y)$$ 
and that the sum $\ad^*\widetilde\p'(y)+ V$ is direct.
\end{proof}

\begin{cor}\label{indexcontractioncor}
Keep the same hypotheses as in the previous Lemma and assume further that $\p_{\Lambda}=\p'$, where $\p_{\Lambda}$ is the canonical truncation of $\p$.
Then \begin{equation}{\rm GKdim}\,Sy\bigl(\widetilde\p\bigr)={\rm GKdim}\,Sy(\p)={\rm index}\,\p'={\rm index}\,\widetilde\p'.\end{equation}
\end{cor}
\begin{proof}

Since $\p$ is algebraic, we have by (\ref{ind})  that $${\rm GKdim}\,Sy(\p)={\rm GKdim}\,Y(\p_{\Lambda})={\rm index}\,\p_{\Lambda}={\rm index}\,\p'$$
by hypothesis.
Similarly since $\widetilde\p$ is algebraic, we also have that $${\rm GKdim}\,Sy\bigl(\widetilde\p\bigr)={\rm GKdim}\,Y\bigl(\widetilde\p_{\Lambda}\bigr)={\rm index}\,\widetilde\p_{\Lambda}={\rm index}\,\widetilde\p'$$  
by Corollary \ref{trunccor}. Since ${\rm index}\,\p'={\rm index}\,\widetilde\p'$ by the previous Lemma, the required equality follows.
\end{proof}

\begin{Rq}\label{indexcontractionrq}\rm

If the hypotheses $i)$ and $ii)$ of the previous Lemma hold, then one has that
\begin{equation}\dim\m-\dim\mathfrak r'\le {\rm index}\,\p'.\end{equation}

Indeed by (\ref{coad2}) $\dim\ad^*\m(y)\le\dim\mathfrak r'$ and of course $\dim\ad^*\mathfrak r'(y)\le\dim\mathfrak r'$, hence
$\dim\ad^*\widetilde\p'(y)\le2\dim\mathfrak r'$.

\end{Rq}

\subsection{}

Lemma \ref{indexcontractionlm} can be generalized as follows. 

\begin{lm}\label{indexcontractionlmgen}

Assume that :
\begin{enumerate}
\item[i)] There exists $y\in\widetilde{\p_{\Lambda}}^*$ and a vector subspace $V$ of $\widetilde{\p_{\Lambda}}^*$ such that $$\ad^*\widetilde{\p_{\Lambda}}(y)+V=\widetilde{\p_{\Lambda}}^*$$ 
\item[ii)] $\dim V={\rm index}\,\p_{\Lambda}.$
\end{enumerate}
Then one has that $${\rm index}\,\widetilde{\p_{\Lambda}}={\rm index}\,\p_{\Lambda}$$
and the sum $\ad^*\widetilde{\p_{\Lambda}}(y)+V$ is direct.
\end{lm}

\begin{proof}

Since the index cannot decrease under contraction \cite[Sec. 4]{Y1} we have ${\rm index}\,\widetilde{\p_{\Lambda}}\ge{\rm index}\,\p_{\Lambda}$.
Now hypothesis $i)$ implies that
$$\dim V\ge\codim_{\widetilde{\p_{\Lambda}}^*}\ad^*\widetilde{\p_{\Lambda}}(y)\ge{\rm index}\,\widetilde{\p_{\Lambda}}$$
by (\ref{defind}) since $\widetilde{\p_{\Lambda}}$ is algebraic (by Lemma \ref{truncalgebraic}). Hypothesis $ii)$ implies then that ${\rm index}\,\widetilde{\p_{\Lambda}}={\rm index}\,\p_{\Lambda}$ and that the sum $\ad^*\widetilde{\p_{\Lambda}}(y)+V$ is direct.
\end{proof}

Similarly as in Remark \ref{indexcontractionrq} we have the following Remark.

\begin{Rq}\label{indexcontractionrqgen}\rm

If the hypotheses $i)$ and $ii)$ of the previous Lemma hold, then one has that
\begin{equation}\dim\m-\dim\mathfrak r'-\dim\h_{\Lambda}\le {\rm index}\,\p_{\Lambda}.\end{equation}
\end{Rq}

Similarly as in Corollary \ref{indexcontractioncor} we have the following.
\begin{cor}\label{indexcontractioncorgen}
Keep the same hypotheses as in the previous Lemma.
Then \begin{equation}{\rm GKdim}\,Sy\bigl(\widetilde\p\bigr)={\rm GKdim}\,Sy(\p)={\rm index}\,\p_{\Lambda}={\rm index}\,\widetilde\p_{\Lambda}.\end{equation}
\end{cor}
\begin{proof}
Indeed the hypotheses of the previous Lemma
imply  by Proposition \ref{egindex} that $\widetilde\p_{\Lambda}=\widetilde{\p_{\Lambda}}$. Then it suffices to apply equation (\ref{ind}).
\end{proof}

\subsection{A Weierstrass section for  the contraction}\label{WScontraction}

Now we can give following Propositions which are a direct application of \cite[Lem. 6.11]{J6bis} and which will be very useful in our paper.
Recall $\Delta(\pi')=\Delta^+\sqcup\Delta_{\pi'}^-$ which is the subset of roots of $(\p,\,\h)$ and also of $(\widetilde\p,\,\h)$.
Recall also the formal character ${\rm ch}\,M$ of an $\h$-module $M=\bigoplus_{\nu\in\h^*}M_{\nu}$ having finite-dimensional weight subspaces $M_{\nu}$. This is given by the formula :
$${\rm ch}\,M=\sum_{\nu\in\h^*}\dim M_{\nu}\,e^{\nu}$$
where for $\nu,\,\nu'\in\h^*$, $e^{\nu+\nu'}=e^{\nu}e^{\nu'}$ so that, for two $\h$-modules $M$, $N$ having a decomposition with finite-dimensional weight subspaces,
one has that $${\rm ch}\,(M\otimes N)={\rm ch}\,M{\rm ch}\,N.$$
We write ${\rm ch}\,M\le{\rm ch}\,N$ whenever $\dim M_{\nu}\le\dim N_{\nu}$ for all $\nu\in\h^*$ (it is for example the case if $M$ is a submodule of $N$).

\begin{prop}\label{WScontractionprop}
We keep the notation and hypotheses i) and ii) of Lemma \ref{indexcontractionlm} and we assume further that :
\begin{enumerate}
\item[iii)] There exists a subset $S\subset\Delta(\pi')$   such that $S_{\mid\h_{\pi'}}$ is a basis for $\h_{\pi'}^*$ and that
$$y=\sum_{\gamma\in S} x_{-\gamma}.$$

\item[iv)] There exists a subset $T\subset\Delta(\pi')$   such that $S\cap T=\emptyset$ and
$$V=\sum_{\gamma\in T} \Bbbk x_{-\gamma}.$$ 
\end{enumerate}
Then $\lvert T\rvert={\rm index}\,\p'={\rm index}\,\widetilde\p'$ and since $S_{\mid\h_{\pi'}}$ is a basis for $\h_{\pi'}^*$ there exists a unique $h\in\h_{\pi'}$ such that $\ad^*h(y)=-y$. Thus $(h,\,y)$ is an adapted pair for $\widetilde\p'$.
Moreover  there exists, for each $\gamma\in T$, a unique $s(\gamma)\in\mathbb Q S$ such that $\gamma+s(\gamma)$ vanishes on $\h_{\pi'}$.
Assume furthermore  that :
\begin{enumerate}
\item[v)] For any $\gamma\in T$, $s(\gamma)\in\mathbb NS$ and that $\gamma+s(\gamma)\neq 0$. 
\item[vi)] $\p_{\Lambda}=\p'$.
\end{enumerate}
Then one has that :
\begin{equation}{\rm ch\,}\,Sy\bigl(\widetilde\p\bigr)={\rm ch}\,Y\bigl(\widetilde\p'\bigr)\le\prod_{\gamma\in T}(1-e^{\gamma+s(\gamma)})^{-1}.\label{eq}\end{equation}
If equality holds in the above inequality, then the restriction map gives an isomorphism
\begin{equation} Sy\bigl(\widetilde\p\bigr)=Y\bigl(\widetilde\p'\bigr)\stackrel{\sim}{\longrightarrow}\Bbbk[y+V].\end{equation}
Hence $y+V$ is a Weierstrass section for $Sy\bigl(\widetilde\p\bigr)$ and $Sy\bigl(\widetilde\p\bigr)$ is a polynomial $\Bbbk$-algebra. 

\end{prop}

\begin{proof}
It is inspired by \cite[Lem. 6.11]{J6bis}. We will give the proof for the reader's convenience. 
Firstly we know by Lemma \ref{indexcontractionlm} that $\dim V={\rm index}\,\widetilde\p'$. 

Hypotheses $i)$ and $ii)$ imply that restriction of functions $Y\bigl(\widetilde\p'\bigr)\xhookrightarrow{\varphi} \Bbbk[y+V]$ is an injective algebras morphism. Like in \cite[Lem. 6.11]{J6bis} one may view $\varphi$ as the composition of two maps. The first one $\varphi_0: Y\bigl(\widetilde\p'\bigr)\longrightarrow\Bbbk[\Bbbk y\oplus V]$ is the restriction map on $\Bbbk y\oplus V$ which consists in sending all root vectors $x_{\alpha}$ for $\alpha\not\in S\cup T$ to zero. The second map $\varphi_{00}:\Bbbk[\Bbbk y\oplus V]\longrightarrow\Bbbk[y+ V]\simeq S(V^*)$ is the evaluation map at $y+V$, which consists in sending every root vector $x_{\alpha}$ with $\alpha\in S$ to one. The first map $\varphi_0$ is a morphism of $\ad\h$-modules, unlike the second one. As $\varphi=\varphi_{00}\circ \varphi_0$ is injective, we deduce that $\varphi_0$ is also injective. Since $\p_{\Lambda}=\p'$ we have that $\widetilde \p_{\Lambda}=\widetilde\p'$ by Corollary \ref{trunccor}. Then $Sy\bigl(\widetilde\p\bigr)=Y\bigl(\widetilde\p'\bigr)$ and the index of $\widetilde\p'$ (which is equal to $\dim V=\lvert T\rvert$) is equal to the Gelfand-Kirillov dimension of $Y\bigl(\widetilde\p'\bigr)$ by Lemma \ref{indexcontractionlm} and Corollary \ref{indexcontractioncor}.

For any $\xi\in\h^*$, denote by $\xi'$ its restriction  to $\h_{\pi'}$. Recall (2) of Remarks \ref{rqs} and in particular the notation $\g_{\mathbb Q}$. For a subspace $\a$ of $\g$, set $\a_{\mathbb Q}=\g_{\mathbb Q}\cap\a$. Then $\h^*_{\mathbb Q}=\h^*\cap\g_{\mathbb Q}^*$ is the $\mathbb Q$-vector space generated by the set $\pi$ of simple roots of $\g$ and $\Delta\subset\h^*_{\mathbb Q}$.
Since $S,\,T\subset\Delta$ and since $S_{\mid\h_{\pi'}}$ is a basis for $\h_{\pi'}^*$ (and then also a basis for ${(\h_{\pi'})_{\mathbb Q}}^*$)  it is easily seen that, for each $\gamma\in T$, there exists a unique $s(\gamma)\in\mathbb QS$ such that $s(\gamma)'=-\gamma'$ that is, such that $\gamma+s(\gamma)$ vanishes on $\h_{\pi'}$.
Writing $\gamma'=-\sum_{\alpha\in S}q_{\alpha,\gamma}\alpha'$, with $q_{\alpha,\,\gamma}\in\mathbb Q$, we have that $s(\gamma)=\sum_{\alpha\in S}q_{\alpha,\gamma}\alpha$.

Then every weight vector $f$ in $Y\bigl(\widetilde\p'\bigr)$ whose image by $\varphi$ has a monomial of the form $\prod_{\gamma\in T}x_{\gamma}^{n_{\gamma}}$ (and then whose image by $\varphi_0$ has a monomial of the form $\prod_{\alpha\in S}x_{\alpha}^{m_{\alpha}}\prod_{\gamma\in T}x_{\gamma}^{n_{\gamma}}$) must have weight
$$\sum_{\gamma\in T} n_{\gamma}\bigl(s(\gamma)+\gamma\bigr).\eqno (*)$$ Indeed $\sum_{\alpha\in S}m_{\alpha}\alpha'+\sum_{\gamma\in T}n_{\gamma}\gamma'=0$ since $f$ is $\ad\,\h_{\pi'}$-invariant and 
using that $(\alpha')_{\alpha\in S}$ is a basis for $\h_{\pi'}^*$ gives  that, for all $\alpha\in S$, $$m_{\alpha}=\sum_{\gamma\in T}n_{\gamma}q_{\alpha,\,\gamma}\eqno (**)$$ from which equality $(*)$ above follows.

Now suppose that, for every $\gamma\in T$, $s(\gamma)\in\mathbb NS$ that is, with the above notation that $q_{\alpha,\,\gamma}\in\mathbb N$ for any $\alpha\in S$. For every $\gamma\in T$, set $$a_{\gamma}=x_{\gamma}\prod_{\alpha\in S}x_{\alpha}^{q_{\alpha,\,\gamma}}\in S\bigl(\widetilde\p'\bigr).$$
Then it is easily seen that the vectors $a_{\gamma}$, $\gamma\in T$, are algebraically independent. Moreover consider $f\in Y\bigl(\widetilde\p'\bigr)$ a weight vector with $\varphi_0(f)$ having $\prod_{\alpha\in S}x_{\alpha}^{m_{\alpha}}\prod_{\gamma\in T}x_{\gamma}^{n_{\gamma}}$ as a monomial. 

One verifies easily, using equality $(**)$ that  $$\prod_{\alpha\in S}x_{\alpha}^{m_{\alpha}}\prod_{\gamma\in T}x_{\gamma}^{n_{\gamma}}=\prod_{\gamma\in T}a_{\gamma}^{n_{\gamma}}.$$ It follows that 
$$\varphi_0\bigl(Y\bigl(\widetilde\p'\bigr)\bigr)\subset \Bbbk[a_{\gamma},\,\gamma\in T].\eqno (***)$$

Finally suppose that, for each $\gamma\in T$, $\gamma+s(\gamma)\neq 0$, namely that the weight of each $a_{\gamma}$ is nonzero. Then every weight  subspace of the polynomial algebra $\Bbbk[a_{\gamma},\,\gamma\in T]$ is finite-dimensional and  the formal character of this polynomial algebra is well defined : it is equal to 
$$\prod_{\gamma\in T}(1-e^{\gamma+s(\gamma)})^{-1}.$$
Using the injectivity of $\varphi_0$ and $(***)$ completes the proof for the inequality (\ref{eq}). Finally if equality holds in the inequality (\ref{eq}) then we have equality in $(***)$ and then $\varphi\bigl(Y\bigl(\widetilde\p'\bigr)\bigr)=S(V^*)$ which gives the surjectivity of the map $\varphi$.
\end{proof}

The above Proposition may be useful in case when  the parabolic subalgebra $\p$ is maximal (by Corollary \ref{trunccor} and Remark \ref{truncrq}). However when $\widetilde\p'\subsetneq\widetilde\p_{\Lambda}$, we need
a generalization of Proposition \ref{WScontractionprop}. This is  the following Proposition.

\begin{prop}\label{WScontractionpropgen}

We keep the notation and hypotheses i) and ii) of Lemma \ref{indexcontractionlmgen} and we assume further that :
\begin{enumerate}
\item[iii)] There exists a subset $S\subset\Delta(\pi')$   such that $S_{\mid\h_{\pi'}\oplus\h_{\Lambda}}$ is a basis for $(\h_{\pi'}\oplus\h_{\Lambda})^*$ and that
$$y=\sum_{\gamma\in S} x_{-\gamma}.$$

\item[iv)] There exists a subset $T\subset\Delta(\pi')$   such that $S\cap T=\emptyset$ and
$$V=\sum_{\gamma\in T} \Bbbk x_{-\gamma}.$$ 
\end{enumerate}
Then $\lvert T\rvert={\rm index}\,\p_{\Lambda}={\rm index}\,\widetilde\p_{\Lambda}$ and since $S_{\mid\h_{\pi'}\oplus\h_{\Lambda}}$ is a basis for $(\h_{\pi'}\oplus\h_{\Lambda})^*$ there exists a unique $h\in\h_{\pi'}\oplus\h_{\Lambda}$ such that $\ad^*h(y)=-y$. Thus $(h,\,y)$ is an adapted pair for $\widetilde\p_{\Lambda}$.
Moreover  there exists, for each $\gamma\in T$, a unique $s(\gamma)\in\mathbb Q S$ such that $\gamma+s(\gamma)$ vanishes on $\h_{\pi'}\oplus\h_{\Lambda}$.

Assume furthermore  that :
\begin{enumerate}
\item[v)] For any $\gamma\in T$, $s(\gamma)\in\mathbb NS$ and that $\gamma+s(\gamma)\neq 0$. \end{enumerate}
 Then one has that :
\begin{equation}{\rm ch}\,Sy\bigl(\widetilde\p\bigr)={\rm ch}\,Y\bigl(\widetilde\p_{\Lambda}\bigr)\le\prod_{\gamma\in T}\bigl(1-e^{\gamma+s(\gamma)}\bigr)^{-1}.\label{eqgen}\end{equation}
If equality holds in the above inequality, then the restriction map gives an isomorphism
\begin{equation} Sy\bigl(\widetilde\p\bigr)=Y\bigl(\widetilde\p_{\Lambda}\bigr)\stackrel{\sim}{\longrightarrow}\Bbbk[y+V].\end{equation}
Hence $y+V$ is a Weierstrass section for $Sy\bigl(\widetilde\p\bigr)$ and $Sy\bigl(\widetilde\p\bigr)$ is a polynomial $\Bbbk$-algebra. 
\end{prop}

\begin{proof}

It is still inspired by \cite[Lem. 6.11]{J6bis}. Here we have that $\dim V={\rm index}\,\widetilde\p_{\Lambda}$ by Lemma \ref{indexcontractionlmgen} and by Proposition \ref{egindex}.
Then the rest of the proof is quite similar to  the proof of Proposition \ref{WScontractionprop}.
\end{proof}

\begin{Rq}\label{WScontractionRq}\rm
The above Proposition is obviously still valid when $\widetilde\p_{\Lambda}=\widetilde\p'$ (with $\h_{\Lambda}=\{0\}$).
Like in  \cite[Remark 6.11]{J6bis} one may observe, when equality holds  in (\ref{eqgen}), that there exists a set of homogeneous algebraically independent generators of  $Sy\bigl(\widetilde\p\bigr)=Y\bigl(\widetilde\p_{\Lambda}\bigr)$ formed by weight vectors for which we can give their weight and degree. Indeed in this case for all $\gamma\in T$, we have that $s(\gamma)\in\mathbb N S$ and setting $\lvert s(\gamma)\rvert=\sum_{\alpha\in S}q_{\alpha,\,\gamma}$ if $s(\gamma)=\sum_{\alpha\in S}q_{\alpha,\,\gamma}\alpha$ ($q_{\alpha,\,\gamma}\in\mathbb N$),
then $1+\lvert s(\gamma)\rvert$ is the degree of an homogeneous generator of weight $\gamma+s(\gamma)$ (since the map $\varphi_0$ introduced in the proof of Proposition \ref{WScontractionprop} preserves the degree).

\end{Rq}

\section{Construction of an adapted pair.}\label{Adap}

In order to apply  Proposition \ref{WScontractionpropgen}, we have to construct an element $y\in\widetilde{\p_{\Lambda}}^*$ and a vector subspace $V\subset\widetilde{\p_{\Lambda}}^*$ satisfying hypotheses $i)$, $ii)$, $iii)$, $iv)$ and $v)$ of Lemma \ref{indexcontractionlmgen} and Proposition 
\ref{WScontractionpropgen}, so that we obtain an adapted pair $(h,\,y)$ for $\widetilde{\p_{\Lambda}}=\widetilde\p_{\Lambda}$. For this purpose we will use an analogue of \cite[Lem. 3.2]{FL1}, which we will give below.

We need to recall the notion of a Heisenberg set (see for instance \cite[Sec. 3]{FL1}).
\subsection{Heisenberg sets.}

 \begin{defi}\rm
 
 Let $\Gamma\subset\Delta$. We say that $\Gamma$ is a Heisenberg set of centre $\gamma\in\Gamma$ if for all $\alpha\in\Gamma\setminus\{\gamma\}$, there exists $\alpha'\in\Gamma\setminus\{\gamma\}$, which is unique, such that $\alpha+\alpha'=\gamma$. We will denote such an $\alpha'$ by $\theta(\alpha)$. Observe that one always has that $\theta(\alpha)\neq\alpha$. A Heisenberg set $\Gamma$ with centre $\gamma$ will be denoted by $\Gamma_{\gamma}$ to emphasize that $\gamma$ is its centre. For a Heisenberg set $\Gamma_{\gamma}$ of centre $\gamma$, we will set $\Gamma_{\gamma}^0=\Gamma_{\gamma}\setminus\{\gamma\}$.
 \end{defi}
 
 Assume that there exists a set $S\subset\Delta^+\sqcup\Delta^-_{\pi'}=\Delta(\pi')$ such that all the Heisenberg sets $\Gamma_{\gamma}$, for $\gamma\in S$, are disjoint. Then set $O=\bigsqcup_{\gamma\in S}\Gamma_{\gamma}^0$. We may observe that $\theta:O\longrightarrow O$ defined above is an involution.
 
 Heisenberg sets were very useful to construct adapted pairs in the nondegenerate case, for maximal parabolic subalgebras (see \cite{FL} and \cite{FL1}) or in type A (see \cite{J5}). Here in the degenerate case, we will see that they continue to play an important role. For any subset $A\subset\Delta(\pi')$, set $\g_{A}=\bigoplus_{\alpha\in A}\g_{\alpha}$ which is a vector subspace of $\widetilde{\p_{\Lambda}}$ and $\g_{-A}=\bigoplus_{\alpha\in A}\g_{-\alpha}$ which is a vector subspace of $\widetilde{\p_{\Lambda}}^*\simeq\p^-_{\Lambda}$.
 
 \subsection{An adapted pair and a Weierstrass section.}
 
 \begin{lm}\label{Adaplm}
 Assume that there exist disjoint subsets of $\Delta(\pi')$ : $S,\,T, \Gamma_{\gamma},\,\gamma\in S,$ where for all $\gamma\in S$, $\Gamma_{\gamma}$ is a Heisenberg set with centre $\gamma$. Set $O=\bigsqcup_{\gamma\in S}\Gamma_{\gamma}^0$ and
$$y=\sum_{\gamma\in S}x_{-\gamma}\in\widetilde{\p_{\Lambda}}^*.$$ 
Denote by $\widetilde\Phi_y$ the skew-symmetric bilinear form on $\widetilde{\p_{\Lambda}}\times\widetilde{\p_{\Lambda}}$ such that
 $$\widetilde\Phi_y(x,\,x')=K(y,\,[x,\,x']_{\widetilde\p})$$ for all $x,\,x'\in\widetilde{\p_{\Lambda}}$. Assume further that :
  \begin{enumerate}
 \item[i)] $S_{\mid\h_{\pi'}\oplus\h_{\Lambda}}$ is a basis for $(\h_{\pi'}\oplus\h_{\Lambda})^*$. 
 \item[ii)] $\Delta(\pi')=\bigsqcup_{\gamma\in S}\Gamma_{\gamma}\sqcup T.$
 \item[iii)]  $\lvert T\rvert={\rm index}\,\p_{\Lambda}$.
 \item[iv)] The restriction of $\widetilde\Phi_y$ to $\g_O\times\g_O$ is nondegenerate. 
\end{enumerate}
Then  one has that $\widetilde\p_{\Lambda}=\widetilde{\p_{\Lambda}}$ and
$$\ad^*\widetilde\p_{\Lambda}(y)\oplus\g_{-T}=\widetilde\p_{\Lambda}^*.$$
In particular $y$ is regular in $\widetilde\p_{\Lambda}^*$ and if we denote by $h\in\h_{\pi'}\oplus\h_{\Lambda}$ the unique element such that $\gamma(h)=1$ for all $\gamma\in S$, then $(h,\,y)$ is an adapted pair for $\widetilde\p_{\Lambda}$. Moreover  for all $\gamma\in T$, denote by $s(\gamma)$ the unique element in $\mathbb Q S$ such that $\gamma+s(\gamma)$ vanishes on $\h_{\pi'}\oplus\h_{\Lambda}$. Then if $\gamma+s(\gamma)\neq 0$ and if $s(\gamma)\in\mathbb NS$, for all $\gamma\in T$,  one has that :
$${\rm ch}\,Sy\bigl(\widetilde\p\bigr)={\rm ch}\,Y\bigl(\widetilde\p_{\Lambda}\bigr)\le\prod_{\gamma\in T}\bigl(1-e^{\gamma+s(\gamma)}\bigr)^{-1}.$$

Finally if equality holds in the above inequality, then restriction of functions gives the algebra isomorphism
$$Sy\bigl(\widetilde\p\bigr)=Y\bigl(\widetilde\p_{\Lambda}\bigr)\stackrel{\sim}{\longrightarrow}\Bbbk[y+\g_{-T}].$$
Hence $y+\g_{-T}$ is a Weierstrass section for $Sy\bigl(\widetilde\p\bigr)$ and $Sy\bigl(\widetilde\p\bigr)$ is a polynomial $\Bbbk$-algebra. 
 \end{lm}

\begin{proof}
The proof is similar as in \cite[Lem. 3.2]{FL1}, itself inspired by \cite[8.6]{J5}. We give it for completeness. Condition $iii)$ implies that, as a vector space, $\widetilde{\p_{\Lambda}}^*=\h_{\pi'}\oplus\h_{\Lambda}\oplus\g_{-O}\oplus\g_{-S}\oplus\g_{-T}$ and that 
$\widetilde{\p_{\Lambda}}=\h_{\pi'}\oplus\h_{\Lambda}\oplus\g_{O}\oplus\g_{S}\oplus\g_{T}$.
Condition $iv)$ and equation (\ref{actioncoadcont})  imply that 
\begin{equation}\g_{-O}\subset \ad^*\g_O(y)+\g_{-S}+\g_{-T}\label{gO}\end{equation} since moreover $O\cap S=\emptyset$.
Condition $i)$ implies that $$\g_{-S}=\ad^*(\h_{\pi'}\oplus\h_{\Lambda})(y)$$ and also that
the restriction of $\widetilde\Phi_y$ to $\g_S\times(\h_{\pi'}\oplus\h_{\Lambda})$ is nondegenerate. Hence one has that
$$\h_{\pi'}\oplus\h_{\Lambda}\subset\ad^*\g_S(y)+\g_{-S}+\g_{-O}+\g_{-T}.$$
It follows that
$$\widetilde{\p_{\Lambda}}^*=\h_{\pi'}\oplus\h_{\Lambda}\oplus\g_{-O}\oplus\g_{-S}\oplus\g_{-T}\subset\ad^*\widetilde{\p_{\Lambda}}(y)+\g_{-T}.$$
Lemma \ref{indexcontractionlmgen} and Proposition \ref{WScontractionpropgen} complete the proof.
\end{proof}

Actually we will construct in the following sections, for some particular parabolic subalgebras, sets $S$, $T$ and Heisenberg sets $\Gamma_{\gamma}$, for $\gamma\in S$, satisfying the hypotheses of Lemma \ref{Adaplm}.
The hypothesis $iv)$ will be the more delicate point to verify. That is why we need a Lemma of nondegeneracy as in the following subsection \ref{Nondegen}. But firstly we need to introduce the notion of stationary roots.

\subsection{Sequences constructed from a root.}\label{Str}

In this subsection, we assume that there exist disjoint subsets $S,\,T$ of $\Delta(\pi')=\Delta^+\sqcup\Delta^-_{\pi'}$ and for every $\gamma\in S$, there exists a Heisenberg set $\Gamma_{\gamma}\subset\Delta(\pi')$ with centre $\gamma$ such that all the Heisenberg sets together with $T$ are disjoint. As in the previous subsections, we set $O=\bigsqcup_{\gamma\in S}\Gamma_{\gamma}^0$,
$y=\sum_{\gamma\in S}x_{-\gamma}\in\widetilde{\p_{\Lambda}}^*$
and we denote by $\widetilde\Phi_y$ the skew-symmetric bilinear form on $\widetilde{\p_{\Lambda}}\times\widetilde{\p_{\Lambda}}$ defined by $\widetilde\Phi_y(x,\,x')=K(y,\,[x,\,x']_{\widetilde\p})$.
Suppose that assumptions $i)$, $ii)$, $iii)$ of Lemma \ref{Adaplm} hold. We want to give a sufficient condition for $y$ to satisfy condition $iv)$ of the same Lemma.
The nondegeneracy of the restriction to $\g_O\times\g_O$ of the bilinear form $\widetilde\Phi_y$  will follow from how the roots in $O$ are arranged. 
Observe that for roots $\alpha,\,\beta\in O$, one has $K(y,\,[x_{\alpha},\,x_{\beta}]_{\widetilde\p})\neq 0$  if and only if $\alpha+\beta\in S$ and $[x_{\alpha},\,x_{\beta}]_{\widetilde\p}\neq 0$, that is, if and only if  $\alpha+\beta\in S$ and $\alpha$ and $\beta$ do not lie both in $\Delta^+\setminus\Delta^+_{\pi'}$.

It follows that we will construct in the following sections Heisenberg sets $\Gamma_{\gamma}$'s which satisfy the following condition {\bf (C)} :

{\bf Condition (C)} : 
\begin{align}\label{ConditionC}\forall\alpha\in O,\;\{\alpha,\theta(\alpha)\}\cap\Delta^+\setminus\Delta^+_{\pi'}\le 1.\end{align}

For every root $\alpha\in O$, we set $S_{\alpha}=\{\beta\in O\mid \alpha+\beta \in S\;\hbox{\rm and}\;[x_{\alpha},\,x_{\beta}]_{\widetilde\p}\neq 0\}$. With condition {\bf (C)} (Eq. \ref{ConditionC}), we  have that $\theta(\alpha)\in S_{\alpha}$. 
Set for any positive integer $n$, $O_n=\{\alpha\in O\mid \lvert S_{\alpha}\rvert=n\}$. If $\alpha\in O_1$, then the only root in $O$ lying in $S_{\alpha}$ is $\theta(\alpha)$.
Then if $O=O_1$ we have that, up to a non zero scalar, \begin{equation*}\det(\widetilde\Phi_{{y}_{\mid\g_O\times\g_O}})=\prod_{\alpha\in O}K(y,\,[x_{\alpha},\,x_{\theta(\alpha)}]_{\widetilde\p})\neq 0.\end{equation*}
Unfortunately it is not always possible to choose Heisenberg sets $\Gamma_{\gamma}$'s satisfying the condition that $O=O_1$.

Like in \cite[Sec. 6]{FL1}, we
 denote by $S^m$ the subset of $S$ consisting of those $\gamma\in S$ for which the Heisenberg set $\Gamma_{\gamma}$ contains both negative and positive roots. 
 Denote also by  $S^+$, resp. $S^-$, the subset of $S$ consisting of those $\gamma\in S$ for which the Heisenberg set $\Gamma_{\gamma}\subset\Delta^+$, resp. $\Gamma_{\gamma}\subset\Delta^-_{\pi'}$.

Set \begin{align*}\Gamma=\sqcup_{\gamma\in S}\Gamma_{\gamma},\;\Gamma^m=\sqcup_{\gamma\in S^m}\Gamma_{\gamma},\;\Gamma^+=\sqcup_{\gamma\in S^+}\Gamma_{\gamma},\;\Gamma^-=\sqcup_{\gamma\in S^-}\Gamma_{\gamma}\\
O=\sqcup_{\gamma\in S}\Gamma_{\gamma}^0,\;O^m=\sqcup_{\gamma\in S^m}\Gamma_{\gamma}^0,\;O^+=\sqcup_{\gamma\in S^+}\Gamma_{\gamma}^0,\;O^-=\sqcup_{\gamma\in S^-}\Gamma_{\gamma}^0
\end{align*}
We have $S=S^m\sqcup S^+\sqcup S^-$, $\Gamma=\Gamma^m\sqcup\Gamma^+\sqcup\Gamma^-$ and $O=O^m\sqcup O^+\sqcup O^-$. 

We will add the following condition {\bf (C')} :

{\bf Condition (C')} : 
\begin{align}\label{ConditionC'}\begin{cases}
 O=O_1\sqcup O_2\sqcup O_3\\
\alpha\in O_3\Longrightarrow\exists\widetilde\alpha\in S_{\alpha}\cap O_2\setminus\{\theta(\alpha)\};\;\theta(\widetilde\alpha)\in O_1\\
\end{cases}\end{align}

\begin{nota}\label{notaseq}\rm
If $\alpha\in O_2$, we will set $\widetilde\alpha=\theta(\alpha)$.
\end{nota}

From now on, we assume that we have constructed disjoint Heisenberg sets $\Gamma_{\gamma}$ with centre $\gamma$ in $\Delta(\pi')$ satisfying conditions {\bf (C)} and {\bf (C')} (Eq. \ref{ConditionC} and Eq. \ref{ConditionC'}).
As in \cite[Sec. 4]{FL1} we define below, for any $\alpha\in O$, the sequences $(\alpha^i)_{i\in\mathbb N}$ and $(\alpha^{(i)})_{i\in\mathbb N}$ constructed from $\alpha$ and $\theta(\alpha)$ respectively. 

\begin{defi}\label{seq}\rm
 Let $\alpha\in O$. We define the sequence $(\alpha^{(n)})_{n\in\mathbb N}$ of roots in $O$ constructed from the root $\theta(\alpha)$ as follows. 

 \begin{enumerate}
 \item Set $\alpha^{(0)}=\theta(\alpha)$.
\item If $\alpha\in O_1$, set $\alpha^{(1)}=\theta(\alpha)=\alpha^{(0)}$.
\item If $\alpha\in O_2$, let $\alpha^{(1)}$ be the root in $S_{\alpha}\setminus\{\theta(\alpha)\}$, namely $\alpha^{(1)}+\alpha\in S$, $\alpha^{(1)}\neq\theta(\alpha)$, $\alpha$ and $\alpha^{(1)}$ do not lie both in $\Delta^+\setminus\Delta^+_{\pi'}$.
\item If $\alpha\in O_3$, then we choose $\alpha^{(1)}\in S_{\alpha}\setminus\{\theta(\alpha),\,\widetilde{\alpha}\}$.
\item Assume that $\alpha\in O_2\sqcup O_3$. Then  $\alpha^{(1)}\in S_{\alpha}\setminus\{\widetilde\alpha,\,\theta(\alpha)\}$ (in view of notation \ref{notaseq}). Then observe that $\alpha^{(1)}\not\in O_1$ since otherwise we should have that $\theta(\alpha^{(1)})=\alpha=\theta(\alpha^{(0)})$ and then that $\alpha^{(1)}=\alpha^{(0)}=\theta(\alpha)$, which is not possible by the previous construction. If moreover $\alpha^{(1)}\in O_2$, then necessarily we have that $(\alpha^{(1)})^{(1)}=\alpha$. Now if $\alpha^{(1)}\in O_3$, we have again that $\alpha=(\alpha^{(1)})^{(1)}$ provided that  $\alpha\neq\widetilde{\alpha^{(1)}}$. 
In both cases, we then  have that $\alpha=(\alpha^{(1)})^{(1)}$ provided that $\alpha\neq\widetilde{\alpha^{(1)}}$. Observe that, if $\alpha^{(1)}\in O_2$, then the condition that $\alpha\neq\widetilde{\alpha^{(1)}}$ is always true since in this case we have set $\widetilde{\alpha^{(1)}}=\theta(\alpha^{(1)})$.
\item This defines inductively a sequence $(\alpha^{(n)})_{n\in\mathbb N}$ of roots in $O$
 such that for any $n\in\mathbb N$ :
 \begin{enumerate}
\item if $\theta(\alpha^{(n)})\in O_1$, then $\alpha^{(n+1)}=\alpha^{(n)}$
\item\label{eqbseq}  if $\theta(\alpha^{(n)})\in O_2\sqcup O_3$, then $\alpha^{(n+1)}\in S_{\theta(\alpha^{(n)})}\setminus\{\alpha^{(n)},\,\widetilde{\theta(\alpha^{(n)})}\}$.
\item\label{eqseq} Let $n\in\mathbb N^*$ such that $\theta(\alpha^{(n-1)})\in O_2\sqcup O_3$. Then $\alpha^{(n)}\not\in O_1$ and $(\alpha^{(n)})^{(1)}=\theta(\alpha^{(n-1)})$, 
provided that $\theta(\alpha^{(n-1)})\neq\widetilde{\alpha^{(n)}}$.

\end{enumerate}
\end{enumerate}
\end{defi}

Similarly we define the sequence $(\alpha^n)_{n\in\mathbb N}$ of roots in $O$ constructed from $\alpha$, as follows.
\begin{defi}\label{seqbis}\rm
Let $\alpha\in O$. We define inductively the sequence $(\alpha^{n})_{n\in\mathbb N}$ of roots in $O$ constructed from the root $\alpha$ as follows. 
 \begin{enumerate}
 \item $\alpha^{0}=\alpha$.
\item For any $n\in\mathbb N$,
\begin{enumerate}
\item if $\theta(\alpha^n)\in O_1$, then $\alpha^{n+1}=\alpha^{n}$
\item\label{eq1bisseq} if $\theta(\alpha^n)\in O_2\sqcup O_3$, then $\alpha^{n+1}\in S_{\theta(\alpha^n)}\setminus\{\alpha^n,\,\widetilde{\theta(\alpha^n)}\}$.
\item\label{eqbisseq} Let $n\in\mathbb N^*$ such that $\theta(\alpha^{n-1})\in O_2\sqcup O_3$. Then $\alpha^{n}\not\in O_1$ and $(\alpha^{n})^{(1)}=\theta(\alpha^{n-1})$, provided that
$\theta(\alpha^{n-1})\neq\widetilde{\alpha^{n}}$.
\end{enumerate}
\end{enumerate}
\end{defi}

\begin{Rqs}\label{Rkseq}\rm
Let  $i$ be any nonnegative integer and $\alpha\in O$.
\begin{enumerate}
\item\label{eqRkseq} We have that $\alpha^{(i)}=\theta(\alpha)^i$.
\item\label{eq2Rkseq} $(\alpha^i)^{1}=\alpha^{i+1}$ and $(\alpha^{(i)})^1=\alpha^{(i+1)}$.
\end{enumerate}
\end{Rqs}

\begin{defi}\rm
Let $\alpha\in O$, for which there exists no $\beta\in O_3$ such that $\widetilde\beta=\alpha$ or $\theta\bigl(\widetilde\beta\bigr)=\alpha$.
\begin{enumerate} 
\item If there exists $n\in\mathbb N$ such that $\alpha^{(n+1)}=\alpha^{(n)}$, we say that the sequence $(\alpha^{(i)})_{i\in\mathbb N}$ is {\bf stationary} and the {\bf rank} of this sequence is  the minimal of such an $n$.
\item Similarly for saying that the sequence $(\alpha^{i})_{i\in\mathbb N}$ is {\bf stationary} and for the {\bf rank} of this stationary sequence. 
\item We say that $\alpha$ is a {\bf stationary root} if both sequences $(\alpha^{(i)})_{i\in\mathbb N}$ and $(\alpha^{i})_{i\in\mathbb N}$ are stationary (see Figure \ref{figstat}).
\end{enumerate}
\end{defi}

\begin{Rq}\rm
The definition of a stationary root will be very useful. Recall that, for any $\alpha\in\Delta$, the vector $x_{\alpha}\in\g_{\alpha}\setminus\{0\}$ is a priori fixed, but we will possibly rescale some of these vectors, except for the vectors $x_{-\gamma}$ with $\gamma\in S$ since $y=\sum_{\gamma\in S}x_{-\gamma}$ is fixed. Take a root $\alpha\in O$ and fix the nonzero root vector $x_{-\alpha}$. Assume that the sequence $(\alpha^i)_{i\in\mathbb N}$ is stationary at rank $n_0$. Set $\gamma_{i_0}=\alpha+\theta(\alpha)\in S$, $\gamma_{i_1}=\alpha^1+\theta(\alpha)\in S$. 
If $n_0=0$ that is, if $\theta(\alpha)\in O_1$, then we have that, up to rescaling the vector $x_{\theta(\alpha)}$
\begin{equation*}\ad^*x_{\theta(\alpha)}(y)=\ad^*x_{\theta(\alpha)}(x_{-\gamma_{i_0}})=x_{-\alpha}\mod \g_{-S}\oplus\g_{-T}.\end{equation*} 

Assume from now on that $n_0\ge 1$. Then $\alpha^1\neq\alpha^0$. Since condition $({\bf C'})$ is assumed, two cases may occur :
\begin{enumerate}
\item If $\theta(\alpha)\in O_2$ we have that, up to rescaling the vectors $x_{\theta(\alpha)}$ and $x_{-\alpha^1}$ 
\begin{equation*}\ad^*x_{\theta(\alpha)}(y)=\ad^*x_{\theta(\alpha)}(x_{-\gamma_{i_0}}+x_{-\gamma_{i_1}})=x_{-\alpha}+x_{-\alpha^1}\mod \g_{-S}\oplus\g_{-T}.\end{equation*}
\item If $\theta(\alpha)\in O_3$ we have that, up to rescaling the vectors $x_{\theta(\alpha)}$, $x_{-\alpha^1}$ and $x_{-\widetilde{\theta(\alpha)}}$ 
\begin{equation*}\ad^*x_{\theta(\alpha)}(y)=x_{-\alpha}+x_{-\alpha^1}+x_{-\widetilde{\theta(\alpha)}}\mod \g_{-S}\oplus\g_{-T}.\end{equation*} 
 Since $\theta\bigl(\widetilde{\theta(\alpha)}\bigr)\in O_1$ we have that, up to rescaling the vector $x_{\theta\bigl(\widetilde{\theta(\alpha)}\bigr)}$
 \begin{equation*}\ad^* x_{\theta\bigl(\widetilde{\theta(\alpha)}\bigr)}(y)=x_{-\widetilde{\theta(\alpha)}}\mod\g_{-S}\oplus\g_{-T}.\end{equation*} 
 \item We may continue until we obtain the root $\alpha^{n_0}$ such that $\theta(\alpha^{n_0})\in O_1$, namely until we obtain that, up to rescaling the vector $x_{\theta(\alpha^{n_0})}$ 
\begin{equation*}\ad^*x_{\theta(\alpha^{n_0})}(y)=x_{-\alpha^{n_0}}\mod \g_{-S}\oplus\g_{-T}.\end{equation*} 
\item We obtain finally that 
 \begin{equation*}x_{-\alpha}\in\ad^*\g_O(y)+\g_{-S}+\g_{-T}\end{equation*} 
 \end{enumerate}
 Thus we obtain Eq. (\ref{gO}) in Proof of Lemma \ref{Adaplm}.
The relevance of stationary roots will be also given by Lemma \ref{statlm} and Prop. \ref{condnondegen}.
\end{Rq}

\begin{nota}\rm
For a root $\alpha\in O$, for which there exists no $\beta\in O_3$ such that $\widetilde\beta=\alpha$ or $\theta\bigl(\widetilde\beta\bigr)=\alpha$, set $C_{\alpha}:=\{\alpha^i,\,\theta(\alpha^i),\,\alpha^{(i)},\,\theta(\alpha^{(i)}); i\in\mathbb N\}$ and call it {\bf the chain passing through $\alpha$}. Denote by $\widetilde{C_{\alpha}}$ the root subsystem of $O$ consisting of the roots $\widetilde\beta_i$ and $\theta\bigl(\widetilde\beta_i\bigr)$ whenever $\beta_i\in O_3\cap C_{\alpha}$.
Then $C_{\alpha}\sqcup\widetilde{C_{\alpha}}$ (which is finite since the root system $\Delta$ is finite) can be represented by a graph, with vertices being the elements in $C_{\alpha}\sqcup\widetilde{C_{\alpha}}$ and edges between two roots in $C_{\alpha}\sqcup\widetilde{C_{\alpha}}$ which are linked if their sum  belongs to $S$ and if they do not lie both in $\Delta^+\setminus\Delta^+_{\pi'}$. In particular if one vertex (say $\beta_i$) in $C_{\alpha}$ belongs to $O_3$ then we draw an edge between $\beta_i$ and $\widetilde\beta_i\in S_{\beta_i}\setminus\{\beta_i^{(1)},\,\theta(\beta_i)\}$
and another edge between $\widetilde\beta_i$ and $\theta\bigl(\widetilde\beta_i\bigr)$. 
\end{nota}


\vskip 0.5cm

\begin{figure}[!h]
\centering

\begin{tikzpicture}
[scale=.5]
\draw[gray, thick] (0,0) -- (2,2);
\draw[gray, thick] (2,2) -- (5,2);
\draw[gray, thick] (5,2) -- (8,2);
\draw[gray, thick] (8,2) -- (11,2);
\draw[gray, thick] (0,0) -- (2,-2);
\draw[gray, thick] (2,-2) -- (5,-2);
\draw[gray, thick] (5,-2) -- (8,-2);
\draw[gray, thick] (5,-2) -- (5,-4);
\draw[gray,thick] (5,-4) -- (5,-6);
\filldraw[black] (2,2) circle (2pt) node[anchor=north]{$\alpha^{(1)}$};
\filldraw[black] (5,2) circle (2pt) node[anchor=north]{$\theta(\alpha^{(1)})$};
\filldraw[black] (0,0) circle (2pt) node[anchor=east]{$\alpha$};
\filldraw[black] (8,2) circle (2pt) node[anchor=south]{$\alpha^{(2)}=\alpha^{(3)}$};
\filldraw[black] (2,-2) circle (2pt) node[anchor=south]{$\theta(\alpha)$};
\filldraw[black](5,-2) circle (2pt) node[anchor=south]{$\alpha^1=\alpha^2$};
\filldraw[black] (8,-2) circle (2pt) node[anchor=south]{$\theta(\alpha^1)$};
\filldraw[black] (11,2) circle (2pt) node[anchor=north]{$\theta(\alpha^{(2)})$};
\filldraw[black] (5,-4) circle (2pt) node[anchor=west]{$\widetilde{\alpha^1}$};
\filldraw[black] (5,-6) circle (2pt) node[anchor=west]{$\theta\Bigl(\widetilde{\alpha^1}\Bigr)$};
\end{tikzpicture}

\caption{The graph $C_{\alpha}\sqcup\widetilde{C_{\alpha}}$ for a stationary root $\alpha$ with the sequence $(\alpha^i)_{i\in\mathbb N}$ stationary  at rank $1$ and the sequence $(\alpha^{(i)})_{i\in\mathbb N}$ stationary at rank $2$ and with $\alpha^1\in O_3$.}
\label{figstat}
\end{figure}
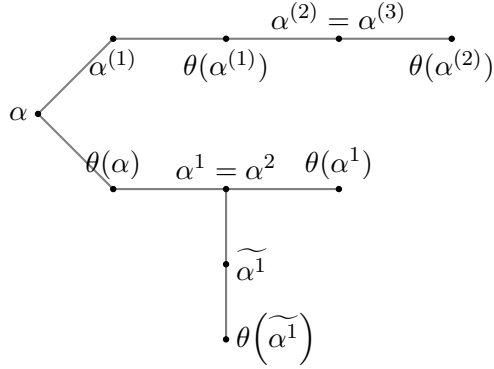

\begin{figure}[!h]

\centering
\begin{tikzpicture}
[scale=0.5]
\draw[gray, thick] (0,0) --(2,2);
\draw[gray, thick] (2,2) -- (5,2);
\draw[gray, thick] (5,2) -- (7,0);
\draw[gray, thick] (0,0) -- (2,-2);
\draw[gray, thick] (2,-2)-- (5,-2);
\draw[gray, thick] (5,-2) -- (7,0);
\draw[gray, thick] (5,-2) -- (5,-4);
\draw[gray, thick] (5,-4) -- (5,-6);
\filldraw[black] (0,0) circle (2pt) node[anchor=east]{$\alpha=\theta(\alpha^{(3)})=\alpha^3$};
\filldraw[black] (2,2) circle (2pt) node[anchor=east]{$\alpha^{(1)}=\theta(\alpha^2)=\alpha^{(4)}$};
\filldraw[black] (5,2) circle (2pt) node[anchor=west]{$\theta(\alpha^{(1)})=\alpha^2$};
\filldraw[black] (7,0) circle (2pt) node[anchor=west]{$\alpha^{(2)}=\theta(\alpha^1)$};
\filldraw[black] (2,-2) circle (2pt) node[anchor=east]{$\theta(\alpha)=\alpha^{(3)}$};
\filldraw[black](5,-2) circle (2pt) node[anchor=west]{$\alpha^1=\theta(\alpha^{(2)})$};
\filldraw[black] (5,-4) circle (2pt) node[anchor=west]{$\widetilde{\alpha^1}$};
\filldraw[black] (5,-6) circle (2pt) node[anchor=west]{$\theta\Bigl(\widetilde{\alpha^1}\Bigr)$};
\end{tikzpicture}

\caption{The graph $C_{\alpha}\sqcup\widetilde{C_{\alpha}}$ for a root $\alpha$ which is not stationary, with $\alpha^1\in O_3$.}
\label{fignstat}
\end{figure}
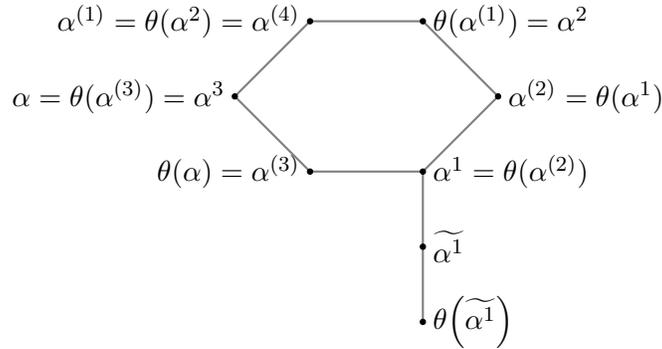

We will show actually that only one of the two situations illustrated  in Figures \ref{figstat} or \ref{fignstat} holds.

\begin{lm}\label{proproots}
Let $\alpha\in O$, for which there exists no $\beta\in O_3$ such that $\alpha=\widetilde\beta$ or $\alpha=\theta\bigl(\widetilde\beta\bigr)$.
\begin{enumerate}
\item\label{prop1roots} Any root $\gamma$ in the chain $C_{\alpha}$ is also such that there exists no $\beta\in O_3$ such that $\gamma=\widetilde\beta$ or $\gamma=\theta\bigl(\widetilde\beta\bigr)$.
\item\label{prop2roots} If one of the two sequences $(\alpha^i)_{i\in\mathbb N}$ or $(\alpha^{(i)})_{i\in\mathbb N}$ is stationary, then the other one is also stationary and then $\alpha$ is a stationary root.
\item\label{prop3roots} If $\alpha$ is a stationary root, then any root in the chain $C_{\alpha}$  is also a stationary root. 
\item\label{prop4roots} If $\alpha$ is not a stationary root, then there exists $j\in\mathbb N$, $j>1$, such that $\alpha=\alpha^j=\theta(\alpha^{(j)})$. 
\item\label{prop5roots} If there exists an integer $j$, $j>1$, such that $\alpha=\alpha^j$ and $\alpha\neq\alpha^{j-1}$, then one has that $\alpha=\alpha^j=\theta(\alpha^{(j)})$ and the root $\alpha$ is not stationary. Such a root $\alpha$ is called {\bf a cyclic root}. 
\item\label{prop6roots} If $\alpha$ is a cyclic root then any root in the chain $C_{\alpha}$ is also a cyclic root.
\end{enumerate}

\end{lm}

\begin{proof}

\begin{enumerate}

\item Let $\alpha\in O$ such that there exists no $\beta\in O_3$ such that $\alpha=\widetilde\beta$ or $\alpha=\theta\bigl(\widetilde\beta\bigr)$. Then it is obviously the same for $\theta(\alpha)$.
Now suppose that, for some $k\in\mathbb N$, we have that, for any $i\in\mathbb N$, $i\le k$, $\alpha^i$ and $\alpha^{(i)}$ are not equal to $\widetilde\beta$ and to $\theta(\widetilde\beta)$ for any $\beta\in O_3$. Let $\gamma=\alpha^{k+1}$. Suppose firstly that there exists $\beta\in O_3$ such that $\gamma=\widetilde\beta$. Then $\gamma\in O_2$ and $\gamma^{(1)}=\beta=(\alpha^{k+1})^{(1)}=\theta(\alpha^k)$ by \ref{eqbisseq} of Definition \ref{seqbis} and the induction hypothesis. But now $\widetilde\beta=\widetilde{\theta(\alpha^k)}=\alpha^{k+1}$ which is impossible by \ref{eq1bisseq} of Definition \ref{seqbis}. Suppose now that there exists $\beta\in O_3$ such that $\gamma=\alpha^{k+1}=\theta\bigl(\widetilde\beta\bigr)$. Then $\gamma=\alpha^{k+1}\in O_1$ which implies that $\theta(\alpha^{k+1})=\theta(\alpha^k)$ and then that $\alpha^{k+1}=\alpha^k=\theta(\widetilde\beta)$, which is impossible by the induction hypothesis. Obviously this also implies that $\theta(\alpha^{k+1})$ is not equal to $\widetilde\beta$ and to $\theta\bigl(\widetilde\beta\bigr)$ for any $\beta\in O_3$. A similar argument applies to show that $\alpha^{(k+1)}$ (and then also $\theta(\alpha^{(k+1)})$) is not equal to $\widetilde\beta$ and to $\theta\bigl(\widetilde\beta\bigr)$ for any $\beta\in O_3$. This proves the first assertion of our lemma.

\item Assume that the sequence $(\alpha^i)_{i\in\mathbb N}$ is stationary at rank $n_0$. Then we will show that it is not possible that, for any $i\in\mathbb N$, $\alpha^{(i)}\neq\alpha^{(i+1)}$. Indeed we will show that the hypothesis 
\begin{equation*}\forall i\in\mathbb N,\;\alpha^{(i)}\neq\alpha^{(i+1)}\tag{H}\label{Hyp}\end{equation*}
implies that, for any $n\in\mathbb N$, the cardinality of the set $\{\alpha^{(k)};\;0\le k\le n\}$ is equal to $n+1$, which is impossible. The cases $n=0$ and $n=1$ are immediate. Assume that, for some $n-1\in\mathbb N^*$, the cardinality of the set $\{\alpha^{(k)};\;0\le k\le n-1\}$ is equal to $n$. Then if $\alpha^{(n)}=\alpha^{(k)}$ for some $1\le k\le n-1$,  it implies that $(\alpha^{(n)})^{(1)}=(\alpha^{(k)})^{(1)}$.  This gives that $\theta(\alpha^{(n-1)})=\theta(\alpha^{(k-1)})$ by hypothesis (\ref{Hyp}) and \ref{eqseq} of Definition \ref{seq}, in view of our hypothesis for $\alpha$ and the first assertion of the lemma. We deduce that $\alpha^{(n-1)}=\alpha^{(k-1)}$ which is a contradiction with the induction hypothesis. It remains to show that we cannot have that $\alpha^{(n)}=\alpha^{(0)}$. Suppose the opposite. Divide $n_0$ by $n$ : $n_0=nq+r$ with $0\le r\le n-1$, $q ,\,r\in\mathbb N$.
We have :
\begin{align*}\alpha^{(n)}=\alpha^{(0)}=\theta(\alpha)\Longrightarrow \alpha=\theta(\alpha^{(n)})\\
\Longrightarrow \alpha^1=\theta(\alpha^{(n)})^1=(\alpha^{(n)})^{(1)}=\theta(\alpha^{(n-1)})
\end{align*}
by \ref{eqRkseq} of Remark \ref{Rkseq} and by \ref{eqseq} of Definition \ref{seq} in view of hypothesis (\ref{Hyp}). 

Using \ref{eq2Rkseq} of Remark \ref{Rkseq} and \ref{eqseq} of Definition \ref{seq} in view of hypothesis (\ref{Hyp}), an induction on $k$ gives that \begin{equation*}\forall k\in\mathbb N,\;0\le k\le n,\;\alpha^k=\theta(\alpha^{(n-k)})\;\tag{$\star$}\label{eqseqlm}.\end{equation*}
In particular we have that $\alpha^{n}=\theta(\alpha^{(0)})=\alpha$.
Then using \ref{eq2Rkseq} of Remark \ref{Rkseq} one obtains that, for any $k\in\mathbb N$, one has $\alpha^{kn}=\alpha$ and especially that 
\begin{equation*}\alpha^{nq}=\alpha.\end{equation*}
Using again \ref{eq2Rkseq} of Remark \ref{Rkseq} and also (\ref{eqseqlm}) with $k=r$ and also with $k=r+1$, one obtains that
\begin{align*}\alpha^{n_0}=\alpha^{nq+r}= \alpha^r=\theta(\alpha^{(n-r)})\end{align*}
and \begin{align*}\alpha^{n_0+1}=\alpha^{nq+r+1}= \alpha^{r+1}=\theta(\alpha^{(n-r-1)}).\end{align*}
Since $\alpha^{n_0+1}=\alpha^{n_0}$, one deduces that $\alpha^{(n-r)}=\alpha^{(n-r-1)}$, which contradicts hypothesis  (\ref{Hyp}). It follows that hypothesis  (\ref{Hyp}) and the fact that the sequence $(\alpha^i)_{i\in\mathbb N}$ is stationary imply that, for any $n\in\mathbb N$, the cardinality of the set $\{\alpha^{(k)};\;0\le k\le n\}$ is equal to $n+1$, which  contradicts the finiteness of the root system of $\g$. Then if the sequence $(\alpha^i)_{i\in\mathbb N}$ is stationary, we cannot have hypothesis (\ref{Hyp}) and then  the sequence $(\alpha^{(i)})_{i\in\mathbb N}$ is also stationary. Exchanging $\alpha$ with $\theta(\alpha)$ gives the reverse implication.

\item Assume that $(\alpha^i)_{i\in\mathbb N}$ is stationary at rank $n_0$ and $(\alpha^{(i)})_{i\in\mathbb N}$ is stationary at rank $n_1$ and let $\gamma\in C_{\alpha}$. By the first assertion of the lemma, we already know that there exist no $\beta\in O_3$ such that $\gamma=\widetilde\beta$ or $\gamma=\theta(\widetilde\beta)$.
Assume firstly that $\gamma=\alpha^i$ for $0\le i\le n_0$. Then by \ref{eq2Rkseq} of Remark \ref{Rkseq} we have that, for any $j\in\mathbb N$, $$\gamma^j=\alpha^{i+j}.$$
 It follows  that $\gamma^{n_0+1-i}=\gamma^{n_0-i}$ hence that $\gamma$ is stationary.
 Assume that $\gamma=\alpha^{(i)}$ for some $0\le i\le n_1$, then by \ref{eq2Rkseq} of Remark \ref{Rkseq}  we have that $$\gamma^{j}=\alpha^{(i+j)}$$ for any $j\in\mathbb N$. It follows  that $\gamma^{n_1+1-i}=\gamma^{n_1-i}$ hence that $\gamma$ is stationary. 
 If $\gamma=\theta(\alpha^i)$ for some $0\le i\le n_0$, then by \ref{eqRkseq} of Remark \ref{Rkseq} and the above we have that for any $j\in\mathbb N$, $$\gamma^{(j)}=\alpha^{i+j}.$$ It follows that $\gamma^{(n_0+1-i)}=\gamma^{(n_0-i)}$ hence that $\gamma$ is stationary. 
 Finally if $\gamma=\theta(\alpha^{(i)})$ for some $0\le i\le n_1$, then by \ref{eqRkseq} of Remark \ref{Rkseq} and the above we have that for any $j\in\mathbb N$, $$\gamma^{(j)}=\alpha^{(i+j)}.$$ It follows that $\gamma^{(n_1+1-i)}=\gamma^{(n_1-i)}$ hence that $\gamma$ is stationary.

\item Assume that the sequence $(\alpha^i)_{i\in\mathbb N}$ and the sequence $(\alpha^{(i)})_{i\in\mathbb N}$ are not stationary. Because of the finiteness of the root system, there exists a positive integer $n$ such that the cardinality of the set $\{\alpha^{(k)};\;0\le k\le n\}$ is strictly smaller than $n+1$ and then there exist integers $k,\, k'$, $0\le k\neq k'\le n$, such that $\alpha^{(k)}=\alpha^{(k')}$. Using \ref{eqseq} of Definition \ref{seq},  in view of the first assertion,  it follows that there exists an integer $j>1$ such that $\alpha^{(0)}=\alpha^{(j)}$ and then that $\alpha=\theta(\alpha^{(j)})$. Finally using \ref{eqRkseq} of Remark \ref{Rkseq} and \ref{eqseq} of Definition \ref{seq} gives that
\begin{equation*} \alpha^j=\theta(\alpha^{(0)})=\alpha\end{equation*} which gives the required equality.

\item Now assume that there exists an integer $j>1$ such that $\alpha=\alpha^j$ and $\alpha\neq\alpha^{j-1}$. Using \ref{eqbisseq} of Definition \ref{seqbis} and \ref{eqRkseq} and \ref{eq2Rkseq} of Remark \ref{Rkseq},  in view of the first assertion, one obtains that
\begin{equation*}\forall 0\le k\le j,\, \alpha^{(k)}=\theta(\alpha^{j-k}).\label{eq3seq}\end{equation*}

Taking $k=j$, we obtain $\alpha^{(j)}=\theta(\alpha^0)=\theta(\alpha)$ and then $\alpha=\alpha^j=\theta(\alpha^{(j)})$.

Moreover a similar argument as in (2) gives that
\begin{equation*}\alpha=\alpha^j\Longrightarrow\forall k\in\mathbb N,\, \alpha^{jk}=\alpha\end{equation*}

Assume that there exists $n\in\mathbb N$ such that $\alpha^{n+1}=\alpha^n$.  Dividing $n$ by $j$ gives $n=jq+r$, with $q,\,r\in\mathbb N$, $0\le r\le j-1$, and we have :
\begin{align*}\begin{cases}\alpha=\alpha^j\\
\alpha^{n+1}=\alpha^n\end{cases}\Longrightarrow \alpha^{n}=\alpha^{jq+r}=\alpha^r=\alpha^{jq+r+1}=\alpha^{r+1}\\
\Longrightarrow\alpha^r=\alpha^{r+1}=\ldots=\alpha^{j-1}=\alpha^j=\alpha\end{align*}
which contradicts the hypothesis that $\alpha\neq\alpha^{j-1}$. Then the hypothesis that there exists an integer $j>1$ such that $\alpha=\alpha^j$ and $\alpha\neq\alpha^{j-1}$ implies that the sequence $(\alpha^i)_{i\in\mathbb N}$ is not stationary and then that the root $\alpha$ is not stationary.

\item Assume that $\alpha$ is a cyclic root and take $\gamma\in C_{\alpha}$. Then we obtain that $\alpha\in C_{\gamma}$ (by similar arguments as before). Now if $\gamma$ is not a cyclic root, then by the previous assertions, $\gamma$ is a stationary root. Then by (3) it follows that $\alpha$ is also stationary, which contradicts (5).
\end{enumerate}\end{proof}

\begin{Rq}\rm

The notion of a cyclic root given here is more general than this given in \cite[Sec. 5]{FL1}. In particular in our present paper, if $\alpha$ is a cyclic root, then the restriction of $\widetilde\Phi_y$ to $\g_{C_{\alpha}}\times\g_{C_{\alpha}}$ need not be nondegenerate.
\end{Rq}

As in \cite[Lem. 4.4]{FL1}, we have the following Lemma.

\begin{lm}\label{statlm}
Let $\alpha\in O$, for which there exists no root $\beta\in O_3$ such that $\alpha=\widetilde\beta$ or $\alpha=\theta\bigl(\widetilde\beta\bigr)$. Assume that $\alpha$ is a stationary root. Let $\vartheta: O\longrightarrow O$ be a permutation such that, for all $\gamma\in O$, one has $\gamma+\vartheta(\gamma)\in S$. Then one has that 
$\vartheta_{\mid C_{\alpha}\sqcup\widetilde{C_{\alpha}}}=\theta_{\mid C_{\alpha}\sqcup\widetilde{C_{\alpha}}}$.
\end{lm}
We give the proof below for the reader's convenience.
\begin{proof}
Denote by $n_0$, resp. $n_1$, the rank of the sequence $(\alpha^i)_{i\in\mathbb N}$, resp. $(\alpha^{(i)})_{i\in\mathbb N}$. Then necessarily we have that $\vartheta(\theta(\alpha^{n_0}))=\alpha^{n_0}$ and $\vartheta(\theta(\alpha^{(n_1)}))=\alpha^{(n_1)}$, since $\theta(\alpha^{n_0})\in O_1$ and $\theta(\alpha^{(n_1)})\in O_1$. Now consider the root $\beta=\theta(\alpha^{n_0-1})$ (if $n_0\ge 1$). Then the image of $\beta$  by $\vartheta$ can be a priori equal to three (if $\beta\in O_3$) resp. two (if $\beta\in O_2$) possible roots :
$\alpha^{n_0-1}$ or $\alpha^{n_0}$ or $\widetilde{\beta}\not\in\{\alpha^{n_0-1},\,\alpha^{n_0}\}$ (the last one only if $\beta\in O_3$). But since $\theta\bigl(\widetilde{\beta}\bigr)\in O_1$, we have necessarily that $\vartheta\bigl(\theta\bigl(\widetilde{\beta}\bigr)\bigr)=\widetilde{\beta}$. Since $\vartheta$ is a bijection, it follows that necessarily $\vartheta(\theta(\alpha^{n_0-1}))=\alpha^{n_0-1}$. A same argument holds for showing that $\vartheta(\theta(\alpha^{(n_1-1)}))=\alpha^{(n_1-1)}$ (if $n_1\ge 1$). By a decreasing induction on $k$ it follows that, for any $0\le k\le n_0$, we have $\vartheta(\theta(\alpha^k))=\alpha^k$ and that, for any $0\le k\le n_1$, we have $\vartheta(\theta(\alpha^{(k)}))=\alpha^{(k)}$. Taking $k=0$, it follows that $\vartheta(\alpha)=\vartheta(\theta(\alpha^{(0)}))=\alpha^{(0)}=\theta(\alpha)$ and $\vartheta(\alpha^{(0)})=\vartheta(\theta(\alpha))=\alpha=\theta(\alpha^{(0)})$. Then an increasing induction on $k$ gives that $\vartheta(\alpha^k)=\theta(\alpha^k)$  for any $0\le k\le n_0$ and that $\vartheta(\alpha^{(k)})=\theta(\alpha^{(k)})$ for any $0\le k\le n_1$. Finally let $\beta\in O_3\cap C_{\alpha}$. Then a priori, since $\widetilde\beta\in O_2$ by condition (C'), we have that $\vartheta\bigl(\widetilde\beta\bigr)=\theta\bigl(\widetilde\beta\bigr)$ or $\vartheta\bigl(\widetilde\beta\bigr)=\beta$. By what we have showed before we have that $\beta=\vartheta\bigl(\theta(\beta)\bigr)$. Since $\widetilde\beta\neq\theta(\beta)$, it follows that $\vartheta\bigl(\widetilde\beta\bigr)=\theta\bigl(\widetilde\beta\bigr)$. And of course we have that $\vartheta\bigl(\theta\bigl(\widetilde\beta\bigr)\bigr)=\widetilde\beta$ since $\theta\bigl(\widetilde\beta\bigr)\in O_1$. This completes the proof.
\end{proof}

\begin{lm}

Let $\alpha\in O$, for which there exists no root $\beta\in O_3$ such that $\alpha=\widetilde\beta$ or $\alpha=\theta\bigl(\widetilde\beta\bigr)$. Assume that $\alpha$ is a cyclic root. Let $\vartheta: O\longrightarrow O$ be a permutation such that, for all $\gamma\in O$, one has $\gamma+\vartheta(\gamma)\in S$. Then $\vartheta_{\mid\widetilde{C_{\alpha}}}=\theta_{\mid\widetilde{C_{\alpha}}}$.
\end{lm}

\begin{proof}

Let $j\in\mathbb N$, $j>1$, such that $\alpha=\alpha^j=\theta(\alpha^{(j)})\neq\alpha^{j-1}$.
Let $\gamma\in C_{\alpha}\cap O_3$. Since $\theta(\widetilde\gamma)\in O_1$ necessarily we have that $\vartheta(\theta(\widetilde\gamma))=\widetilde\gamma$ since $\widetilde\gamma$ is the only root in $O$ such that $\widetilde\gamma+\theta(\widetilde\gamma)\in S$. Now for $\vartheta(\widetilde\gamma)$ there are two possibilities (since $\widetilde\gamma\in O_2$) either $\vartheta(\widetilde\gamma)=\theta(\widetilde\gamma)$ or $\vartheta(\widetilde\gamma)=\gamma$ because $\widetilde\gamma\in S_{\gamma}\setminus\{\theta(\gamma),\,\gamma^{(1)}\}$. Suppose that $\vartheta(\widetilde\gamma)=\gamma$ and that $\gamma=\alpha^{(i)}$ for some $i\in\mathbb N$, $0\le i\le j-1$. Then necessarily $\vartheta(\theta(\gamma))=\theta(\gamma)^{(1)}=\gamma^1=(\alpha^{(i)})^1=\alpha^{(i+1)}$ by \ref{eqRkseq} and \ref{eq2Rkseq} of Remark \ref{Rkseq}. An induction gives that necessarily $\vartheta(\theta(\alpha^{(j-1)}))=\alpha^{(j)}$. Then $\vartheta(\alpha)=\vartheta(\theta(\alpha^{(j)}))=\alpha^{(1)}$ and $\vartheta(\theta(\alpha^{(1)}))=\theta(\alpha^{(1)})^{(1)}=(\alpha^{(1)})^1=\alpha^{(2)}$. By induction we obtain that, if $i\ge 1$, $\vartheta(\theta(\alpha^{(i-1)}))=\alpha^{(i)}=\gamma$. This implies, since $\vartheta$ is injective, that $\theta(\alpha^{(i-1)})=\widetilde\gamma$ if $i\ge 1$ or that $\widetilde\gamma=\theta(\alpha^{(j-1)})$ if $i=0$. This contradicts \ref{prop1roots} of Lemma \ref{proproots}. Similar arguments hold for the other cases that is, when $\gamma=\theta(\alpha^{(i)})$ or $\gamma=\alpha^i$ or $\gamma=\theta(\alpha^i)$.
\end{proof}

\subsection{A lemma of nondegeneracy.}\label{Nondegen}

Now we can give the following Proposition, which gives a sufficient condition for the nondegeneracy of the restriction to $\g_O\times\g_O$ of the skew-symmetric bilinear form $\widetilde\Phi_y$.

\begin{prop}\label{condnondegen}
We assume that :
\begin{enumerate}
\item $S_{\mid\h_{\pi'}\oplus\h_{\Lambda}}$ is a basis for $(\h_{\pi'}\oplus\h_{\Lambda})^*$.
\item If $\alpha\in O^+$ then $S_{\alpha}\cap O^+=\{\theta(\alpha)\}$.
\item If $\alpha\in O^-$ then $S_{\alpha}\cap O^-=\{\theta(\alpha)\}$.
\item If $\alpha\in O^m$ then $\alpha$ is a stationary root.
\end{enumerate}
Let $y=\sum_{\gamma\in S}x_{-\gamma}$ and $\widetilde\Phi_y$ be the skew-symmetric bilinear form defined by $\widetilde\Phi_y(x,\,x')=K(y,\,[x,\,x']_{\widetilde\p})$ for all $x,\,x'\in\widetilde\p$. Then the restriction to $\g_O\times\g_O$ of  $\widetilde\Phi_y$ is nondegenerate.
\end{prop}
The proof is similar as in \cite[Lem. 6.1]{FL1} (see also \cite[8.4,\,8.5]{J5}). We give it for the reader's convenience.
\begin{proof}
Let $\rho$ be the linear form on $\h^*$ defined by $\rho(\alpha)=1$ for all $\alpha\in\pi$ and set $\rho(A)=\sum_{\alpha\in A}\rho(\alpha)$ for every subset $A$ of roots. We set $z(t)=\sum_{\gamma\in S}t^{\lvert\rho(\gamma)\rvert}x_{-\gamma}$ for all $t\in\Bbbk$ and $d(t)=\det(\widetilde\Phi_{{z(t)}_{\mid\g_O\times\g_O}})$, which is a polynomial in the variable $t$. Denote by $H$ the adjoint group of $\h_{\pi'}\oplus\h_{\Lambda}$. Fix $t_0\in\Bbbk$. By hypothesis (1) $z(ct_0)$ and $z(t_0)$ are in the same $H$-coadjoint orbit for all $c\in\Bbbk\setminus\{0\}$. Moreover $\g_O\times\g_O$ is stable under the adjoint action of $H$. Then $d(t_0)=0$ is equivalent to $d(ct_0)=0$ for all $c\in\Bbbk\setminus\{0\}$. It follows that either $d(t)$ is identically zero or it vanishes only at $t=0$. Then $d(t)$ is a multiple of a single power of $t$ (see also \cite[Rem. 8.4]{J5}).
Choose a basis of $\g_O$ formed by root vectors and write the determinant $d(t)$ in this basis. Then the term 
\begin{align*}\prod_{\alpha\in O^m}t^{\lvert\rho(\alpha+\theta(\alpha))\rvert}\prod_{\alpha\in O^+}t^{\lvert\rho(\alpha+\theta(\alpha))\rvert}\prod_{\alpha\in O^-}t^{\lvert\rho(\alpha+\theta(\alpha))\rvert}\\
=t^{\sum_{\alpha\in O^m}\lvert\rho(\alpha+\theta(\alpha))\rvert+\rho(O^+)-\rho(O^-)}\end{align*} appears in $d(t)$ (up to a nonzero scalar). 
Indeed this comes from  Lemma \ref{statlm} for the elements in $O^m$, since the only factor in $d(t)$ involving a root $\alpha\in O^m$, and then also $\theta(\alpha)\in O^m$, is $t^{2\lvert\rho(\alpha+\theta(\alpha))\rvert}$ up to a nonzero scalar. It is also true for a root $\alpha\in O^+$ or $\alpha\in O^-$ by hypotheses (2) and (3). Indeed if $\alpha\in O^\pm$ is such that there exists $\beta\in O^m\cap S_{\alpha}$ with $t^{\lvert\rho(\alpha+\beta))\rvert}$ appearing as a factor in $d(t)$, then this contradicts Lemma \ref{statlm} in view of hypothesis (4). 

Now if there exist roots $\alpha\in O^\pm$ such that $S_{\alpha}\cap O^\mp\neq\emptyset$
then such roots provide polynomials in $t$ of degree strictly smaller than $\rho(O^+)-\rho(O^-)+\sum_{\alpha\in O^m}\lvert\rho(\alpha+\theta(\alpha))\rvert$. This comes from the fact that if $\alpha\in O^+$ and $\beta\in O^-$ then one has that $\lvert\rho(\alpha+\beta)\rvert<\lvert\rho(\alpha)\rvert+\lvert\rho(\beta)\rvert$ while $\lvert\rho(\alpha+\theta(\alpha))\rvert=\lvert\rho(\alpha)\rvert+\lvert\rho(\theta(\alpha))\rvert$
and $\lvert\rho(\beta+\theta(\beta))\rvert=\lvert\rho(\beta)\rvert+\lvert\rho(\theta(\beta))\rvert$. Since $d(t)$ is a multiple of a single power of $t$,
the latter polynomials must vanish. Then the above term is the only one (up to a nonzero scalar) in $d(t)$ and then $d(t)\neq 0$.\end{proof}

\section{Weierstrass sections for In\"on\"u-Wigner contractions of even maximal parabolic subalgebras in type B.}

In this section, we will assume that $\g$ is simple of type ${\rm B}_n$ and that the standard parabolic subalgebra $\p$ of $\g$ is maximal namely, that it is associated with the subset $\pi'$ of simple roots $\pi$ such that $\pi'=\pi\setminus\{\alpha_s\}$ for $1\le s\le n$ (Bourbaki's notation \cite[Planche II]{BOU}). Moreover we will assume that $s$ is even.
We will say in this case that such a maximal parabolic subalgebra $\p$ is {\it even}. 

Recall Remark \ref{truncrq} that the canonical truncation $\p_{\Lambda}$ of $\p$ is equal to the derived subalgebra $\p'$ of $\p$ and then that the canonical truncation $\widetilde\p_{\Lambda}$ of $\widetilde\p$ is equal to the derived subalgebra $\widetilde\p'$ of $\widetilde\p$ by Corollary \ref{trunccor}. We will construct an element $y\in\widetilde\p'^*$ and a vector subspace $V$ of $\widetilde\p'^*$  verifying hypotheses $i),\,ii),\,iii),\,iv)$ of Lemma \ref{indexcontractionlm} and of Proposition \ref{WScontractionprop}.
 
We denote by $\ep_i$, for all $1\le i\le n$, the elements of an orthonormal basis with respect to the inner product $(\,,\,)$ in $\mathbb Q^n$ from which the root system $\Delta$ of $\g$ simple of type ${\rm B}_n$ is defined. For any real number $x$, denote by $[x]$ the unique integer such that $[x]\le x<[x]+1$.

\subsection{The set $S$}\label{SetS}

\begin{enumerate}
\item Suppose that $n>s$ and that $3s/2\le n$.

For the set $S^m$ we set :
\begin{align*}S^m=&\{\ep_s,\;\ep_{s-(2k-1)}+\ep_{s+k},\;\ep_{s-2k}-\ep_{s+k};\;1\le k\le (s-2)/2\}
\end{align*}

For the sets $S^+$ and $S^-$ we set :
\begin{enumerate}

\item If $s/2$ is even 
\begin{align*}S^+=\{\ep_1+\ep_{3s/2},\;\ep_{2i-1}+\ep_{2i};\;1\le i\le s/2-1,\\
\ep_{s+2j+1}+\ep_{s+2j+2};\;s/4\le j\le[(n-2-s)/2]\}\\
S^-=\{-\ep_{s+2j}-\ep_{s+2j+1};\;s/4\le j\le[(n-1-s)/2]\}.\end{align*}

\item If $s/2$ is odd
\begin{align*}S^+=\{\ep_1+\ep_{3s/2},\;\ep_{2i-1}+\ep_{2i};\;1\le i\le s/2-1,\\
\ep_{s+2j+2}+\ep_{s+2j+3};\;(s-2)/4\le j\le[(n-3-s)/2]\}\\
S^-=\{-\ep_{s+2j+1}-\ep_{s+2j+2};\;(s-2)/4\le j\le[(n-2-s)/2]\}.\end{align*}
\end{enumerate}

\item Suppose that $n>s$ and that $3s/2>n$.

For the set $S^m$ we set:
\begin{align*}S^m=&\{\ep_s,\;\ep_{s-(2k-1)}+\ep_{s+k},\;\ep_{s-2k}-\ep_{s+k};\;1\le k\le n-s\}.
\end{align*}

For the sets $S^+$ and $S^-$ we set:
\begin{align*} S^+=\{\ep_{2i-1}+\ep_{2i};\;1\le i\le s/2-1\}\\
S^-=\{\ep_{3s-2n-k}-\ep_k;\;1\le k\le 3s/2-n-1\}.\end{align*}

\item Suppose that $n=s$.

For the sets $S^m,\,S^+$ and $S^-$ we set:
\begin{align*} S^m=\{\ep_s\}\\
S^+=\{\ep_{2i-1}+\ep_{2i};\;1\le i\le s/2-1\}\\
S^-=\{\ep_{s-k}-\ep_k;\;1\le k\le s/2-1\}.\end{align*}

Observe that, for $n=s$, these sets coincide with the sets constructed in \cite[Sec. 7]{FL1}.
\end{enumerate}

\begin{lm}\label{basisB}
Let $S=S^m\cup S^+\cup S^-$. Then $S_{\mid \h_{\pi'}}$ is a basis for $\h_{\pi'}^*$.
\end{lm}

\begin{proof}

The proof was already done in \cite[Proof of Lem. 7.1]{FL1} for the case $n=s$. 

Suppose that $n>s$ and that $3s/2\le n$. We will make the proof for the case $s/2$ even. The case $s/2$ odd is very similar. Firstly we observe that $\lvert S\rvert=n-1=\dim\h_{\pi'}$.

Order the elements in $S=\{t_i\}_{1\le i\le n-1}$ so that the $s/2-1$ first elements are the $t_i=\ep_{2i-1}+\ep_{2i}$ for $1\le i\le s/2-1$. Then set $t_{s/2}=\ep_s$. The $s-2$ following elements in $S$ are the $\ep_{s-(2k-1)}+\ep_{s+k}$ and $\ep_{s-2k}-\ep_{s+k}$ for $1\le k\le (s-2)/2$, namely we set
$$t_{s/2+2k-1}=\ep_{s-(2k-1)}+\ep_{s+k},\;t_{s/2+2k}=\ep_{s-2k}-\ep_{s+k}$$ for $1\le k\le (s-2)/2$.
Then set $t_{3s/2-1}=\ep_1+\ep_{3s/2}$.

Finally for all $0\le k\le [(n-2-s)/2]-s/4$ set $$t_{3s/2+2k}=-\ep_{3s/2+2k}-\ep_{3s/2+2k+1},\;t_{3s/2+2k+1}=\ep_{3s/2+2k+1}+\ep_{3s/2+2k+2}.$$

Observe that if $n$ is even then $t_{n-1}=\ep_{n-1}+\ep_n$ is the last element of this list and if $n$ is odd the last element of this list is $t_{n-2}=\ep_{n-2}+\ep_{n-1}$.
Then if $n$ is odd, we set
$$t_{n-1}=-\ep_{n-1}-\ep_n.$$

Now we choose a basis $\{h_j\}_{1\le j\le n-1}$ in $\h_{\pi'}$ ordered as follows:

\begin{align*}\{\alpha^\vee_{2i};1\le i\le s/2-1,\;\alpha^\vee_{s-1},\;\alpha^\vee_{s+k},\,\alpha^\vee_{s-(2k+1)};\;1\le k\le s/2-1,\\
\;\alpha^\vee_{3s/2+j};\;0\le j\le n-3s/2\}\end{align*}

It is easily seen that the matrix $(t_i(h_j))_{1\le i,\,j\le n-1}$ is a lower triangular matrix with $1$ or $-1$ on the diagonal, except for the last one which is equal to $\pm 2$. 
Then $\det(t_i(h_j))_{1\le i,\,j\le n-1}\neq 0$ and we are done in this case.

Suppose that $n>s$ and that $3s/2>n$.

Order the elements $t_i$ of $S$ similarly as in the previous case that is, set
\begin{align*} t_i=\ep_{2i-1}+\ep_{2i};\; 1\le i\le s/2-1,\;t_{s/2}=\ep_s\\
t_{s/2+2k-1}=\ep_{s-(2k-1)}+\ep_{s+k},\;t_{s/2+2k}=\ep_{s-2k}-\ep_{s+k};\;1\le k\le n-s\\
t_{2n-3s/2+j}=\ep_{3s-2n-j}-\ep_j;\;1\le j\le 3s/2-n-1\end{align*}

Finally choose a basis $\{h_j\}_{1\le j\le n-1}$ in $\h_{\pi'}$ ordered as follows:

\begin{align*}\{\alpha^\vee_{2i};1\le i\le s/2-1,\;\alpha^\vee_{s-1},\;\alpha^\vee_{s+k},\,\alpha^\vee_{s-(2k+1)};\;1\le k\le n-s,\\
\alpha^\vee_{2j-1},\,\alpha^\vee_{3s-2n-(2j+1)};\;1\le j\le [(3s-2n)/4]\}\end{align*}

Then one checks easily that the matrix $(t_i(h_j))_{1\le i,\,j\le n-1}$ is a lower triangular matrix with $1$ or $-1$ on the diagonal. 

This implies that $\det(t_i(h_j))_{1\le i,\,j\le n-1}\neq 0$, which completes the proof.
\end{proof}

\subsection{The Heisenberg sets}\label{HS}

For $\gamma\in S$, recall that we set $\Gamma_{\gamma}^0=\Gamma_{\gamma}\setminus\{\gamma\}$, with $\Gamma_{\gamma}$ a Heisenberg set with centre $\gamma$.
We want to construct disjoint Heisenberg sets $\Gamma_{\gamma}$ with centre $\gamma\in S$, verifying {\bf Condition (C)} (Eq. \ref{ConditionC}).


The root system $\Delta_{\pi'}$ of $(\mathfrak r',\,\h_{\pi'})$ is spanned by two irreducible components of the set of simple roots $\pi'$. More precisely $\pi'=\pi'_1\sqcup\pi'_2$ with $\pi'_1=\{\alpha_1,\,\ldots,\,\alpha_{s-1}\}$ and $\pi'_2=\{\alpha_{s+1},\,\ldots,\,\alpha_n\}$. The set $\pi'_1$ spans a root system of type ${\rm A}_{s-1}$ (denote it by $\Delta_{\pi'_1}$)  and the set $\pi'_2$ spans a root system of type ${\rm B}_{n-s}$ (denote it by $\Delta_{\pi'_2}$). Similarly we set $\Delta^\pm_{\pi'_i}=\Delta_{\pi'_i}\cap\Delta^\pm$ for $i=1$ or $i=2$.

By the definitions given in subsection \ref{Nondegen}, we have that $\Gamma^+\subset\Delta^+$, $\Gamma^-\subset\Delta^-_{\pi'}$ and $\Gamma^m$ contains both positive and negative roots.

For $n=s$, the construction of $\Gamma$ is the same as in \cite[Sec. 7]{FL1} since all the Heisenberg sets built there satisfy {\bf Condition (C)} (Eq. \ref{ConditionC}) and give an adapted pair as required  in Lemma \ref{Adaplm}.We send the reader to \cite{FL1} for more details for this case.

From now on, suppose that $n>s$.
\begin{enumerate}
\item Let $1\le i\le s/2-1$. We explain below how to construct a Heisenberg set with centre $t_i=\ep_{2i-1}+\ep_{2i}$, which satisfies {\bf Condition (C)} (Eq. \ref{ConditionC}).
 Observe that the $t_i=\ep_{2i-1}+\ep_{2i}$, $1\le i\le s/2-1$,  correspond to the $s/2-1$ first strongly orthogonal positive roots which form the Kostant cascade for $\g$ (for more details, see for instance \cite[Sec. 7]{FL1}). Then denote by $H_i$ the maximal Heisenberg set with centre $t_i$ which is included in $\Delta^+$ (see for instance \cite[Example 3.1]{FL1} or \cite[2.2]{J1}). In other words one has that $$H_i=\{\ep_{2i-1}+\ep_{2i},\;\ep_{2i-1}\pm\ep_{j};\;\ep_{2i}\mp\ep_j;\;2i+1\le j\le n\}\subset\Delta^+.$$ 

Then the involution $\theta$ sends every $\ep_{2i-1}\pm\ep_{j}$ to $\ep_{2i}\mp\ep_j$, for all $2i+1\le j\le n$. Moreover 
\begin{align*}\ep_{2i-1}+\ep_{2i}\in\Delta^+\setminus\Delta^+_{\pi'}\\
\forall 2i+1\le j\le n,\;\ep_{2i-1}+\ep_{j},\;\ep_{2i}+\ep_j\in\Delta^+\setminus\Delta^+_{\pi'}\end{align*} and 
\begin{align*}\theta(\ep_{2i-1}+\ep_j)=\ep_{2i}-\ep_j\in\Delta^+_{\pi'_1}\iff 2i+1\le j\le s\\
\theta(\ep_{2i}+\ep_j)=\ep_{2i-1}-\ep_j\in\Delta^+_{\pi'_1}\iff 2i+1\le j\le s.\end{align*}

We set: 
$$\Gamma_{t_i}=\{\ep_{2i-1}+\ep_{2i},\;\ep_{2i-1}\pm\ep_{j};\;\ep_{2i}\mp\ep_j;\;2i+1\le j\le s\}\subset H_i$$
so that it is a Heisenberg set with centre $t_i=\ep_{2i-1}+\ep_{2i}$ which satisfies {\bf Condition (C)} (Eq. \ref{ConditionC}) and $\Gamma_{t_i}\subset\Gamma^+$.

\item Assume that $3s/2\le n$. We set:
\begin{align*} \Gamma_{\ep_1+\ep_{3s/2}}=\{\ep_1+\ep_{3s/2},\\
\;\ep_1\pm\ep_{3s/2+k},\;\ep_{3s/2}\mp\ep_{3s/2+k};\;1\le k\le n-3s/2\}.\end{align*}
Since $\ep_{3s/2}\mp\ep_{3s/2+k}=\theta(\ep_1\pm\ep_{3s/2+k})\in\Delta^+_{\pi'_2}$ for all $1\le k\le n-3s/2$, $\Gamma_{\ep_1+\ep_{3s/2}}$ is a Heisenberg set with centre $\ep_1+\ep_{3s/2}$, which satisfies {\bf Condition (C)} (Eq. \ref{ConditionC}). Moreover $\Gamma_{\ep_1+\ep_{3s/2}}\subset\Gamma^+$.

\item Assume that $3s/2\le n$ and $s/2$ is even. We define a Heisenberg set with centre $\gamma_j=\ep_{s+2j+1}+\ep_{s+2j+2}$ for all $s/4\le j\le[(n-2-s)/2]$ and a Heisenberg set with centre $\gamma'_j=-\ep_{s+2j}-\ep_{s+2j+1}$ for all $s/4\le j\le[(n-1-s)/2]$, by setting:
\begin{align*}\Gamma_{\gamma_j}=\{\gamma_j,\;\ep_{s+2j+1}\pm\ep_{k};\;\ep_{s+2j+2}\mp\ep_k;\;s+2j+3\le k\le n\}\subset\Delta^+_{\pi'_2}\\
\Gamma_{\gamma'_j}=\{\gamma'_j,\;-\ep_{s+2j}\pm\ep_{k};\;-\ep_{s+2j+1}\mp\ep_k;\;s+2j+2\le k\le n\}\subset\Delta^-_{\pi'_2}.\end{align*}
Then obviously $\Gamma_{\gamma_j}$ and $\Gamma_{\gamma'_j}$ satisfy {\bf Condition (C)} (Eq. \ref{ConditionC}). Moreover $\Gamma_{\gamma_j}\subset\Gamma^+$ and  $\Gamma_{\gamma'_j}\subset\Gamma^-$.

\item Assume that $3s/2\le n$ and $s/2$ is odd.  We define a Heisenberg set with centre $\gamma_j=\ep_{s+2j+2}+\ep_{s+2j+3}\in\Delta^+_{\pi'_2}$, for all $(s-2)/4\le j\le [(n-3-s)/2]$, resp. a Heisenberg set with centre $\gamma'_j=-\ep_{s+2j+1}-\ep_{s+2j+2}\in\Delta^-_{\pi'_2}$ for all $(s-2)/4\le j\le[(n-2-s)/2]$, by setting:
\begin{align*}\Gamma_{\gamma_j}=\{\gamma_j,\;\ep_{s+2j+2}\pm\ep_{k};\;\ep_{s+2j+3}\mp\ep_k;\;s+2j+4\le k\le n\}\subset\Delta^+_{\pi'_2}\\
\Gamma_{\gamma'_j}=\{\gamma'_j,\;-\ep_{s+2j+1}\pm\ep_{k};\;-\ep_{s+2j+2}\mp\ep_k;\;s+2j+3\le k\le n\}\subset\Delta^-_{\pi'_2}.\end{align*}
They satisfy {\bf Condition (C)} (Eq. \ref{ConditionC}). Moreover $\Gamma_{\gamma_j}\subset\Gamma^+$ and  $\Gamma_{\gamma'_j}\subset\Gamma^-$.

\item Assume that $3s/2>n$. Let $\gamma_j=\ep_{3s-2n-j}-\ep_j$ for $1\le j\le 3s/2-n-1$. Observe that $\gamma_j\in\Delta^-_{\pi'_1}$. Then we set: 
\begin{equation*}\Gamma_{\gamma_j}=\{\gamma_j,\;\ep_{k}-\ep_j,\,\ep_{3s-2n-j}-\ep_k;\;j+1\le k\le 3s-2n-j-1\}\subset\Delta^-_{\pi'_1}\end{equation*}
Obviously $\Gamma_{\gamma_j}$ is a Heisenberg set with centre $\gamma_j$ and it satisfies {\bf Condition (C)} (Eq. \ref{ConditionC}). One has that $\Gamma_{\gamma_j}\subset\Gamma^-$.

\item Assume that $3s/2\le n$ or $3s/2>n$. Set for all $1\le k\le \min((s-2)/2,\,n-s)$, $\delta_k=\ep_{s-(2k-1)}+\ep_{s+k}$ and $\delta'_k=\ep_{s-2k}-\ep_{s+k}$. We will explain how to construct a Heisenberg set with centre $\delta_k$, resp. $\delta'_k$, which satisfies {\bf Condition (C)} (Eq. \ref{ConditionC}). We set:
\begin{align*}\Gamma_{\delta_k}=\{\delta_k,\;\ep_{s-(2k-1)}\pm\ep_{s+k+i},\;\ep_{s+k}\mp\ep_{s+k+i};\;1\le i\le n-s-k,\\
\,\ep_{s+k}+\ep_{s-(2k-1)-j},\;\ep_{s-(2k-1)}-\ep_{s-(2k-1)-j};\;1\le j\le s-2k\}\end{align*}

and
\begin{align*}\Gamma_{\delta'_k}=\{\delta'_k,\;\ep_{s-2k}\mp\ep_{s+k+i},\;-\ep_{s+k}\pm\ep_{s+k+i};\;1\le i\le n-s-k,\\
-\ep_{s+k}+\ep_{s-2k-j},\;\ep_{s-2k}-\ep_{s-2k-j};\;1\le j\le s-2k-1\}.
\end{align*}

Since $\theta(\ep_{s-(2k-1)}\pm\ep_{s+k+i})=\ep_{s+k}\mp\ep_{s+k+i}\in\Delta_{\pi'_2}$ for all $1\le i\le n-s-k$ and $\theta(\ep_{s+k}+\ep_{s-(2k-1)-j})=\ep_{s-(2k-1)}-\ep_{s-(2k-1)-j}\in\Delta_{\pi'_1}$ for all $1\le j\le s-2k$, 
the Heisenberg set $\Gamma_{\delta_k}$ with centre $\delta_k$ satisfies condition (C). A similar argument shows that $\Gamma_{\delta'_k}$ is a Heisenberg set with centre $\delta'_k$, which satisfies {\bf Condition (C)} (Eq. \ref{ConditionC}). We have that $\Gamma_{\delta_k}\subset\Gamma^m$ and $\Gamma_{\delta'_k}\subset\Gamma^m$.

\item Assume that $3s/2\le n$ or $3s/2>n$. We set:
\begin{align*}\Gamma_{\ep_s}=\{\ep_s,\;\pm\ep_i,\,\ep_s\mp\ep_i;\;s+1\le i\le n,\;\ep_s-\ep_j,\;\ep_j;\;1\le j\le s-1\}.\end{align*}
Then $\Gamma_{\ep_s}$ is a Heisenberg set with centre $\ep_s$.

Moreover $\theta(\ep_s\pm\ep_i)=\mp\ep_i\in\Delta_{\pi'_2}$ for all $s+1\le i\le n$ and $\theta(\ep_j)=\ep_s-\ep_j\in\Delta_{\pi'_1}$ for all $1\le j\le s-1$. Then $\Gamma_{\ep_s}$ satisfies {\bf Condition (C)} (Eq. \ref{ConditionC}) and $\Gamma_{\ep_s}\subset\Gamma^m$.

\end{enumerate}

One verifies furthermore that all the Heisenberg sets described above are disjoint.

\subsection{The set $T$}

Denote by $T$ the complement of $\Gamma$ in $\Delta(\pi')=\Delta^+\sqcup\Delta^-_{\pi'}$.

\begin{lm}\label{TB}
We have that $\lvert T\rvert={\rm index}\,\p'$.
\end{lm}

\begin{proof}

Recall for instance \cite[Proof of Lem. 7.5]{FL1} that ${\rm index}\,\p'=n-s/2+1$.

For $n=s$, we have that $T=\{\ep_{s-1}+\ep_s,\,\ep_{2i-1}-\ep_{2i};\;1\le i\le s/2\}$ and then $\lvert T\rvert={\rm index}\,\p'$.

Assume that $n>s$ and that $3s/2\le n$.

 If $s/2$ is even,  one checks that 
\begin{align*}T=\{\ep_1-\ep_{3s/2},\,\ep_{s-1}+\ep_s,\,\ep_{2i-1}-\ep_{2i};\;1\le i\le s/2,\\
\ep_{s-(2k-1)}-\ep_{s+k};\;1\le k\le (s-2)/2,\\
\ep_{s+2j+1}-\ep_{s+2j+2};\;s/4\le j\le[(n-2-s)/2],\\
-\ep_{s+2l}+\ep_{s+2l+1};\;s/4\le l\le[(n-1-s)/2]\}.\end{align*}

Then $\lvert T\rvert=n-s/2+1={\rm index}\,\p'.$

If $s/2$ is odd, one checks that
\begin{align*}T=\{\ep_1-\ep_{3s/2},\,\ep_{s-1}+\ep_s,\,\ep_{2i-1}-\ep_{2i};\;1\le i\le s/2,\\
\ep_{s-(2k-1)}-\ep_{s+k};\;1\le k\le (s-2)/2,\\
\ep_{s+2j+2}-\ep_{s+2j+3};\;(s-2)/4\le j\le[(n-3-s)/2],\\
-\ep_{s+2l+1}+\ep_{s+2l+2};\;(s-2)/4\le l\le[(n-2-s)/2]\}.\end{align*}

Then $\lvert T\rvert=n-s/2+1={\rm index}\,\p'.$

Assume that $n>s$ and that $3s/2>n$.

Then one checks that
\begin{align*}T=\{\ep_{s-1}+\ep_s,\,\ep_{2i-1}-\ep_{2i};\;1\le i\le s/2,\\
\ep_{s-(2k-1)}-\ep_{s+k};\;1\le k\le n-s\}.\end{align*}

Then $\lvert T\rvert=n-s/2+1={\rm index}\,\p'.$
\end{proof}

\subsection{The nondegeneracy of the restriction of $\widetilde\Phi_y$ to $\g_O\times\g_O$}

Recall that $y=\sum_{\gamma\in S} x_{-\gamma}$. We will verify below that conditions $\it(2)$, $\it(3)$ and $\it(4)$ of Prop. \ref{condnondegen} and Condition $\bf (C')$ are satisfied.


\begin{lm}\label{O}
Let $\alpha\in O^{\pm}$. Then $S_{\alpha}\cap O^{\pm}=\{\theta(\alpha)\}$.

\end{lm}

\begin{proof}

It can be checked by direct computation.
\end{proof}

\begin{lm}\label{lmconditionC'}

The condition $(\bf C')$ (Eq. \ref{ConditionC'}) is satisfied.
\end{lm}

\begin{proof}
By direct computation, one can check that any root $\alpha\in O$ belongs to $O_1\sqcup O_2\sqcup O_3$. Moreover one may also verify that $O_3\neq\emptyset$ only when $3s/2> n$ and that in this case the only roots $\alpha\in O_3$ are of the form $\alpha=\varepsilon_v-\varepsilon_u\in\Delta^-_{\pi'}$ with $1\le u\le 3s/2-n-1<3s/2-n+1\le v\le s-2$. Two cases may occur. 
1) The first case is when $1\le u\le 3s/2-n-1<3s/2-n+1\le v\le 3s-2n-1$. 
\begin{enumerate}
\item[i)] Assume that $v$ is odd. 
\begin{enumerate}
\item[a)] If moreover $u+v\le 3s-2n-1$ then $\alpha\in\Gamma^0_{\ep_{3s-2n-u}-\ep_u}$ and $\theta(\alpha)=\ep_{3s-2n-u}-\ep_v$. Hence we have two roots in $S_{\alpha}\setminus\{\theta(\alpha)\}$, namely $\alpha^{(1)}=\ep_u-\ep_{3s-2n-v}$ and $\widetilde{\alpha}=\ep_u+\ep_{v+1}$.
\item[b)] If $u+v\ge 3s-2n+1$ then $\alpha\in\Gamma^0_{\ep_{v}-\ep_{3s-2n-v}}$ and then $\theta(\alpha)=\ep_u-\ep_{3s-2n-v}$. Similarly we have two roots in $S_{\alpha}\setminus\{\theta(\alpha)\}$, namely $\alpha^{(1)}=\ep_{3s-2n-u}-\ep_v$ and $\widetilde{\alpha}=\ep_u+\ep_{v+1}$.
\item[c)] In both cases above, we observe that $\widetilde{\alpha}\in\Gamma^0_{\ep_u+\ep_{u+1}}$ if $u$ is odd, resp. $\widetilde{\alpha}\in\Gamma^0_{\ep_{u-1}+\ep_{u}}$ if $u$ is even. Then $\theta(\widetilde\alpha)=\ep_{u+1}-\ep_{v+1}$ if $u$ is odd, resp. $\theta(\widetilde\alpha)=\ep_{u-1}-\ep_{v+1}$ if $u$ is even. We check that $\theta(\widetilde\alpha)\in O_1$ and that $\widetilde\alpha\in O_2$. 
\end{enumerate}
\item[ii)] Assume that  $v$ is even. Then similar considerations imply that $\theta(\widetilde\alpha)\in O_1$ and that $\widetilde\alpha\in O_2$.
\end{enumerate}
2) The second case is when $1\le u\le 3s/2-n-1<3s-2n\le v\le s-2$. 
\begin{enumerate}
\item[i)] Assume that $v$ is odd. In this case $\alpha=\ep_v-\ep_u\in \Gamma^0_{\delta_k}$ with $k\in\mathbb N^*$ such that $s-(2k-1)=v$. It follows that $\alpha\in\Gamma^0_{\ep_v+\ep_{3s/2-(v-1)/2}}$. Then $\theta(\alpha)=\ep_u+\ep_{3s/2-(v-1)/2}$ and there are two roots in $S_{\alpha}\setminus\{\theta(\alpha)\}$, namely $\alpha^{(1)}=\ep_{3s-2n-u}-\ep_v$ and $\widetilde\alpha=\ep_u+\ep_{v+1}$. One can conclude as above.
\item[ii)] Assume that $v$ is even. Similar considerations give again Eq. \ref{ConditionC'}.
\end{enumerate}
\end{proof}

\begin{Rq}\label{rqtilde}\rm

By the above proof, one may observe that the roots $\widetilde\alpha$, with $\alpha\in O_3$ are of the form $\widetilde\alpha=\ep_u+\ep_v\in\Delta^+\setminus\Delta^+_{\pi'}$ with $1\le u\le 3s/2-n-1<v\le s-2$ and $\theta\bigl(\widetilde\alpha\bigr)=\ep_{u+1}-\ep_v$ if $u$ is odd, resp. $\theta\bigl(\widetilde\alpha\bigr)=\ep_{u-1}-\ep_v$ if $u$ is even.

\end{Rq}

\begin{lm}\label{Omstat}
Any root $\alpha\in O^m$ is stationary.
\end{lm}

\begin{proof}
Recall that $$O^m=\bigsqcup_{1\le k\le\min((s-2)/2,\,n-s)}\Gamma^0_{\delta_k}\sqcup \Gamma^0_{\delta'_k}\sqcup\Gamma^0_{\ep_s}$$
with for any $1\le k\le\min((s-2)/2,\,n-s)$, $\delta_k=\ep_{s-(2k-1)}+\ep_{s+k}$ and $\delta'_k=\ep_{s-2k}-\ep_{s+k}$, which both belong to $\Delta^+\setminus\Delta^+_{\pi'}$.
We will show that every root $\alpha\in O^m\cap\Delta^+\setminus\Delta^+_{\pi'}$ is stationary. This will imply by Lemma \ref{proproots} that $\theta(\alpha)$ is also stationary and then that any root in $O^m$ is stationary. Firstly one may check that every root $\alpha\in O^m\cap\Delta^+\setminus\Delta^+_{\pi'}$  is such that there exists no root $\gamma\in O_3$ verifying that $\alpha=\widetilde\gamma$ nor $\alpha=\theta(\widetilde\gamma)$ and that $\alpha\in O_1\sqcup O_2$ (by a simple observation using Remark \ref{rqtilde}, and Proof of Lemma \ref{lmconditionC'}).

Take $1\le k\le\min((s-2)/2,\,n-s)$.

1) Assume that $\alpha=\ep_{s-(2k-1)}+\ep_{s+k+i}\in\Gamma^0_{\delta_k}\cap\Delta^+\setminus\Delta^+_{\pi'}$ with $1\le i\le n-s-k$. 
\begin{enumerate}
\item[i)] Assume that $k+i=s/2$. Then $\alpha^{(1)}=\ep_1-\ep_{s-(2k-1)}\in\Gamma^0_{\ep_1+\ep_2}$ and $\theta(\alpha^{(1)})=\ep_2+\ep_{s-(2k-1)}\in O_1$. Then $\alpha^{(2)}=\alpha^{(1)}$ and the sequence $(\alpha^{(i)})_{i\in\mathbb N}$ is stationary at rank one. It follows by Lemma \ref{proproots} that $\alpha$ is stationary in this case.
\item[ii)] Assume that $k+i>s/2$. Then one checks that $\alpha\in O_1$ and it follows that $\alpha^{(1)}=\theta(\alpha)=\alpha^{(0)}$ and the sequence $(\alpha^{(i)})_{i\in\mathbb N}$ is stationary at rank 0. Again the root $\alpha$ is stationary.
\item[iii)] Assume that $k+i< s/2$. Since one also has that $k+i\le n-s$, one has that $\alpha^{(1)}=\ep_{s-(2(k+i)-1)}-\ep_{s-(2k-1)}\in\Gamma^0_{\ep_{s-2(k+i)+1}+\ep_{s-2(k+i)+2}}$. Then $\theta(\alpha^{(1)})=\ep_{s-2k+1}+\ep_{s-2(k+i)+2}$.
\begin{enumerate}
\item[a)] If $k=1$, then $\theta(\alpha^{(1)})\in O_1$ and $\alpha^{(2)}=\alpha^{(1)}$.
\item[b)] If $k\ge 2$, then $\alpha^{(2)}=\ep_{s-2k+2}-\ep_{s-2(k+i)+2}\in \Gamma^0_{\delta'_{k-1}}$ and $\theta(\alpha^{(2)})=\ep_{s-2(k+i-1)}-\ep_{s+k-1}$. One checks that $\alpha^{(3)}=\ep_{s+k-1}-\ep_{s+k+i-1}\in\Gamma^0_{\delta_{k-1}}$ and $\theta(\alpha^{(3)})=\ep_{s+k-1+i}+\ep_{s-2(k-1)+1}\in\Delta^+\setminus\Delta^+_{\pi'}$. 
\end{enumerate}
Then we can repeat the same argument as for $\alpha$ for $\theta(\alpha^{(3)})$ with $k-1$ instead of $k$. One obtains that  $\theta(\alpha^{(3(k-1)+1)})\in O_1$. It follows  that $\alpha$ is stationary.
\end{enumerate}
2) Assume that $\alpha=\ep_{s-(2k-1)}-\ep_{s+k+i}\in\Gamma^0_{\delta_k}\cap\Delta^+\setminus\Delta^+_{\pi'}$ with $1\le i\le n-s-k$.

\begin{enumerate}
\item[i)] Assume that $k+i>(s-2)/2$. Then one checks that $\alpha\in O_1$ and then $\alpha^{(1)}=\alpha^{(0)}$ and the root $\alpha$ is stationary.
\item[ii)] Assume that $k+i\le(s-2)/2$. Then one checks that $\alpha^{(1)}=\ep_{s-2(k+i)}-\ep_{s-(2k-1)}\in\Gamma^0_{\ep_{s-2(k+i)}+\ep_{s-2(k+i)-1}}\cap\Delta^+_{\pi'}$ and $\theta(\alpha^{(1)})=\ep_{s-(2k-1)}+\ep_{s-2(k+i)-1}\in\Delta^+\setminus\Delta^+_{\pi'}$.
\begin{enumerate}
\item[a)] If $k=1$, then one has that $\theta(\alpha^{(1)})\in O_1$. Hence $\alpha^{(2)}=\alpha^{(1)}$ and the root $\alpha$ is stationary.
\item[b)] If $k\ge 2$, then one has that $\alpha^{(2)}=\ep_{s-2(k-1)}-\ep_{s-2(k+i)-1}\in\Delta^-_{\pi'_1}\cap\Gamma^0_{\delta'_{k-1}}$ and $\theta(\alpha^{(2)})=\ep_{s-2(k+i)-1}-\ep_{s+k-1}\in\Delta^+\setminus\Delta^+_{\pi'}$.
 If $k+i=n-s$ then one has that $\theta(\alpha^{(2)})\in O_1$. Hence $\alpha^{(3)}=\alpha^{(2)}$ and the root $\alpha$ is stationary.
Otherwise one checks that $\alpha^{(3)}=\ep_{s+k-1}+\ep_{s+k+i+1}\in\Delta^+_{\pi'_2}\cap\Gamma^0_{\delta_{k-1}}$. Hence $\theta(\alpha^{(3)})=\ep_{s-2(k-1)+1}-\ep_{s+k+i+1}\in\Delta^+\setminus\Delta^+_{\pi'}\cap\Gamma^0_{\delta_{k-1}}$. Repeating the same argument for the root $\theta(\alpha^{(3)})$ as for the root $\alpha$ with $k-1$ instead of $k$ and $i+2$ instead of $i$, one deduces that  the root $\alpha$ is stationary.

\end{enumerate}
\end{enumerate}
3) Assume that $\alpha=\ep_{s+k}+\ep_{s-(2k-1)-j}\in\Gamma^0_{\delta_k}\cap\Delta^+\setminus\Delta^+_{\pi'}$ with $1\le j\le s-2k$. 
\begin{enumerate}
\item[i)] Assume that $j$ is odd.
\begin{enumerate}
\item[a)] If $j=1$ or if $k+(j-1)/2>\min((s-2)/2,\,n-s)$ then one checks that $\alpha\in O_1$. Hence $\alpha^{(1)}=\alpha^{(0)}$ and the root $\alpha$ is stationary.
\item[b)] Assume that $j\ge 3$ and $k+(j-1)/2\le\min((s-2)/2,\,n-s)$. Then one checks that $\alpha^{(1)}=-\ep_{s+k}-\ep_{s+k+(j-1)/2}\in\Delta^-_{\pi'_2}\cap\Gamma^0_{\delta'_k}$. Then $\theta(\alpha^{(1)})=\ep_{s-2k}+\ep_{s+k+(j-1)/2}\in\Delta^+\setminus\Delta^+_{\pi'}$. 
If $j=3$ then one checks that $\theta(\alpha^{(1)})\in O_1$. Hence $\alpha^{(2)}=\alpha^{(1)}$ and the root $\alpha$ is stationary.
Assume that $j\ge 5$. Then one checks that $\alpha^{(2)}=-\ep_{s-2k}+\ep_{s-2k-j+2}\in\Delta^+_{\pi'_1}\cap\Gamma^0_{\ep_{s-2k-j+2}+\ep_{s-2k-j+3}}$. Then $\theta(\alpha^{(2)})=\ep_{s-2k}+\ep_{s-2k-j+3}\in\Delta^+\setminus\Delta^+_{\pi'}$. One checks then that $\alpha^{(3)}=\ep_{s-2k-1}-\ep_{s-2k-j+3}\in\Delta^-_{\pi'_1}\cap\Gamma^0_{\delta_{k+1}}$. Hence $\theta(\alpha^{(3)})=\ep_{s+k+1}+\ep_{s-2(k+1)+1-(j-4)}\in\Delta^+\setminus\Delta^+_{\pi'}\cap\Gamma^0_{\delta_{k+1}}$. We continue with $k+1$ instead of $k$ and $j-4$ instead of $j$. It follows that the root $\alpha$ is stationary.
\end{enumerate}
\item[ii)] Assume that $j$ is even. Similar considerations as above allow us to deduce that the root $\alpha$ is also stationary in this case.
\end{enumerate}
\item 
One deduces that any root  in $\Gamma^0_{\delta_k}$ is stationary.

4) Assume that $\alpha=\ep_{s-2k}+\ep_{s+k+i}\in\Gamma^0_{\delta'_k}\cap\Delta^+\setminus\Delta^+_{\pi'}$, with $1\le i\le n-s-k$.  
We have $\theta(\alpha)=-\ep_{s+k}-\ep_{s+k+i}\in\Delta^-_{\pi'_2}$ and $\theta(\alpha)\in O_1\sqcup O_2$.
\begin{enumerate}
\item[i)] If $k+i\le(s-2)/2$ then $\alpha^1=\ep_{s-2(k+i)}+\ep_{s+k}\in\Gamma^0_{\delta_k}$. It follows by the above that there exists $n_0\in\mathbb N$ such that $({\alpha^1})^{n_0+1}=({\alpha^1})^{n_0}$ that is, $\alpha^{n_0+2}=\alpha^{n_0+1}$ by \ref{eq2Rkseq} of Remark \ref{Rkseq}. Hence the root $\alpha$ is stationary by Lemma \ref{proproots}.
\item[ii)] Assume that $k+i>(s-2)/2$ (then $3s/2\le n$). \begin{enumerate}
\item[a)] Assume that $s/2$ is even. If $k+i=n-s$ and $k+i$ even, then $\theta(\alpha)\in O_1$. If $k+i$ is even and $k+i+1\le n-s$, then $\alpha^1=\ep_{s+k}-\ep_{s+k+i+1}\in\Gamma^0_{\delta_k}$. If $k+i$ is odd and $i\ge 2$ then $\alpha^1=\ep_{s+k}-\ep_{s+k+i-1}\in\Gamma^0_{\delta_k}$. If $k+i$ is odd and $i=1$ then $\theta(\alpha)\in O_1$. 
\item[b)] Assume that $s/2$ is odd.  If $k+i$ is even and $i\ge 2$, then $\alpha^1=\ep_{s+k}-\ep_{s+k+i-1}\in\Gamma^0_{\delta_k}$.  If $k+i$ is odd and $k+i+1\le n-s$, then $\alpha^1=\ep_{s+k}-\ep_{s+k+i+1}\in\Gamma^0_{\delta_k}$. 
\end{enumerate} We conclude as above that $\alpha$ is stationary.

\end{enumerate}
 
 5) Assume that $\alpha=\ep_{s-2k}-\ep_{s+k+i}\in\Gamma^0_{\delta'_k}\cap\Delta^+\setminus\Delta^+_{\pi'}$, with $1\le i\le n-s-k$. Similarly as above, we conclude that $\alpha^1=\alpha$ or that $\alpha^1\in\Gamma^0_{\delta_k}$ which implies that $\alpha$ is stationary.
 
 6) Assume that $\alpha=-\ep_{s+k}+\ep_{s-2k-j}\in\Gamma^0_{\delta'_k}\cap\Delta^+\setminus\Delta^+_{\pi'}$ with $1\le j\le s-2k-1$. 
 \begin{enumerate}
 \item[i)] Assume that $j$ is odd. If $k+(j+1)/2\le\min((s-2)/2,\,n-s)$ or if $k+(j+1)/2=s/2\le n-s$ then $\alpha^{(1)}=\ep_{s+k}+\ep_{s+k+(j+1)/2}\in\Gamma^0_{\delta_k}\cap\Delta^+_{\pi'_2}$. Otherwise we can check that $\alpha\in O_1$. In any case, we obtain by the above that $\alpha$ is stationary by Lemma \ref{proproots}.
 \item[ii)] Assume that $j$ is even.  If $k+j/2\le\min((s-2)/2,\,n-s)$  then $\alpha^{(1)}=\ep_{s+k}-\ep_{s+k+j/2}\in\Gamma^0_{\delta_k}\cap\Delta^+_{\pi'_2}$. Otherwise we can check that $\alpha\in O_1$. In any case, we obtain by the above that $\alpha$ is stationary by Lemma \ref{proproots}.
 \end{enumerate}
 One deduces that any root in $\Gamma^0_{\delta'_k}$ is stationary.

7) Assume that $\alpha=\ep_s-\ep_i\in\Gamma^0_{\ep_s}\cap\Delta^+\setminus\Delta^+_{\pi'}$, with $s+1\le i\le n$. 
\begin{enumerate}
\item[i)] If $i\ge 3s/2$ one checks easily that $\alpha\in O_1$ and then $\alpha^{(1)}=\alpha^{(0)}$ and the root $\alpha$ is stationary.
 
\item[ii)]  If $i<3s/2$ one checks that $\alpha^{(1)}=\ep_{3s-2i}-\ep_s\in\Delta^+_{\pi'_1}\cap\Gamma^0_{\ep_{3s-2i-1}+\ep_{3s-2i}}$ and then $\theta(\alpha^{(1)})=\ep_s+\ep_{3s-2i-1}\in O_1$. It follows that $\alpha^{(2)}=\alpha^{(1)}$ and the root $\alpha$ is stationary.
  \end{enumerate}
 8) Assume that $\alpha=\ep_s+\ep_i\in\Gamma^0_{\ep_s}\cap\Delta^+\setminus\Delta^+_{\pi'}$, with $s+1\le i\le n$. 
\begin{enumerate}
\item[i)] If $s+1<i\le 3s/2$ one checks that $\alpha^{(1)}=\ep_{3s-2i+1}-\ep_s\in\Delta^+_{\pi'_1}\cap\Gamma^0_{\ep_{3s-2i+1}+\ep_{3s-2i+2}}$ and then $\theta(\alpha^{(1)})=\ep_s+\ep_{3s-2i+2}\in O_1$. Hence the root $\alpha$ is stationary.

\item[ii)] If $i>3s/2$ or if $i=s+1$ then one checks that $\alpha\in O_1$ and again $\alpha$ is stationary.
\end{enumerate}

9) Assume that $\alpha=\ep_j\in\Gamma^0_{\ep_s}\cap\Delta^+\setminus\Delta^+_{\pi'}$, with $1\le j\le s-1$. 
\begin{enumerate}
\item[i)] If $j$ is odd, one checks that $\alpha^{(1)}=\ep_{3s/2+(1-j)/2}\in\Delta^+_{\pi'_2}\cap\Gamma^0_{\ep_s}$ and then $\theta(\alpha^{(1)})=\ep_s-\ep_{3s/2+(1-j)/2}\in\Gamma^0_{\ep_s}$. By the above (case (7)) we deduce that the root $\alpha$ is stationary.

\item[ii)] If $j$ is even, one checks that $\alpha^{(1)}=-\ep_{3s/2-j/2}\in\Delta^-_{\pi'_2}\cap\Gamma^0_{\ep_s}$ and then $\theta(\alpha^{(1)})=\ep_s+\ep_{3s/2-j/2}\in\Gamma^0_{\ep_s}$. By the above (case (8)) one has again that the root $\alpha$ is stationary.
\end{enumerate}

We deduce that any root  in $\Gamma^0_{\ep_s}$ is stationary. This completes the proof of the Lemma.
\end{proof}

\subsection{An adapted pair.}

Recall that $\h_{\Lambda}=\{0\}$ in the present case.
Then by Lemmas \ref{basisB}, \ref{O} and \ref{Omstat} every condition of Prop. \ref{condnondegen} is satisfied. Finally Lemma \ref{TB} gives also  that every condition of Lemma \ref{Adaplm} for having an adapted pair  is satisfied. Then one can deduce the following.

\begin{cor}\label{corb} Let $y=\sum_{\gamma\in S}x_{-\gamma}$.
One has that $\widetilde\p_{\Lambda}=\widetilde{\p_{\Lambda}}=\widetilde\p'$ and
$$\ad^*\widetilde\p_{\Lambda}(y)\oplus\g_{-T}=\widetilde\p_{\Lambda}^*.$$
In particular $y$ is regular in $\widetilde\p_{\Lambda}^*$ and if we denote by $h\in\h_{\pi'}$ the unique element such that $\gamma(h)=1$ for all $\gamma\in S$, then $(h,\,y)$ is an adapted pair for $\widetilde\p'$. Moreover  for all $\gamma\in T$, denote by $s(\gamma)$ the unique element in $\mathbb Q S$ such that $\gamma+s(\gamma)$ vanishes on $\h_{\pi'}$. Then if, for all $\gamma\in T$, $\gamma+s(\gamma)\neq 0$ and $s(\gamma)\in\mathbb NS$, one has that 
\begin{equation}{\rm ch}\,Sy\bigl(\widetilde\p\bigr)={\rm ch}\,Y\bigl(\widetilde\p'\bigr)\le\prod_{\gamma\in T}\bigl(1-e^{\gamma+s(\gamma)}\bigr)^{-1}.\label{ub}\end{equation}\end{cor}

To prove that $y+\g_{-T}$ is a Weierstrass section for $Sy\bigl(\widetilde\p\bigr)$ it remains to compute $s(\gamma)$, for all $\gamma\in T$, and  by Lemma \ref{Adaplm} to show that the upper bound given above coincides with the lower bound constructed in \cite{F00}.

Since for $n=s$, we have taken exactly the same Heisenberg sets as in the nondegenerate case (by what we explained in the beginning of subsection \ref{HS}), we have already obtained for $n=s$ a Weierstrass section for $Sy\bigl(\widetilde\p\bigr)$ in this case by \cite[Lem. 7.9, Thm. 7.10]{FL1}. From now on we will assume that $n>s$.

\subsection{A Weierstrass section.}

Here we recall the lower bound constructed in \cite{F00}. By \cite[Prop. 9.10.1]{F00} one has the following Proposition.
\begin{prop}
Let $E(\pi')$ be the set of $\langle ij\rangle$-orbits in $\pi$ defined as in \cite[Sec. 8]{F00}. For any $\Gamma\in E(\pi')$, set $\delta_{\Gamma}=w_{0}'d_{\Gamma}-w_0d_{\Gamma}$, with $d_{\Gamma}=\sum_{\gamma\in\Gamma}\varpi_{\gamma}$. One has that
\begin{equation}\prod_{\Gamma\in E(\pi')}(1-e^{\delta_{\Gamma}})^{-1}\label{lb}\end{equation} is a lower bound for ${\rm ch}\,Sy\bigl(\widetilde\p\bigr)$ when the latter is well defined, namely when 
every weight subspace $Sy\bigl(\widetilde\p\bigr)_{\nu}$, for $\nu\in\h^*$, of $Sy\bigl(\widetilde\p\bigr)$ is finite-dimensional.
\end{prop}

Moreover by \cite[Lem. 7.8]{FL1}, the lower bound is given by the following Eq. \ref{cha}. (Since in \cite{FL1} one has considered the opposite parabolic subalgebra, the weights in the present paper must be changed to their opposite.)
\begin{lm}{\cite[Lem. 7.8]{FL1}}
Let $\g$ be a simple Lie algebra of type ${\rm B}_n$ and $\p$ be a maximal parabolic subalgebra of $\g$ associated with the subset $\pi'=\pi\setminus\{\alpha_s\}$of simple roots, with $s$ even. Then if $n>s$, one has that \begin{equation}\prod_{\Gamma\in E(\pi')}(1-e^{\delta_{\Gamma}})^{-1}=\bigl(1-e^{\varpi_s}\bigr)^{-2}\bigl(1-e^{2\varpi_s}\bigr)^{-(n-1-s/2)}.\label{cha}\end{equation}
\end{lm}

We show below that the upper bound given in Corollary \ref{corb} (right hand of Eq. \ref{ub}) actually exists and that it coincides with the lower bound constructed in \cite{F00} (which is given by Eq. \ref{cha}).
\begin{lm}\label{eqbounds}
(1) For all $\gamma\in T$, $s(\gamma)\in\mathbb NS$ and $\gamma+s(\gamma)\neq 0$. It follows that the formal character of $Sy\bigl(\widetilde\p\bigr)$ is well defined and that we have inequality \ref{ub}.

(2) Moreover we have that \begin{equation}\prod_{\Gamma\in E(\pi')}(1-e^{\delta_{\Gamma}})^{-1}=\prod_{\gamma\in T}\bigl(1-e^{\gamma+s(\gamma)}\bigr)^{-1}.\label{egborn}\end{equation}

(3) Then by the end of Lemma \ref{Adaplm}, restriction of functions is an isomorphism from $Sy\bigl(\widetilde\p\bigr)$ to the algebra of polynomial functions $\Bbbk[y+\g_{-T}]$ on $y+\g_{-T}$.
\end{lm}

\begin{proof}
Recall the set $T$ given in Proof of Lemma \ref{TB} and the set $S$ given in subsection \ref{SetS}. For every $\gamma\in T$, we will also compute the number $1+\lvert s(\gamma)\rvert$, which we will denote by $\partial_{\gamma}$ (in order to compute the degree of every homogeneous generator of weight $\gamma+s(\gamma)$, as observed in Remark \ref{WScontractionRq}).

Take $\gamma=\ep_{s-1}+\ep_s\in T$. Then one checks that $$s(\gamma)=(\ep_1+\ep_2)+(\ep_3+\ep_4)+\ldots+(\ep_{s-3}+\ep_{s-2})\in\mathbb NS$$ so that $\gamma+s(\gamma)=\ep_1+\ldots+\ep_s=\varpi_s$. Moreover one has that $\partial_{\gamma}=(s-2)/2+1=s/2$.

Take $\gamma=\ep_{s-1}-\ep_s\in T$. Then one checks that $$s(\gamma)=(\ep_1+\ep_2)+(\ep_3+\ep_4)+\ldots+(\ep_{s-3}+\ep_{s-2})+2\ep_s\in\mathbb NS$$ so that $\gamma+s(\gamma)=\varpi_s$. Moreover $\partial_{\gamma}=(s-2)/2+2+1=s/2+2$.

Take $\gamma=\ep_{s-(2k-1)}-\ep_{s+k}\in T$ with $1\le k\le\min((s-2)/2,\,n-s)$. Then one checks that \begin{align*}s(\gamma)=2\sum_{j=1}^{k-1}\bigl(\ep_{s-(2j-1)}+\ep_{s+j}\bigr)+2\sum_{j=1}^{k-1}\bigl(\ep_{s-2j}-\ep_{s+j}\bigr)+(\ep_{s-(2k-1)}+\ep_{s+k})+\\
2\ep_s+2\sum_{i=1}^{s/2-k}\bigl(\ep_{2i-1}+\ep_{2i}\bigr)\in\mathbb NS\end{align*} so that $\gamma+s(\gamma)=2\varpi_s$. Moreover
$\partial_{\gamma}=s+2k$.

Let $u\in\mathbb N$ be equal to zero if $3s/2\le n$ and otherwise $u=3s/2-n$ and take $\gamma=\ep_{2i-1}-\ep_{2i}\in T$, with $u+1\le i\le s/2-1$.  Then one checks that \begin{align*}s(\gamma)=2\sum_{j=1}^{u}\bigl(\ep_{2j-1}+\ep_{2j}\bigr)+2\sum_{j=u+1}^{i-1}\bigl(\ep_{2j-1}+\ep_{2j}\bigr)+2\ep_s+(\ep_{2i-1}+\ep_{2i})+\\
2\sum_{k=1}^{s/2-i}\bigl(\ep_{s-(2k-1)}+\ep_{s+k}\bigr)+2\sum_{k=1}^{s/2-i}\bigl(\ep_{s-2k}-\ep_{s+k}\bigr)\in\mathbb NS\end{align*} so that $\gamma+s(\gamma)=2\varpi_s$. Moreover $\partial_{\gamma}=2s-2i+2$.

Assume now that $n<3s/2$.

 Take $\gamma=\ep_{2i-1}-\ep_{2i}\in T$ with $1\le i\le [3s/4-n/2]$. One checks that \begin{align*}s(\gamma)=2\sum_{j=1}^{2i-1}\bigl(\ep_{3s-2n-j}-\ep_j\bigr)+4\sum_{j=1}^{i-1}\bigl(\ep_{2j-1}+\ep_{2j}\bigr)+2\sum_{k=1}^{n-s}\bigl(\ep_{s-(2k-1)}+\ep_{s+k}\bigr)+\\
2\sum_{k=1}^{n-s}\bigl(\ep_{s-2k}-\ep_{s+k}\bigr)+
2\ep_s+3(\ep_{2i-1}+\ep_{2i})+
2\sum_{j=i+1}^{3s/2-n-i}\bigl(\ep_{2j-1}+\ep_{2j}\bigr)\in\mathbb NS\end{align*} so that $\gamma+s(\gamma)=2\varpi_s$. Moreover $\partial_{\gamma}=2n-s+4i$.

Take $\gamma=\ep_{2i-1}-\ep_{2i}\in T$ with $[3s/4-n/2]<i\le 3s/2-n$. One checks that \begin{align*}s(\gamma)=2\sum_{j=1}^{3s-2n-2i}\bigl(\ep_{3s-2n-j}-\ep_j\bigr)+4\sum_{j=1}^{3s/2-n-i}\bigl(\ep_{2j-1}+\ep_{2j}\bigr)+\\
2\sum_{j=3s/2-n-i+1}^{i-1}\bigl(\ep_{2j-1}+\ep_{2j}\bigr)+(\ep_{2i-1}+\ep_{2i})+2\ep_s+2\sum_{k=1}^{n-s}\bigl(\ep_{s-(2k-1)}+\ep_{s+k}\bigr)+\\
2\sum_{k=1}^{n-s}\bigl(\ep_{s-2k}-\ep_{s+k}\bigr)\in\mathbb NS\end{align*}
 so that $\gamma+s(\gamma)=2\varpi_s$. Moreover $\partial_{\gamma}=5s-2n-4i+2$.

Assume now that $3s/2\le n$. 

Take $\gamma=\ep_1-\ep_{3s/2}\in T$. One checks that 

\begin{equation*}s(\gamma)=(\ep_1+\ep_{3s/2})+2\sum_{k=1}^{s/2-1}\bigl(\ep_{s-(2k-1)}+\ep_{s+k}\bigr)+2\sum_{k=1}^{s/2-1}\bigl(\ep_{s-2k}-\ep_{s+k}\bigr)+2\ep_s\in\mathbb NS\end{equation*} so that $\gamma+s(\gamma)=2\varpi_s$.
Moroever $\partial_{\gamma}=2s$.

Set $u=0$ if $s/2$ is even and $u=1$ if $s/2$ is odd.

Take $\gamma=\ep_{s+2j+1+u}-\ep_{s+2j+2+u}\in T$ with $s/4-u/2\le j\le [(n-s-2-u)/2]$. One checks that
\begin{align*}s(\gamma)=(\ep_{s+2j+1+u}+\ep_{s+2j+2+u})+2\sum_{\ell=s/4-u/2}^j\bigl(-\ep_{s+2\ell+u}-\ep_{s+2\ell+1+u}\bigr)+\\
2\sum_{\ell=s/4-u/2}^{j-1}\bigl(\ep_{s+2\ell+1+u}+\ep_{s+2\ell+2+u}\bigr)+2(\ep_1+\ep_{3s/2})+\\
2\sum_{k=1}^{s/2-1}\bigl(\ep_{s-(2k-1)}+\ep_{s+k}\bigr)+2\sum_{k=1}^{s/2-1}\bigl(\ep_{s-2k}-\ep_{s+k}\bigr)+2\ep_s\in\mathbb NS\end{align*} so that $\gamma+s(\gamma)=2\varpi_s$.
Moreover $\partial_{\gamma}=s+4j+4+2u$.

Finally take $\gamma=-\ep_{s+2j+u}+\ep_{s+2j+1+u}\in T$ with $s/4-u/2\le j\le[(n-1-s-u)/2]$. One checks that

\begin{align*}s(\gamma)=(-\ep_{s+2j+u}-\ep_{s+2j+1+u})+2\sum_{\ell=s/4-u/2}^{j-1}\bigl(\ep_{s+2\ell+1+u}+\ep_{s+2\ell+2+u}\bigr)+\\
2\sum_{\ell=s/4-u/2}^{j-1}\bigl(-\ep_{s+2\ell+u}-\ep_{s+2\ell+1+u}\bigr)+2(\ep_1+\ep_{3s/2})+\\
2\sum_{k=1}^{s/2-1}\bigl(\ep_{s-(2k-1)}+\ep_{s+k}\bigr)+2\sum_{k=1}^{s/2-1}\bigl(\ep_{s-2k}-\ep_{s+k}\bigr)+2\ep_s\in\mathbb NS\end{align*} so that $\gamma+s(\gamma)=2\varpi_s$.
Moreover $\partial_{\gamma}=s+4j+2+2u$.

In view of Eq. \ref{cha}, we obtain that Eq. \ref{egborn} holds, which completes the proof. 
\end{proof}

We now can summarize our work in the following Theorem.

\begin{thm}\label{nonsingthm}
Let $\g$ be a simple Lie algebra of type ${\rm B}_n$ ($n\ge 2$) and $\p$ be a maximal standard parabolic Lie subalgebra of $\g$ associated with the subset $\pi'=\pi\setminus\{\alpha_s\}$, with $s$ even, of simple roots. Let $\widetilde\p$ be the In\"on\"u-Wigner contraction with respect to the decomposition $\p=\mathfrak r\oplus\m$ where $\mathfrak r$ is the standard Levi factor of $\p$ and $\m$ the nilpotent radical of $\p$. 
\begin{enumerate}
\item The algebra $Sy\bigl(\widetilde\p\bigr)$ of symmetric semi-invariants associated with $\widetilde\p$ is a polynomial algebra over the base field $\Bbbk$ and admits a Weierstrass section.

\item The derived subalgebra $\widetilde\p'$ of $\widetilde\p$ is nonsingular that is, the set $\widetilde\p'^*\setminus\widetilde\p'^*_{reg}$ of singular elements (namely non regular elements, as defined in subsection \ref{index}) is of codimension greater or equal to two.
\end{enumerate}
\end{thm}

\begin{proof}
(1) is a consequence of Lemma \ref{Adaplm}.

(2) Set $c\bigl(\widetilde\p'\bigr)=1/2\bigl(\dim\widetilde\p'+{\rm index}\,\widetilde\p'\bigr)$. Since here $\widetilde\p'=\widetilde\p_{\Lambda}$ we know that the fundamental semi-invariant $p$ of $\widetilde\p'$ (as defined for example in \cite[4.1]{JS}) is an invariant by Corollary \ref{trunccor}. Moreover by \cite[4.1]{JS} the algebra $\widetilde\p'$ is nonsingular that is, $\codim\bigl(\widetilde\p'^*\setminus\widetilde\p'^*_{reg}\bigr)\ge 2$, if and only if $p$ is a scalar. Denote by $f_1,\,\ldots,\, f_{\lvert T\rvert}$ a set of homogeneous algebraically independent generators of the polynomial algebra $Sy\bigl(\widetilde\p\bigr)=Y\bigl(\widetilde\p'\bigr)$ where recall that $\lvert T\rvert$ is the index of $\widetilde\p'$. By \cite[Thm. 5.7]{JS} we have the equality
\begin{equation} c\bigl(\widetilde\p'\bigr)-\deg p=\sum_{i=1}^{\lvert T\rvert}\deg f_i.\label{nonsing}\end{equation}

Recall that ${\rm index}\,\widetilde\p'={\rm index}\,\p'=\lvert T\rvert=n-s/2+1$ by Lemma \ref{TB}. Then we have :

\begin{align*}c\bigl(\widetilde\p'\bigr)=\frac{1}{2}\Bigl(n^2+\frac{s(s-1)}{2}+(n-s)^2+n-1+n-\frac{s}{2}+1\Bigr)\\
=n^2+n+\frac{3s^2}{4}-ns-\frac{s}{2}.\end{align*}

On the other hand, the degree of every homogeneous generator $f_1,\,\ldots,\, f_{\lvert T\rvert}$ of $Y\bigl(\widetilde\p'\bigr)$ of weight $\gamma+s(\gamma)$, for $\gamma\in T$, is equal to $\partial_{\gamma}$ computed in the proof of Lemma \ref{eqbounds}. Adding these degrees, we obtain that
$$c\bigl(\widetilde\p'\bigr)=\sum_{\gamma\in T}\partial_{\gamma}.$$
It follows by Eq. \ref{nonsing} that $\deg p=0$, hence that $\widetilde\p'$ is nonsingular by what we said above.
\end{proof}

\subsection{}

\begin{Rq}\rm
With the above hypotheses, one can check easily that an adapted pair for $\widetilde\p'$ is also an adapted pair for $\p'$. But the converse is not true in general. That is why we need here to construct new adapted pairs in the degenerate case. Indeed when we compute the coadjoint orbit of an element $y$ constructed in \cite{FL1} (in the nondegenerate case), two many zeroes appear in general, then this element $y$ could not be regular in $\widetilde\p'^*$ and then $y$ could not give an adapted pair in the degenerate case. However Thm. \ref{nonsingthm}  also holds in the nondegenerate case (for the polynomiality of $Sy(\p)$ and even the existence of a Weierstrass section, it was already shown in \cite[Sec. 7]{FL1}). We also can conclude that the derived subalgebra $\p'$ of $\p$ is nonsingular. Finally a new indexation shows that the degrees computed in \cite[Lem. 7.7]{FL1} coincide with the degrees computed here in the proof of Lemma \ref{eqbounds}, which is of course what is expected, since the degrees of the homogeneous generators of the polynomial algebra $Sy(\p)=Y(\p')$ do not depend of which adapted pair was constructed for $\p'$.

\end{Rq}

\end{document}